\documentclass[11pt,dvipsnames,reqno,hypertexnames=false]{amsart}
\usepackage{xcolor}
\usepackage[most]{tcolorbox}
\usepackage{a4wide}
\usetikzlibrary{calc,fadings}
\usepackage{yhmath,mathrsfs,amsthm,amsmath,amssymb,amsfonts,enumerate,lipsum, appendix,mathtools, bm}
\usepackage{bbold,comment}
\usepackage{pdfpages}
\usepackage{orcidlink}
\usepackage[mathscr]{euscript}
\usepackage{graphicx}
\usepackage{pgf,tikz}
\usepackage{tkz-euclide}
\usepackage{graphicx}
\usepackage{subcaption}
\usetikzlibrary{shapes.geometric}
\usetikzlibrary{arrows,hobby}
\usepackage[shortlabels]{enumitem}
\usepackage[english]{babel}

\allowdisplaybreaks

\usepackage[sort,numbers]{natbib}
\usepackage{hyperref}
\usepackage{esint}

\hypersetup{
	colorlinks=true,
	linkcolor=blue,
	filecolor=blue,      
	urlcolor=I,
	citecolor=I,
}
\definecolor{I}{rgb}{0.01, 0.75, 0.24}%{0.72, 0.25, 0.05}
\usepackage[margin=0.90in, heightrounded]{geometry}

\usepackage{bigints}
\usepackage{esint}
\newtheorem{theorem}{Theorem}[section]
\newtheorem{lemma}[theorem]{Lemma}
\newtheorem{proposition}[theorem]{Proposition}

\newtheorem{definition}[theorem]{Definition}
\theoremstyle{definition}

\newtheorem{remark}[theorem]{Remark}
\numberwithin{equation}{section}
\usepackage[hyperpageref]{backref}

\usepackage{fancyhdr}
\newcommand*\N{\mathbb{N}}
\newcommand{\al} {\alpha}

\newcommand{\pa} {\partial}
\newcommand{\be} {\beta}
\newcommand{\de} {\delta}

\newcommand{\Om} {\Omega}
\newcommand{\la} {\lambda}
\newcommand{\La} {\Lambda}
\newcommand{\si} {\sigma}

\newcommand{\no} {\nonumber}
\newcommand{\noi} {\noindent}
\newcommand{\eps} {\varepsilon}

\newcommand{\ra} {\rightarrow}

\newcommand{\bee} {\begin{equation}}
	\newcommand{\eee} {\end{equation}}
\newcommand{\bea} {\begin{eqnarray}}
	\newcommand{\eea} {\end{eqnarray}}
\newcommand{\Bea} {\begin{eqnarray*}}
	\newcommand{\Eea} {\end{eqnarray*}}

\def\d{\,{\rm d}}
\def\dx{\,{\rm d}x}
\def\dy{\,{\rm d}y}
\def\dt{\,{\rm d}t}

\def\ds{\,{\rm d}s}

\def\RN{{\mathbb R}^N}
\def\({{\Big(}}
\def\){{\Big)}}

\def\dt{\,{\rm d}t}
\def\dx{\,{\rm d}x}
\def\dz{\,{\rm d}z}

\DeclarePairedDelimiter\abs{\lvert}{\rvert}%
\DeclarePairedDelimiter\norm{\lVert}{\rVert}%

\vspace{-1.5\baselineskip}

\usepackage{xpatch}
\makeatletter   
\xpatchcmd{\@tocline}
{\hfil\hbox to\@pnumwidth{\@tocpagenum{#7}}\par}
{\ifnum#1<0\hfill\else\dotfill\fi\hbox to\@pnumwidth{\@tocpagenum{#7}}\par}
{}{}

\def\l@subsection{\@tocline{2}{0pt}{2pc}{6pc}{}}
\def\l@subsubsection{\@tocline{3}{0pt}{8pc}{8pc}{}}
\makeatother
\DeclareMathOperator*{\essinf}{ess\,inf}
\let\tmp\phi \let\phi\varphi \let\varphi\tmp

\newcommand{\fra}{(-\Delta)^s}
\newcommand{\del}{-\Delta}

\newcommand{\R}{\mathbb{R}}
\newcommand{\RR}{\mathbb{R}^N}

\newcommand{\ov}{\overline}

\newcommand{\NN}{\mathbb{N}}

\newcommand{\A}{\mathbb{A}}
\renewcommand{\AA}{\mathcal{A}}
\newcommand{\BB}{\mathcal{B}}
\newcommand{\CC}{\mathcal{C}}
\newcommand{\DD}{\mathcal{D}}

\newcommand{\GG}{\mathcal{G}}

\newcommand{\HH}{\mathcal{H}}
\newcommand{\KK}{\mathcal{K}}
\newcommand{\LL}{\mathcal{L}}

\newcommand{\PP}{\mathcal{P}}
\renewcommand{\SS}{\mathcal{S}}
\newcommand{\XX}{\mathcal{X}}

\newcommand\restr[2]{{% we make the whole thing an ordinary symbol
  \left.\kern-\nulldelimiterspace % automatically resize the bar with \right
  #1 % the function
  \littletaller % pretend it is a little taller at normal size
  \right|_{#2} % this is the delimiter
  }}
\usepackage[sort,numbers]{natbib}
\newcommand{\littletaller}{\mathchoice{\vphantom{\big|}}{}{}{}}

\title[]{Brezis-Nirenberg problems for mixed local-nonlocal operators with superlinear perturbations: compactness and applications}
\author[M. Bhakta, N. Biswas, and P. Das]{Mousomi Bhakta\,$^\dagger$,
Nirjan Biswas\,\orcidlink{0000-0002-3528-8388}, \and Paramananda Das\,\orcidlink{0009-0009-2821-6675}}
\address{Department of Mathematics, Indian Institute of Science Education and Research (IISER-Pune), Dr. Homi Bhabha Road, Pune-411008, India}
\email[M. Bhakta]{mousomi@iiserpune.ac.in} \email[N. Biswas]{nirjan.biswas@acads.iiserpune.ac.in, nirjaniitm@gmail.com} 
\email[P. Das]{paramananda.das@students.iiserpune.ac.in, pd348225@gmail.com}
\thanks{$^\dagger$Corresponding author}	
\begin{document}
\begin{abstract}
In this paper, we consider the following mixed local nonlocal Brezis-Nirenberg problem 
\begin{equation}\label{crit_pro_abstract}\tag{$\mathcal{P}_{2^*}$}
-\Delta u+(-\Delta)^s u=\lambda |u|^{p-2}u+|u|^{2^*-2}u\text{ in  }\Omega,\quad
u=0\text{ in }\mathbb{R}^N \setminus \Omega,
\end{equation}
where $\Omega\subset\mathbb{R}^N$ is a smooth bounded domain, $N\geq3$, $s\in(0,1)$, $\lambda>0$, and $2\leq p<2^*=\frac{2N}{N-2}$.
We establish a compactness result for the following class of subcritical/critical problems
    \begin{equation}\label{sub_pro_abstract}\tag{$\mathcal{P}_{p_n}$}
    -\Delta u+(-\Delta)^s u=\lambda |u|^{p-2}u+|u|^{p_n-2}u\text{ in  }\Omega,\quad
    u=0\text{ in }\mathbb{R}^N \setminus \Omega,
\end{equation}
where $p_n \in (p,2^* ]$ and $p_n\to 2^*$. Specifically, for $p \in (2+\frac{4s}{N-2},2^*)$ when $N>6-4s$, and for $p \in (2^*-1,2^*)$ when $N\leq6-4s$, we prove that any bounded sequence of solutions $\{u_n\}$ to \eqref{sub_pro_abstract} is relatively compact in the energy space, and converges strongly to a nontrivial solution to \eqref{crit_pro_abstract}. To the best of our knowledge, this is the first paper to address this type of compactness result for a non-homogeneous operator. Due to the presence of the non-homogeneous operator, proving the compactness result requires several delicate new and novel estimates, which we believe will be of independent interest for further studies of related problems. As an application of this compactness result, under the same ranges of $N$ and $p$, we prove that \eqref{crit_pro_abstract} admits infinitely many sign-changing solutions.
\end{abstract}

\maketitle
\tableofcontents
\section{Introduction}
This paper deals with the following superlinear Brezis-Nirenberg problem driven by the mixed local-nonlocal operator: 
\begin{equation}\tag{$\PP_{2^*}$}\label{main_PDE}
\begin{cases}
            \del u+\fra u=\lambda |u|^{p-2}u+|u|^{2^\ast-2}u \; &\text{in }\Om,\\
            u=0&\text{in }\RR \setminus \Om,
        \end{cases}
\end{equation}
where $\Om$ is a smooth bounded domain  in  $\RR$ with $N \ge 3$, $s\in(0,1)$, $\la> 0$ is a parameter, $2\leq p<2^*:= \frac{2N}{N-2}$, the critical Sobolev exponent,  and fractional Laplacian is defined for smooth enough functions as $$\fra u(x)=c_{N,s}\,\text{P.V.}\int_{\RR}\frac{u(x)-u(y)}{|x-y|^{N+2s}}\dx\dy,$$
where $c_{N,s}$ is a normalization constant.

\subsection{Main results} 
We first consider the space
\begin{equation}
  X_0(\Omega):=\left\{u\in H^1(\RR):u|_{\Om}\in H_0^1(\Om),\, u=0\text{ a.e. in }\RR\setminus\Om\right\}.
\end{equation}
It is easy to see that $X_0(\Omega)$ is a Hilbert space with the norm 
$$\rho(u):=\left(\|\nabla u\|_2^2+[u]_s^2\right)^{\frac{1}{2}},$$ which is associated with the inner product $$\langle u,v\rangle=\int_{\Om}\nabla u\cdot\nabla v\dx+\iint_{\R^{2N}}\frac{(u(x)-u(y))(v(x)-v(y))}{|x-y|^{N+2s}}\dx\dy.$$
Here $[\cdot]_s$, called the Gagliardo seminorm, is defined as $$[u]_s^2=\iint_{\R^{2N}}\frac{|u(x)-u(y)|^2}{|x-y|^{N+2s}}\dx\dy.$$
Note that that the norm $\rho$ is equivalent to the gradient norm $\norm{\nabla\cdot}_{L^2(\Om)}$ and as a consequence, $X_0(\Omega)$ is continuously embedded into $L^r(\Om)$ for every $r\in[1,2^*]$ and the embedding is compact for $r<2^*$. 
A function $u\in X_0(\Omega)$ is a said to be a weak solution to $$\del u+\fra u=f(u)\text{ in }\Om, \quad u=0\text{ in }\R^N \setminus \Omega,$$ if for every $v\in X_0(\Omega)$ it holds
\begin{equation}\label{weak}
        \int_{\Om}\nabla u\cdot\nabla v\dx+\iint_{\RR}\frac{(u(x)-u(y))(v(x)-v(y))}{|x-y|^{N+2s}}\dx\dy=\int_{\Om} f(u)v\dx.
\end{equation}
We now present our main theorem, which asserts that every bounded sequence of weak solutions associated with a family of  superlinear perturbation of subcritical/critical problems converges strongly in $X_0(\Omega)$ to a weak solution of the critical problem \eqref{main_PDE}. This is known as `\textit{compactness}' in the literature.
\begin{theorem}[Compactness]\label{Uniform_Theorem}
    Let $\la>0$, $N \ge 3$, $s \in (0,1)$ and 
\begin{equation}\label{p_cond}\tag{$A_p$}
        p\in\begin{cases}
            (2+\frac{4s}{N-2},2^*),& \text{if }N>6-4s;\\
            (2^*-1,2^*),& \text{if }N\leq 6-4s.
        \end{cases}
    \end{equation}
In addition, we assume that $s>\frac{1}{2}$ if $N=3$. Let $\{u_n\}$ be a bounded  sequence in $X_0(\Omega)$ such that for each $n \in \N$, $u_n$ weakly solves \begin{equation}\label{sub_pro}%\tag{$\PP_{p_n}$}
\begin{cases}
    \del u+\fra u=\lambda |u|^{p-2}u+|u|^{p_n-2}u&\text{in }\Om,\\
    u=0&\text{in }\RR \setminus \Om,
\end{cases}
\end{equation}
where $p_n\in(p,2^* ]$ and $p_n\to 2^*$. Then, up to a subsequence, $u_n \ra u$ in $X_0(\Omega)$ and $u$ weakly solves \eqref{main_PDE}.
\end{theorem}
Regarding the compactness result, we make a few remarks. Set \begin{equation}
    p_0:=\begin{cases}
    2+\frac{4s}{N-2},& \text{if }N>6-4s;\\
    2^*-1, &\text{if }N\leq6-4s.
\end{cases}
\end{equation}
\begin{remark}
   One may expect the compactness phenomenon to persist for all $p>2$. On the other hand, in the limiting case $p=2$, compactness cannot, in general, be expected. To see this, first notice that if $p=2$ and $\lambda<\lambda_1$, where $\lambda_1$ is the first eigenvalue of $(-\Delta+(-\Delta)^s,\Omega)$ and $\{u_n\}$ is a bounded sequence of nontrivial solutions to \eqref{sub_pro}, then if the compactness Theorem~\ref{Uniform_Theorem} holds that would imply $u_n\to u_0$ for some $u_0\in X_0$. Further, 
$$
    \rho(u_n)^2
    = \lambda \|u_n\|_2^2 + \|u_n\|_{p_n}^{p_n}
    \leq \frac{\lambda}{\lambda_1}\rho(u_n)^2
    + C^{p_n}\rho(u_n)^{p_n}.
$$
In other words,
$$
\rho(u_n)
\geq
\left(\frac{\lambda_1-\lambda}{\lambda_1}\right)^{\frac{1}{p_n-2}}
C^{\frac{-p_n}{p_n-2}}.$$
Since $p_n\to 2^*$, it follows that $\rho(u_n)\geq C$ for some $C$ independent of $n$ and hence $\rho(u_0)\geq C$. Now following the same method as in the proof of Theorem~\ref{Existence_Theorem} would imply $u_0$ is a nontrivial solution to  \eqref{main_PDE} with $p=2$ for every $\lambda<\lambda_1$.  However, the Pohozaev identity \cite[Corollary~1.5]{An} implies  \eqref{main_PDE} does not have any nontrivial solutions in any star-shaped domain when $p=2$ and $\lambda<(1-s)\lambda_{1,s}<\lambda_1$, where $\lambda_{1,s}$ is the first eigenvalue of $((-\Delta)^s,\Omega)$. Hence for $p=2$ and $\la>0$ small Theorem~\ref{Uniform_Theorem} fails when $\Om$ is star-shaped. In this direction, it may also be natural to raise the following question:
\begin{center}
\noi\textbf{Question: }\textit{Does the compactness result remain valid in general smooth domain} for the linear perturbation problem (i.e.,~$p=2$) when $\lambda>0$ is sufficiently large, and more generally for all $p>2$ and every $\lambda$?
\end{center}
\end{remark}

\begin{remark}\label{nontrivial_remark}
    If $\{u_n\}$ is a sequence of nontrivial solutions to \eqref{sub_pro}, uniformly bounded in $X_0(\Om)$, then the limit $u$, in Theorem~\ref{Uniform_Theorem} can not be trivial. To see this, we observe that, by the Sobolev inequality,
    \begin{align*}
        &\rho(u_n)^2=\la\|u_n\|_p^p+\|u_n\|_{p_n}^{p_n}\leq\la C\rho(u_n)^p+C\rho(u_n)^{p_n}\Longrightarrow1\leq \la C \rho(u_n)^{p-2}+C\rho(u_n)^{p_n-2}.
    \end{align*}
    If the limit $u\equiv0$ then $\rho(u_n)\to0$. Since $2<p<p_n$, we obtain a contradiction.
\end{remark}
\begin{remark}
    The assumptions on $N$ and $p$ arise naturally from the blow-up analysis. Whereas the assumption $s>\frac12$ in the case $N=3$ is technical in nature; it is specifically required to ensure that the local Talenti bubbles belong to the space $\dot{H}^s(\RR)=\ov{\CC_c^\infty(\RR)}^{[\cdot]_s}$.
\end{remark}
Due to the presence of the critical nonlinearity, the energy functional associated to \eqref{main_PDE}  fails to satisfy the Palais-Smale condition for every $c\in\R$. Consequently, the minimax theorem cannot be applied directly to obtain infinitely many solutions. As an application of Theorem \ref{Uniform_Theorem}, we also obtain the existence of infinitely many sign-changing solutions  as stated in the following theorem:
\begin{theorem}\label{Existence_Theorem}
Let  $N \ge 3$, $\la>0$ and $s \in (0,1)$. We assume that $s>\frac{1}{2}$ if $N=3$. Suppose \eqref{p_cond} holds. Then \eqref{main_PDE} admits infinitely many sign-changing solutions.
\end{theorem}
\begin{remark} (a) A closer examination of \eqref{main_PDE} and \eqref{sub_pro} naturally leads to the following question: what happens if one perturbs the lower order term instead of the critical term? Let $$S_p:=\text{set of all weak solutions to \eqref{main_PDE} in $X_0(\Om)$.}$$ By Theorem \ref{Existence_Theorem}, when \eqref{p_cond} holds, $S_p$ has infinitely many elements for each $p$. Set 
\begin{align*}
    S:= \underset{p\in(p_0,2^*)}{\bigcup} S_p.
\end{align*}
Clearly, if $\{v_n\}\subseteq S_p$ for some fixed $p\in(p_0,2^*)$ is a uniformly bounded sequence in $X_0(\Om)$, then using Theorem \ref{Uniform_Theorem}, $\{v_n\}$ has a strongly convergent subsequence in $X_0(\Om)$. 

(b) Now consider the case where $v_n\in S_{p_n}$ for $p_n\in (p_0,2^*)$ and $\{v_n\}$ is a uniformly bounded sequence in $X_0(\Om)$. That is, for each $n \in \N$, $v_n$ weakly solves
\begin{equation}\label{lim_sup_pro}
\begin{cases}
    \del u+\fra u=\lambda |u|^{p_n-2}u+|u|^{2^*-2}u&\text{in }\Om,\\
    u=0&\text{in }\RR \setminus \Om.
\end{cases}
\end{equation}
 Adapting the arguments used in the proof of Theorem \ref{Uniform_Theorem}, we observe that (up to a subsequence) if $p_n\to \hat{p} \in[p_0, 2^*)$, then $\{v_n\}$ is relatively compact in $X_0(\Om)$.
\end{remark}
\subsection{Literature review}
Consider the purely local problem 
\begin{equation}\label{loc_BrNi2}
       \del u=\la |u|^{p-2}u+|u|^{2^*-2}u \; \text{ in }\Om,\quad u=0\;\text{ on }\pa\Om.
\end{equation}
The case $p=2$ is corresponds to the classical Brezis-Nirenberg problem, while $p>2$ gives a superlinear perturbation.

In the case $p=2$, Brezis and Nirenberg, in their seminal work \cite{BrNi}, showed that
\begin{enumerate}
    \item \eqref{loc_BrNi2} has no nontrivial solution on any star-shaped domains for $\la\leq0$;
    \item There exists $\la^*\in[0,\mu_1)$ such that \eqref{loc_BrNi2} has a positive solution for all $\la\in (\la^*,\mu_1)$, where $\mu_1$ is the first Dirichlet eigenvalue of $(\del,\Om)$. If $N\geq4$, $\la^*=0$ and if $N=3$ and $\Om$ is a ball then $\la^*=\frac{\mu_1}{4}$;
    \item \eqref{loc_BrNi2} does not admit any positive solution for $\la\geq\mu_1$.
\end{enumerate}
Since then, various multiplicity results have been obtained by many authors. One of the important contributions in this direction was made by Devillanova and Solimini. In \cite{DeSo}, they proved the following: let $N\geq7$, $\la>0$ and $U$ be a bounded set in $H_0^1(\Om)$ whose elements are weak solutions to
\begin{equation}\label{loc_BrNi1}
        \del u=\la u+|u|^{q-2}u \; \text{ in }\Om,\quad u=0\;\text{ on }\pa\Om,
\end{equation}
for $q$ varying in $[2,2^*]$. Then $$\sup_{u\in U}\sup_{x\in\Om}|u(x)|\leq C.$$
This in turn implies a compactness result similar to Theorem \ref{Uniform_Theorem}. As an application of the compactness result, using standard minimization arguments, they proved the existence of infinitely many nontrivial solutions to \eqref{loc_BrNi2}
for any $\la>0$ provided $N\geq7$. This is an important contribution because the existence of infinitely many nontrivial solutions to \eqref{loc_BrNi2} was known only when $\Omega$ is a symmetric domain. For instance, in \cite{FoJa}, Fortunato and Jannelli proved the existence of infinitely many non-trivial solutions of \eqref{loc_BrNi2} for $N\geq4$, $\la>0$ and $\Om$ satisfying some symmetry conditions. However, the result of Devillanova and Solimini does not require any symmetry assumption on $\Omega$.

Observe that, the restriction $N\geq7$ in the compactness theorem of Devillanova and Solimini can not be removed for $\la>0$ small. To see this, note that when $\Om$ is a ball, minimizing over the space of radial functions, compactness theorem implies the existence of infinitely many radial solutions to \eqref{loc_BrNi2} for any $\la>0$. For $N\geq3$ and $\la>0$, Srikanth \cite{Sri} proved that \eqref{loc_BrNi2} admits at most one positive solution when $\Om$ is a ball. Hence, \eqref{loc_BrNi2} has infinitely many radial sign-changing solutions. However, for $4\leq N \leq 6$ and $\Om$ is a ball,  Atkinson, Brezis and Peletier in \cite{AtBrePe} showed that \eqref{loc_BrNi2} does not admit any radial sign-changing solution if $\la>0$ is small enough. Hence, $N\geq 7$ is crucial in the theorem of Devillanova and Solimini.

Schechter and Zou in \cite{ScZo} proved an abstract critical point theorem and using it along with the compactness theorem of Devillanova and Solimini, they established the existence of infinitely many sign-changing
solutions of \eqref{loc_BrNi2} for all $\la>0$ when $N\geq7$.

Next, for the superlinear perturbation case (i.e. $p>2$), Brezis and Nirenberg in \cite{BrNi} also showed that 
\begin{enumerate}
    \item [(4)]If $N\geq4$ or $N=3$ and $p\in(4,6)$, \eqref{loc_BrNi2} has a positive solution for every $\la>0$.
    \item[(5)] If $N=3$ and $p\in(2,4]$ , \eqref{loc_BrNi2} has a positive solution for $\la$ large enough.
\end{enumerate}
Servadei and Valdinoci in \cite{SeVa, SeVa1} studied the following fractional analog of \eqref{loc_BrNi2}, namely
\begin{equation}\label{nonloc_BrNi1}
    \fra u=\la |u|^{p-2}u+|u|^{2_s^*-2}u\;\text{ in } \Om,\quad u=0\; \text{ in }\RR\setminus\Om,
\end{equation}
where $2_s^*=\frac{2N}{N-2s}$ is the nonlocal critical Sobolev exponent.\\
In the case $p=2$, they proved that
\begin{enumerate}
    \item for $N\geq4s$ and for any $\la\in(0,\la_{1,s})$, \eqref{nonloc_BrNi1} admits a positive solution;
    \item for $2s< N<4s$ there exists $\la_s>0$ such that for any $\la>\la_s$, \eqref{nonloc_BrNi1} has a nontrivial solution.
\end{enumerate}

When $p\in(2,2_s^*)$,  Barrios, Colorado, Servadei and Soria in \cite{BaCoSeSo} showed that
\begin{enumerate}
    \item [(3)] for $N>\frac{2s(p+2)}{p}$ and $\la>0$ or $N<\frac{2s(p+2)}{p}$ and $\la>0$ sufficiently large, \eqref{nonloc_BrNi1} admits a positive solution.
\end{enumerate}
Following the arguments of \cite{DeSo} and using the Caf{}farelli-Silvestre extension method, Yan, Yang and Yu \cite{YaYaYu} demonstrated a compactness result, comparable to Theorem \ref{Uniform_Theorem}, for the \textit{spectral Laplace} counterpart of \eqref{nonloc_BrNi1} when $N>6s$ . Building on the compactness result of \cite{YaYaYu}, Li, Su and Tersian \cite{LiSuTe} leveraged those findings to prove the existence of infinitely many sign-changing solutions. However, to our knowledge, it remains an open question whether this compactness result holds for the case $N \leq 6s$.

Following the arguments of Devillanova and Solimini \cite{DeSo}, similar compactness results have been established for other elliptic operators, such as the $p$-Laplace operator \cite{CaPeYa}, the Kohn Laplace operator \cite{ChFaLi}, and other quasilinear operators (see, e.g., \cite{GaGu}).  To the best of our knowledge, this is the first paper to address this type of compactness result for a non-homogeneous operator. 
After submitting our work, we became aware of a parallel work by Carletti and Premoselli \cite{CaPr}, in which the authors also prove a compactness result for critical equations involving mixed higher-order elliptic operators, using a different approach.

Finally, we present some known results for mixed local-nonlocal Brezis-Nirenberg problems. Biagi, Dipierro, Valdinoci and Vecchi \cite{BiDiVaVe} recently proved that\\
when $p=2$,
\begin{enumerate}
    \item for $\la\leq0$, \eqref{main_PDE} has no nontrivial solution;
    \item there exists $\la^*\in[\la_{1,s},\la_1)$ such that for any $\la\in(\la^*,\la_1)$, \eqref{main_PDE} has a positive solution; 
    \item for $0<\la\leq \la_{1,s}$, \eqref{main_PDE} has no positive solution belonging to a certain ball in $L^{2^*}(\RR)$;
    \item for $\la\geq\la_1$, \eqref{main_PDE} does not admit any positive solution;
\end{enumerate}
and when $p \in (2,2^*)$, set 
\begin{equation}\label{Biagi_cond}
    \kappa_{s,N}:=\min\{2-2s,N-2\},\quad\beta_{p,N}:=N-\frac{p(N-2)}{2},
\end{equation}
\begin{enumerate}
    \item [(5)] if $\kappa_{s,N}>\beta_{p,N}$, then \eqref{main_PDE} admits a positive solution for every $\la>0$;
    \item [(6)] if $\kappa_{s,N}\leq\beta_{p,N}$, then \eqref{main_PDE} admits a positive solution for $\la$ large enough.
\end{enumerate}
Regarding the multiplicity of nontrivial solutions to the mixed local-nonlocal problems, da Silva, Fiscella and Viloria in \cite{DaFiVi} showed that 
\begin{enumerate}[(1)]
    \item for $p=2$, \eqref{main_PDE} admits $m$ pairs of nontrivial solutions when $\la>\la_{k+1}-\SS_N|\Om|^{-\frac2N}$ and $\la_{k}\leq \la<\la_{k+1}=\cdots=\la_{k+m}<\la_{k+m+1}$, where $\SS_N$ is the best constant in the Sobolev inequality \eqref{critical};
    \item for $p>2$, there exists $\la_{**}>0$ such that for any $\la\in(0,\la_{**})$, \eqref{main_PDE} admits at least $\text{cat}_{\Om}(\Om)$ many nontrivial solutions. They employed Ljusternik-Schnirelmann category theory to establish this result. 
\end{enumerate}

\begin{remark}
    Observe that when $N>6-4s$, the condition $p>2+\frac{4s}{N-2}$, in the hypothesis of our Theorem \ref{Existence_Theorem}, is equivalent to $\kappa_{s,N}>\beta_{p,N}$.
\end{remark}

\subsection{Difficulties and Strategies}\label{difficulties} To prove Theorem~\ref{Uniform_Theorem}, we follow the strategy by Devillanova and Solimini in \cite{DeSo}.  We brief{}ly outline the main idea. Observe that $\{u_n\}$ is a Palais-Smale sequence of $I_\la$. Thus, using the Palais-Smale decomposition of the energy functional associated with \eqref{main_PDE}, (see Proposition \ref{PS}, also see \cite{ChGuMaSr} for $p=2$) $u_n$ can be decomposed as the sum of the weak limit of $u_n$ and finitely many local Talenti bubbles. To establish strong convergence, it remains to prove that no such bubbles arise in the decomposition.

For the classical Brezis-Nirenberg problem \eqref{loc_BrNi2} with $p=2$, Devillanova and Solimini considered a \textit{safe region} suitably away from any concentration points of the Talenti bubbles. In that \textit{safe region}, they showed boundedness and some gradient estimates of $u_n$. Using the {\it local} Pohozaev's identity on a ball centred at a concentration point, with the boundary lying inside the safe region, they obtain control over the $L^2$ mass of $u_n$ in the whole ball via boundary integrals inside the safe region. Exploiting the estimates in the safe region yields an upper bound, while the Palais-Smale decomposition together with the asymptotic behaviour of the bubbles provides a corresponding lower bound. A comparison of these bounds ultimately shows that no bubbles can occur when $N\geq7$.

Next, we discuss some key steps and dif{}ficulties for proving Theorem \ref{Uniform_Theorem}:
\begin{enumerate}
    \item To get estimates on $u_n$, Devillanova and Solimini used the following identity 
    \begin{equation}\label{form_1}
        \frac{\d}{\dt}\left(\frac{1}{t^{N-1}}\int_{\pa B_t^N(x_0)}u(y)\d S\right)+\left(\frac1{t^{N-1}}\int_{B_t^N(x_0)}(\del u(y))\dy\right)=0,
    \end{equation}
where $B_t^N(x_0)\subset\Om$. Similarly, for \eqref{nonloc_BrNi1}, Yan et al. employed a related identity based on Caf{}farelli-Silvestre's extension of $u$. Inspired by \cite[Proposition C.1]{CaPeYa}, 
we obtain that $u_n$ satisfies the following expression through the application of the weak Harnack inequality \cite{GaKi} and the Wolf{}f potential estimates \cite{BySo}:
\begin{equation}\label{form_2}
    \left(\frac1{r^N}\int_{B_r^N(x_0)}|u|^\gamma\dx\right)^{\frac1{\gamma}}\leq C+C\int_r^R\left(\frac1{\rho^{N-1}}\int_{B_\rho^N(x_0)}|(\del u+\fra u)|\dx\right)\d\rho,
\end{equation}
for any $\gamma\in[1,\frac{N}{N-2})$.
\item After establishing $L^\gamma$ estimate for some $\gamma$, \cite{CaPeYa,YaYaYu} used the Moser iteration to get the local boundedness. However, due to the presence of a nonlocal term, Moser iteration for the mixed local-nonlocal problem introduces a tail term (see \cite{GaKi}). We observe that, due to the presence of the tail term, the method of Devillanova and Solimini can not be applied directly.

\item To the best of our knowledge, the {\it local} Pohozaev's identity is not known for the fractional Laplacian operator and also for the mixed operator. 

\item In view of the above two dif{}ficulties, following \cite{YaYaYu}, we consider Caf{}farelli-Silvestre's (CS) extension for \eqref{main_PDE}. Although a Moser iteration scheme can still be implemented in the CS extension set-up, the nonhomogeneous nature of the operator prevents a direct application of the techniques used in \cite{CaPeYa, YaYaYu} to obtain the desired $L^\gamma$ estimates. We also mention that the use of CS extension has recently been applied in the context of mixed local-nonlocal operators, see \cite{Go}.
\item \label{loc_poh_diff} In Theorem \ref{Uniform_Theorem}, we prove a {\it local} Pohozaev's identity for the CS extension of \eqref{main_PDE}. In this case, the {\it local} Pohozaev's identity yields (see Subsection \ref{subsec_proof_uniform}):
$$\la\int_{B^N_{\si_n^{-\frac12}}(x_n)}|u_n|^p\dx\leq \text{Boundary integrals}+C\int_{B^{N+1}_{\si_n^{-\frac12}}(x_n,0)\cap\{y>0\}}y^{1-2s}|\nabla U_n|^2\dx\dy,$$
where $U_n$ is the Caf{}farelli-Silvestre extension of $u_n$, and $\si_n, \,(x_n,0)$ are respectively the blow-up rate and the concentration point of the slowest concentrating bubble in the (PS) decomposition of $\{U_n\}$ (see the Subsection \ref{safe_reg_subsec}). Using the estimates obtained on the safe region, we have $\text{Boundary integrals}\leq C\si_n^{1-\frac N2}$. Also, using the bubble behaviour, we get
\begin{align*}
    & \la\int_{B^N_{\si_n^{-\frac12}}(x_n)}|u_n|^p\dx\geq C\la\si_n^{\frac{(N-2)p}{2}-N}, \quad\text{ and } \\
    &\int_{B^{N+1}_{\si_n^{-\frac12}}(x_n,0)\cap\{y>0\}}y^{1-2s}|\nabla U_n|^2\dx\dy \sim \si_n^{2s-2}.
\end{align*}
This implies
$$\si_n^{\frac{N-2}{2}p-N}\leq C\si_n^{2s-2}+C\si_n^{1-\frac{N}{2}}\leq\begin{cases}
    C\si_n^{2s-2},&\text{ if }N>6-4s;\\
    C\si_n^{1-\frac{N}{2}},&\text{ if }N\leq6-4s.
\end{cases}$$
We observe that the above inequality, with the fact that $\{\sigma_n\}$ is unbounded, implies 
\begin{equation*}
\begin{cases}
    \frac{N-2}{2}p-N\leq 2s-2\Longrightarrow p\leq 2+\frac{4s}{N-2},&\text{ if }N>6-4s;\\
    \frac{N-2}{2}p-N\leq1-\frac{N}{2}\implies p\leq 2^*-1,&\text{ if }N\leq6-4s.
\end{cases}
\end{equation*}
Therefore, in the ranges mentioned in the hypothesis of Theorem \ref{Uniform_Theorem}, no bubble can appear in the decomposition. We emphasize that this is the only step where the restriction on $p$ arises.
\end{enumerate}
\begin{remark}\label{diffculties1}
    As discussed in \eqref{loc_poh_diff} above, the condition on $p$ appears because of the method we follow. In this method, the following quantity  $$\int_{B^{N+1}_{\si_n^{-\frac12}}(x_n,0)\cap\{y>0\}}y^{1-2s}|\nabla U_n|^2\dx\dy$$ in the {\it local} Pohozaev's identity is hindering us to get the compactness result till $p=2$. This term emerges from the non-homogeneous nature of the mixed operator rather than from the nonlinearity. This suggests that one may {\it possibly} encounter a similar situation in other mixed local-nonlocal problems when employing this method. 
\end{remark}

As mentioned before, a key step in proving the existence of infinitely many solutions to \eqref{main_PDE} is the analysis of the \textit{compactness} established in Theorem \ref{Uniform_Theorem}. The same strategy has been implemented to obtain the existence of infinitely many solutions for different elliptic operators (see e.g. \cite{CaPeYa, DeSo, YaYaYu}). To prove the existence of infinitely many sign changing solutions, we apply the methods developed by Schechter and Zou \cite{ScZo}.

\subsection{Outline and Notation} The paper is organized as follows: Section \ref{prelim} describes the Caf{}farelli-Silvestre's (CS) extension for the mixed local nonlocal operator, and we establish an upper estimate on the normal derivative of the CS extension of the solution associated with the mixed local nonlocal operator (see Lemma~\ref{convergence}). Section \ref{compact} establishes a (PS) decomposition for the CS extension of \eqref{main_PDE} and provides several key estimates over the safe regions to rule out bubbles in the decomposition. This section concludes with the proof of the compactness theorem. In Section \ref{infinite}, we utilize these results to prove the existence of infinitely many sign-changing solutions to \eqref{main_PDE}. Appendix \ref{A} contains some technical lemmas. 

We use the following notation and convention: 
\begin{enumerate}[(i)]
    \item We denote $\d\mu := \frac{\dx\dy}{|x-y|^{N+2s}}.$
    \item We denote $$\A(u,v) := \iint_{\R^{2N}}\frac{(u(x)-u(y))(v(x)-v(y))}{|x-y|^{N+2s}}\dx\dy.$$
    \item $B_r^N$ is a ball centered at $x_0\in\RR$ and radius $r$ in $\RR$. Every point in $\R^{N+1}$ is written as $(x,y)$, where $x\in\RR$ and $y\in\R$. We further denote $$\R_+^{N+1}:=\R^{N+1}\cap\{y>0\},\, B_r^+(x_0,0):=B_r^{N+1}(x_0,0)\cap\{y>0\}.$$
    \item For any open set $U\subset\RR$, we denote $$X_0(U):=\left\{u\in H^1(\RR):u|_{U}\in H_0^1(U),\, u=0\text{ a.e. in }\RR\setminus U\right\}.$$
    \item We define the space $\DD^{1,2}(\R_+^{N+1},y^{1-2s})$ as 
\begin{align*}
    \DD^{1,2}(\R_+^{N+1},y^{1-2s}) := \overline{\CC_c^{\infty}(\R_+^{N+1})}^{\|\cdot\|_{\DD^{1,2}(\R_+^{N+1},y^{1-2s})}}, 
\end{align*}
where $$\|\cdot\|_{\DD^{1,2}(\R_+^{N+1},y^{1-2s})}:=\left(\int_{\R_+^{N+1}}y^{1-2s}\abs{\nabla \cdot}^2\dx\dy\right)^{\frac12}.$$
\item For $U\in \DD^{1,2}(\R_+^{N+1},y^{1-2s})$, we denote the trace of $U$ on the boundary $\R^N\times\{y=0\}$ by $\text{Tr}(U)$.
     \item We denote $$\XX^s(\R_+^{N+1}):=\left\{U\in \DD^{1,2}(\R_+^{N+1},y^{1-2s}):\text{Tr}(U)\in \DD^{1,2}(\RR)\right\}.$$
    \item We also denote $$\XX_{\Om}^s(\R_+^{N+1}):=\left\{U\in \DD^{1,2}(\R_+^{N+1},y^{1-2s}):\text{Tr}(U)\in X_0(\Om)\right\}.$$
    \item The Tail term is denoted as $$\text{Tail}(u;x_0,R):=R^2\int_{\RR\setminus B_R^N(x_0)}\frac{|u(x)|}{|x-x_0|^{N+2s}}\dx.$$
    \item The excess functional is defined as $$E(u;x_0,R):=\fint_{B_R^N(x_0)}|u-(u)_{B_R^N(x_0)}|\dx+\text{Tail}(u-(u)_{B_R^N(x_0)};x_0,R).$$ 
    \item $\norm{\cdot}_p := \norm{\cdot}_{L^p(\RR)}$  for $p \in (0, \infty)$.
    \item $\d\HH^k$ denotes the $k$-dimensional Hausdorf{}f measure.
    \item $C$ denotes a generic positive constant. 
\end{enumerate}

\section{Preliminary}\label{prelim}

We begin with Caf{}farelli-Silvestre's extension \cite{CaSi}.
\subsection{Caf{}farelli-Silvestre's extension}  Let $u\in \dot{H}^s(\RR)$ and $E_s(u):=U\in \DD^{1,2}(\R_+^{N+1},y^{1-2s})$ be the $s$-harmonic extension of $u$ to the upper half space $\R_+^{N+1}$, i.e. $U$ weakly solves
\begin{equation*}
    \begin{cases}
        \text{div}(y^{1-2s}\nabla U)=0\text{ in }\R_+^{N+1},\\
        U(x,0)=u(x)\text{ in }\RR.
    \end{cases}
\end{equation*}
Then we have the following properties:
\begin{enumerate}[(a)]
    \item $U$ can be expressed explicitly by means of the Poisson kernel. For $x\in\RR,y>0$ ,
    \begin{equation}\label{Poisson_formula}
        U(x,y)=\int_{\RR}P(x-\xi,y)u(\xi)\d\xi,
        \end{equation}
    where the Poisson kernel $P(x,y)$ is defined as
    \begin{equation}\label{Poisson_kernel}
        P(x,y)=C_{N,s}\frac{y^{2s}}{(|x|^2+y^2)^{\frac{N+2s}{2}}}, \quad x\in\RR,\, y>0.
    \end{equation}
    \item The extension is related to the fractional Laplacian in the following way
    \begin{equation}\label{DtoN}
        -C_{N,s}\lim_{y\to0^+}y^{1-2s}\frac{\pa U}{\pa y}(x,y)=\fra u(x),\quad x\in\RR.
    \end{equation}
    \item Due to the relation
    \begin{equation}\label{isometry}
        [u]_s^2=C_{N,s}\int_{\R_+^{N+1}}y^{1-2s}|\nabla U|^2\dx\dy,
    \end{equation}
    the extension map $E_s:\dot{H}^s(\RR)\to \DD^{1,2}(\R_+^{N+1},y^{1-2s})$ sending $u$ to its $s$-extension $U$ is an isometry.
    \item Trace inequality:
    \begin{equation}\label{trace_ineq}
        \|u\|_{L^{2^*}(\RR)}\leq C\|U\|_{\DD^{1,2}(\R_+^{N+1},y^{1-2s})}.
    \end{equation}
    \item Let $U\in \DD^{1,2}(\R_+^{N+1},y^{1-2s})$ and $B_r^+(x,0)\subset\R_+^{N+1}$. Then (see \cite[Lemma 2.1]{TaXi}) there exist $C(r),\,\de=\de(N,s)>0$ such that for any $1\leq t\leq\frac{N+1}{N}+\de$,
    \begin{equation}\label{Tan-Xiong}
        \|U\|_{L^{2t}(B_r^+(x,0),y^{1-2s})}\leq C(r)\|\nabla U\|_{L^2(B_r^+(x,0),y^{1-2s})}.
    \end{equation}
\end{enumerate}
We define the Hilbert space $$\XX_{\Om}^s(\R_+^{N+1}):=\left\{U\in \DD^{1,2}(\R_+^{N+1},y^{1-2s}):\text{Tr}(U)\in X_0(\Om)\right\},$$ equipped with the norm $$\tilde{\rho}(U):=\left(\int_{\Om}|\nabla_x U(x,0)|^2\dx+C_{N,s} \int_{\R_+^{N+1}}y^{1-2s}|\nabla U|^2\dx\dy\right)^{\frac{1}{2}},\quad U\in \XX_{\Om}^s(\R_+^{N+1}).$$
We also define the Hilbert space $\XX^s(\R_+^{N+1})$ as 
$$\XX^s(\R_+^{N+1}):=\left\{U\in \DD^{1,2}(\R_+^{N+1},y^{1-2s}):\text{Tr}(U)\in \DD^{1,2}(\RR)\right\},$$ equipped with the norm $$\|U\|_{\XX^s(\R_+^{N+1})}^2:=\|U\|_{\DD^{1,2}(\R_+^{N+1},y^{1-2s})}^2+\|\text{Tr}(U)\|_{\DD^{1,2}(\RR)}^2.$$ 
Using the isometry of the extension map, we have $$\rho(\text{Tr}(U))=\tilde{\rho}(U),\quad U\in \XX_{\Om}^s(\R_+^{N+1}).$$
Let $u_n\in X_0(\Om)$ weakly solve \eqref{sub_pro}. Then its Caf{}farelli-Silvestre's extension $U_n\in \XX_{\Om}^s(\R_+^{N+1})$ satisfies
\begin{equation}\label{ext_sub_pro}
    \begin{cases}
        \text{div}(y^{1-2s}\nabla U_n)=0\text{ in }\R_+^{N+1},\\
        U_n(x,0)=u_n(x)\text{ in }\RR,\\
        \del u_n(x)- C_{N,s} \lim_{y\to0^+}y^{1-2s}\frac{\pa U_n}{\pa y}(x,y)=\la |u_n|^{p-2}u_n+|u_n|^{p_n-2}u_n.
    \end{cases}
\end{equation}
The weak formulation of \eqref{ext_sub_pro} is 
\begin{align}
    \int_{\Om}\nabla u_n(x)\cdot\nabla_x\phi(x,0)\dx & + C_{N,s} \int_{\R_+^{N+1}}y^{1-2s}\nabla U_n(x,y)\cdot\nabla \phi(x,y)\dx\dy \no \\
    &=\int_{\Om}(\la |u_n|^{p-2}u_n+|u_n|^{p_n-2}u_n)\phi(x,0)\dx,\; \forall \, \phi\in \XX_{\Om}^s(\R_+^{N+1}).
\end{align}
We consider the energy functional 
\begin{equation}\label{ext_func}
\tilde{I}_\la(U)=\frac12\tilde{\rho}(U)^2-\frac\la p\int_{\Om}|U(x,0)|^p\dx-\frac1{2^*}\int_{\Om}|U(x,0)|^{2^*}\dx,\quad U\in \XX_{\Om}^s(\R_+^{N+1}).
\end{equation}
Critical points of $\tilde{I}_\la$ weakly solve the extended problem associated with \eqref{main_PDE}, i.e.
\begin{equation}\label{ext_main_pro}
    \begin{cases}
        \text{div}(y^{1-2s}\nabla U)=0\text{ in }\R_+^{N+1},\\
        U(x,0)=u(x)\text{ in }\RR,\\
        \del u(x)- C_{N,s} \lim_{y\to0^+}y^{1-2s}\frac{\pa U}{\pa y}(x,y)=\la |u|^{p-2}u+|u|^{2^*-2}u\text{ in }\Om.
    \end{cases}
\end{equation}
\begin{remark}
    For the sake of brevity, henceforth we will omit the normalization constant $C_{N,s}$.
\end{remark}
In view of the above discussion, the statement of Theorem \ref{Uniform_Theorem} can be formulated as follows.

\begin{theorem}\label{Uniform_Theorem_CK}
Let  $N \ge 3$, $\la>0$ and $s \in (0,1)$. For $N=3$, we assume that $s>\frac{1}{2}$. Suppose \eqref{p_cond} holds. Let $\{U_n\}$ be a bounded set in $\XX_{\Om}^s(\R_+^{N+1})$ such that for each $n \in \N$, $U_n$ weakly solves the  problem \eqref{ext_sub_pro}, where $p<p_n\leq 2^*,\, p_n\to2^*$ as $n \ra \infty$. Then, up to a subsequence, $U_n \ra U$ in $\XX_{\Om}^s(\R_+^{N+1})$ and $U$ weakly solves the critical problem \eqref{ext_main_pro}.
\end{theorem}

\subsection{Normal derivative of Caffarelli-Silvestre's extension}
The following lemma plays a crucial role in the Pohozaev's identity. 

\begin{lemma}\label{convergence}
Let $u \in \mathcal{C}^2(\Om)\cap \mathcal{C}^{1,\al}(\ov{\Om})\cap  X_0(\Omega)$ weakly solve $$\del u+\fra u=g\text{ in }\Om,\quad u=0\text{ in }\RR \setminus \Omega,$$where $\|g\|_{\mathcal{C}^{0,1}(\Om)}\leq1$. Let $U \in \XX_{\Om}^s(\R_+^{N+1})$ be the $s$-harmonic extension of $u$ to the upper half space $\R_+^{N+1}$. Then for every $(x,y)\in\R_+^{N+1}$,
\begin{align}
    &y^{1-2s}|\pa_y U(x,y)| \no \\
    &\leq C\left(1+\chi_{\{s<\frac12\}}+\chi_{\{s=\frac12\}} \left(|\log(\text{dist}(x,\pa\Om))|+\text{dist}(x,\pa\Om)^{-\frac12}\right)+\chi_{\{s>\frac12\}}\text{dist}(x,\pa\Om)^{1-2s}\right).
\end{align} 
\end{lemma}
\begin{proof}
% From \cite{SuVaWeZh1}, $u\in \mathcal{C}^2(\Om)$ and $u\in \mathcal{C}^{1,\al}(\ov{\Om})$ for every $\al\in(0,\min\{1,2-2s\})$. Using $u=0$ in $\RR \setminus \Omega$, we get $u\in \mathcal{C}^{0,1}(\RR)$. Further, $u \in L^{\infty}(\RR)$, and using $\text{supp}(\nabla u) \subset \Omega$, we also have $\norm{\nabla u}_{L^{\infty}(\RR)} \le C$ for some $C>0$.
% Recall that, 
% \begin{align*}
%     U(x,y):= \int_{\RR} P(x-\xi, y) u(\xi)  \d \xi, 
% \end{align*}
% where $$P(x,y):=C_{N,s}\frac{y^{2s}}{(|x|^2+y^2)^{\frac{N+2s}{2}}}, \quad x\in\RR,\, y>0.$$
% Then it is easy to verify that $U \in \mathcal{C}^{\infty}(\R^{N+1}_+)$. Further, from the change of variable $z=x-\xi$,  
% \begin{align*}
%     U(x,y) = \int_{\RR} P(z, y) u(x-z) \dz, 
% \end{align*}
% and hence, using Leibniz's integral rule for every $y>0$, we get 
% \begin{align*}
%     \abs{\nabla_xU(\cdot,y)} \le \int_{\RR} |P(z, y) \nabla u(x-z)| \dz \le \norm{\nabla u}_{L^{\infty}(\RR)} \int_{\RR} P(z, y) \dz = \norm{\nabla u}_{L^{\infty}(\RR)} \le C,
% \end{align*}
% which implies 
% \begin{align}\label{bound-1}
%    \norm{\nabla_xU(\cdot,y)}_{L^{\infty}(\RR)} \le C,  
% \end{align}
% where $C>0$ is independent of $y$. 
% Moreover, the pointwise convergence holds 
% \begin{align}\label{converge-1}
% \nabla_xU(x,\eps)\to\nabla u(x) \text{ for every }x\in B^N\setminus\pa\Om, \text{ as }\eps\to0.  \end{align}
% Further, using \eqref{DtoN}, for every $x\in\RR$,
% \begin{align}\label{converge-2}
%  -\eps^{1-2s}\frac{\pa U}{\pa y}(x,\eps)\to(-\Delta)^s u(x), \text{ as }\eps\to0.   
% \end{align} 
From \cite[Eq. (2.2)]{An}, we recall the following bound: there exists a constant $C>0$ such that,
\begin{equation}\label{frac_Laplace_bound}
    |(-\Delta)^s u(x)|\leq C(1+\chi_{\{s<\frac12\}}+\chi_{\{s=\frac12\}}|\log(\text{dist}(x,\pa\Om))|+\chi_{\{s>\frac12\}}\text{dist}(x,\pa\Om)^{1-2s}), \; \forall \, x\in\Omega.
\end{equation}
In view of \cite{SuVaWeZh1,anup-mitesh2023}, we have $u\in \mathcal{C}^{1,\al}(\ov{\Om})$ and $u\in \mathcal{C}^{0,1}(\RR)$. Now following the arguments of \cite[Theorem 2.2]{An}, it can readily be checked that \eqref{frac_Laplace_bound} holds for every $x\in\RR \setminus \overline{\Omega}$. 
% For $s < \frac{1}{2}$, $(-\Delta)^s u \in L^{\infty}(B^N)$ and hence in $L^1(B^N)$. 
% Notice that, there exists $\tilde{r}>0$ such that $\text{dist}(x,\pa\Om) \le \tilde{r}$ for every $x \in B^N$. Moreover, for $\hat{r} \in (0,1)$,
% \begin{align*}
%    &\text{ when } s=\frac{1}{2}: \int_{0}^{\hat{r}} |\log(t)| \dt = -\int_{0}^{\hat{r}} \log(t) \dt = -(t\log(t)-t)\bigg|_{t=0}^{t=\hat{r}}< \infty, \\
%    &\text{ when } s>\frac{1}{2}: \int_{0}^{\hat{r}} t^{1-2s} \dt = \frac{t^{2-2s}}{2-2s}\bigg|_{t=0}^{t=\hat{r}}< \infty. 
% \end{align*}
% Therefore, for $s\in (0,1)$, $(-\Delta)^s u \in L^{1}(B_n^N)$. This integrability will confirm that 
% \begin{align*}
%     \int_{B^N}\left|((-\Delta)^s u)((x-x_0)\cdot\nabla u)\phi(x,0)\right| \dx < \infty.
% \end{align*} 
% We consider the function 
% \begin{align*}
%     G(x, \eps):= ((x-x_0)\cdot\nabla_xU(x,\eps))\phi(x,\eps)\left(-\eps^{1-2s}\frac{\pa U}{\pa y}(x,\eps)\right), \text{ where } x \in B^N.
% \end{align*}
% Using \eqref{converge-1} and \eqref{converge-2}, we have $G(x, \eps) \ra ((-\Delta)^s u)((x-x_0)\cdot\nabla u)\phi(x,0)$ for every $x \in B^N \setminus \pa \Omega$, as $\eps \ra 0$. Further, using \eqref{bound-1}, 
% \begin{align}\label{bound-3}
%     |G(x, \eps)| \le C \phi(x,0) \left|-\eps^{1-2s}\frac{\pa U}{\pa y}(x,\eps)\right|,
% \end{align}
% for some $C>0$ independent of $\eps$.
Set $V(x,y):=-y^{1-2s}\pa_yU(x,y)$. Using $\text{div}(y^{1-2s}\nabla U)=0$ in $\R_+^{N+1}$, we get $\pa_yV=y^{1-2s}\Delta_xU$. We calculate 
\begin{align*}
    \text{div}(y^{2s-1}\nabla V)&=y^{2s-1}\Delta_xV+\pa_y(y^{2s-1}\pa_yV)\\&=y^{2s-1}(-y^{1-2s}\pa_y\Delta_xU)+\pa_y\Delta_xU=0\text{ in }\R_+^{N+1}.
\end{align*}
Therefore, $V$ satisfies the conjugate extension problem:
\begin{equation*}
\begin{cases}
    \text{div}(y^{2s-1}\nabla V)=0\text{ in }\R_+^{N+1},\\
    V(x,0)=-\lim_{y\to0^+}y^{1-2s}\pa_yU(x,y)=(-\Delta)^su(x) \text{ in }\RR.
\end{cases}
\end{equation*}
By Poisson kernel representation, we get 
\begin{align}\label{representation-I}
    V(x,y)=\int_{\RR}Q(x-\xi,y)(-\Delta)^s u(\xi)\d\xi,
\end{align}
where the conjugate Poisson kernel $$Q(x,y)=C_{N,s}\frac{y^{2-2s}}{(|x|^2+y^2)^{\frac{N+2-2s}{2}}},\quad x\in\RR,y>0.$$ 
Further,  
\begin{align}\label{normalize}
    \int_{\RR} Q(x-\xi,y) \d\xi =1. 
\end{align}
If $s< \frac{1}{2}$, then \eqref{frac_Laplace_bound} implies $|(-\Delta)^s u(x)| \le C$ for all $x\in \RR\setminus\pa\Om$. Combining this with \eqref{representation-I}, and \eqref{normalize}, for every $x\in \RR$ and $y>0$, $|V(x,y)| \le C$ where $C$ does not depend on $y$. For brevity, we denote $d(x) = \text{dist}(x, \pa \Omega)$. We split 
\begin{align}\label{V_n-1}
    |V(x,y)| \le \left( \int_{\{|z| \le \frac{d(x)}{2}\}} + \int_{\{|z| > \frac{d(x)}{2}\}} \right) Q(z,y) (-\Delta)^s u(x-z)\dz := I_{1} + I_{2},  
\end{align}
Now we consider $s> \frac{1}{2}$ and $s=\frac{1}{2}$ separately. 

\noi \textbf{For $s> \frac{1}{2}$:} In the case $|z|\leq\frac{d(x)}{2}$,  a simple computation yields $ d(x-z) \ge d(x) - |z| \ge \frac{d(x)}{2}$. Therefore, 
\begin{align*}
    I_1 \le C \left(1+d(x)^{1-2s}\right) \int_{\RR} Q(z,y) \dz = C \left(1+d(x)^{1-2s}\right).
\end{align*} 
To estimate $I_2$, we decompose $\{ |z| \ge \frac{d(x)}{2} \}= \cup_{k=0}^{\infty} A_k$, where $A_k:=\{ 2^{k-1}d(x) \le |z| \le 2^k d(x)\}$ for $k \ge 0$. 
Set $R_k=2^kd(x)$. Hence

\begin{align*}
    \int_{A_k} Q(z,y) d(x-z)^{1-2s} \dz \le  C_{N,s}\frac{y^{2-2s}}{(R_{k-1}^2+y^2)^{\frac{N+2-2s}{2}}} \int_{B_{R_k}(0)} d(x-z)^{1-2s} \dz. 
\end{align*}
Now using the change of variable $\xi=x-z$ and the  layer cake representation, 
\begin{align*}
   \int_{B_{R_k}(0)} d(x-z)^{1-2s} \dz &= \int_{B_{R_k}(x)} d(\xi)^{1-2s} \d\xi \\
   &= \int_{0}^{\infty} |\{\xi \in B_{R_k}(x) : d(\xi)^{1-2s} >t \}| \dt \\
    &= \int_{0}^{\infty} |\{\xi \in B_{R_k}(x) : d(\xi) < t^{\frac{1}{1-2s}} \}| \dt \\
    &\le (2s-1) \int_{0}^{\infty} |\{\xi \in B_{R_k}(x) : d(\xi) < \tau \}| \tau^{-2s} \d \tau.
\end{align*}
Observe that $d(\xi)\le d(x) +|z| \le 3R_k$ for every $\xi \in B_{R_k}(x)$. 
Further, if $\tau \le 3R_k$, from the co-area formula, 
\begin{align*}
    |\{\xi \in B_{R_k}(x) : d(\xi) \le \tau \}| = \int_{0}^{\tau} \left( \int_{\{\xi \in B_{R_k}(x) : d(\xi) = s \}} \d\HH^{N-1} \right) \ds \le C\tau R_k^{N-1}. 
\end{align*}
If $\tau > 3R_k$, then 
\begin{align*}
    |\{\xi \in B_{R_k}(x) : d(\xi) \le \tau \}| = |B_{R_k}(x)| = \omega_{N-1} R_k^N. 
\end{align*}
Therefore, 
\begin{align*}
    \int_{B_{R_k}(x)} d(\xi)^{1-2s} \d\xi \le C(N,s) \left(R_k^{N-1} \int_{0}^{3R_k} \tau^{1-2s} \d \tau + R_k^N \int_{3R_k}^{\infty} \tau^{-2s} \d \tau\right) \le C(N,s)R_k^{N+1-2s}. 
\end{align*}
As $R_k=2R_{k-1}$, the above estimates lead to
\begin{align*}
   \int_{A_k} Q(z,y) d(x-z)^{1-2s} \dz & \le   C(N,s)\frac{y^{2-2s}}{(R_{k-1}^2+y^2)^{\frac{N+2-2s}{2}}} R_{k-1}^{N+1-2s} \\
    &=   C(N,s) R_{k-1}^{1-2s} \frac{\left( \frac{y}{R_{k-1}}\right)^{2-2s}}{\left(1+\left(\frac{y}{R_{k-1}}\right)^2 \right)^{\frac{N+2-2s}{2}}}. 
\end{align*}
Next, we observe that 
\begin{align*}
    h(t) = \frac{t^{2-2s}}{(1+t^2)^{\frac{N+2-2s}{2}}} \ra 0 \text{ as } t \ra 0 \text{ and } t \ra \infty,
\end{align*}
which infers that $|h(t)| \le M$ for $t \in \R^+$. Therefore, for every $k \ge 0$, 
\begin{align*}
    \int_{A_k} Q(z,y) d(x-z)^{1-2s} \dz \le C(N,s) (2^kd(x))^{1-2s}. 
\end{align*}
The above estimate gives 
\begin{align*}
    I_2 \le C + C \int_{\{|z| > \frac{d(x)}{2}\}} Q(z,y) (-\Delta)^s u(x-z)\dz & \le C(N,s) \left(1 +  d(x)^{1-2s} \sum_{k=0}^{\infty} (2^k)^{1-2s} \right)\\
    & \le C(N,s) \left(1 +  d(x)^{1-2s} \right). 
\end{align*}
Therefore, from \eqref{V_n-1}, we infer that for $s>\frac{1}{2}$, $|V(x,y)| \le C(N,s) \left(1 +  d(x)^{1-2s} \right)$ where $(x,y) \in \R_+^{N+1}$. 

\noi \textbf{For $s= \frac{1}{2}$:} In the case $|z|\leq\frac{d(x)}{2}$, using the following bounds
\begin{align*}
  \frac{1}{2} d(x) \le  d(x) - |z| \le d(x-z) \le d(x) +|z| \le \frac{3}{2} d(x),
\end{align*}
we get $|\log\big(d(x-z)\big)| \le C(1+|\log(d(x))|)$, which implies 
\begin{align*}
    I_1 \le C(1+|\log(d(x))|)\int_{\RR} Q(z,y) \dz = C(1+|\log(d(x))|).
\end{align*} 
Next, we estimate $I_2$. For that purpose, we observe that $(-\Delta)^{\frac{1}{2}} u_n \in L^p(\RR)$ for every $p\in (1, \infty)$. Indeed, for any $R>0$, 
\begin{align*}
 \int_{B_R(0)} \left|(-\Delta)^{\frac{1}{2}} u_n(x) \right|^p \dx \le C \int_{B_R(0)} (1+|\log(d(x))|)^p \dx < \infty, \text{ since }   \int_{0}^{\delta} |\log(t)|^p \dt < \infty, 
\end{align*}
for some $\delta>0$. Now we choose $R>0$ such that $\Omega \subset B_{\frac{R}{2}}(0)$. Then, for $|x|>R$, using $u(x)=0$  we get 
\begin{align*}
    \left|(-\Delta)^{\frac{1}{2}} u(x)\right| = \left|-C_{N,s} \text{P.V.} \int_{\Omega} \frac{u(y)}{|x-y|^{N+1}} \dy\right| \le C(N,s) \norm{u}_{L^{\infty}(\Omega)} \frac{1}{|x|^{N+1}}, 
\end{align*}
where we have used the fact that for $y\in\Om$ it holds  $|x-y| \ge |x|-|y| \ge \frac{|x|}{2}$. 
Therefore,
\begin{align*}
    \int_{\RR \setminus B_R(0)} \left|(-\Delta)^{\frac{1}{2}} u(x) \right|^p \dx \le C(N) \|u\|^p_{L^{\infty}(\Omega)} \int_{R}^{\infty} t^{N-1-(N+1)p} \dt < \infty.  
\end{align*}
Further, we see that for every $y>0$, 
\begin{align*}
    Q(z,y) = C_{N,\frac{1}{2}} \frac{y}{(|z|^2+y^2)^{\frac{N+1}{2}}} \le \frac{C}{|z|^{N}},
\end{align*}
where the final bound comes using the fact the $y= |z|N^{-\frac{1}{2}}$ maximizes $y(|z|^2+y^2)^{-\frac{N+1}{2}}$. 
Therefore, applying H\"{o}lder's inequality with the conjugate pair $(\frac{2N}{2N-1}, 2N)$, we obtain
\begin{align*}
    I_2 \le  C(N) \norm{(-\Delta)^{\frac{1}{2}} u}_{L^{2N}(\RR)} \left(\int_{\{|z| > \frac{d(x)}{2}\}} \frac{\dz}{|z|^{\frac{2N^2}{2N-1}}}\right)^{\frac{2N-1}{2N}},
\end{align*}
where 
\begin{align*}
    \int_{\{|z| > \frac{d(x)}{2}\}} \frac{\dz}{|z|^{\frac{2N^2}{2N-1}}} = C(N) \int_{\frac{d(x)}{2}}^{\infty} \frac{t^{N-1}}{t^{\frac{2N^2}{2N-1}}} \dt \le c(N) d(x)^{N-\frac{2N^2}{2N-1}}.
\end{align*}
Hence
\begin{align*}
   I_2 \le  C(N) \norm{(-\Delta)^{\frac{1}{2}} u}_{L^{2N}(\RR)} d(x)^{-\frac{1}{2}}. 
\end{align*}
Accumulating all the estimates, for $(x,y) \in \R_+^{N+1}$, we arrive at 
\begin{align*}
    |V(x,y)| \le C\left(1+|\log(d(x))|+d(x)^{-\frac{1}{2}}\right), 
\end{align*}
where $C>0$ is independent of $y$.  

Thus, we have the following estimates for $V$ over $\R_+^{N+1}$: 
\begin{equation}\label{final-1}
    |V(x,y)| \le \left\{\begin{aligned}
&C,\,&\text{if}\;s<\frac{1}{2}; \\
&C\left(1+|\log(d(x))|+d(x)^{-\frac{1}{2}}\right),\,&\text{if}\;s=\frac{1}{2}; \\
&C\left(1 +  d(x)^{1-2s} \right),\,&\text{if}\;s>\frac{1}{2},
\end{aligned}
\right.
\end{equation}
where $C$ does not depend on $y$. This completes the proof. 
\end{proof}

\section{Compactness}\label{compact}

\subsection{Palais-Smale decomposition} In this section, we aim to prove a Palais-Smale decomposition for the energy functional $\tilde{I}_\la$.
\begin{definition}
    A sequence $\{ U_n \} \subset \XX_{\Om}^s(\R_+^{N+1})$ is said to be a Palais-Smale (PS) sequence for $\tilde{I}_\la$ at level $\eta$, if $\tilde{I}_\la(U_n) \ra \eta$ and $\tilde{I}_\la'(U_n) \ra 0$ in $\XX_{\Om}^s(\R_+^{N+1})^*$ as $n \ra \infty$. The functional $\tilde{I}_\la$ is said to satisfy (PS) condition at level $\eta$, if every (PS) sequence at level $\eta$ has a convergent subsequence. 
\end{definition}
We claim that, up to a subsequence, $\{U_n\}$ in Theorem \ref{Uniform_Theorem_CK} is a (PS) sequence for $\tilde{I}_\la$. Using the uniform boundedness of $\{U_n\}\subset \XX_{\Om}^s(\R_+^{N+1})$ and the trace inequality, we see that $\{\tilde{I}_\la(U_n)\}_n$ is bounded. Observe that for any $V\in \XX_{\Om}^s(\R_+^{N+1})$, $$\langle \tilde{I}_\la'(U_n),V\rangle=\int_{\Om}(|u_n(x)|^{p_n-2}u_n(x)-|u_n(x)|^{2^*-2}u_n(x))V(x,0)\dx,$$
and $$\||u_n(x)|^{p_n-2}u_n(x)-|u_n(x)|^{2^*-2}u_n(x)\|_{L^{\frac{2^*}{2^*-1}}(\Om)}\leq C\||u_n|^{2^*-1}+1\|_{L^{\frac{2^*}{2^*-1}}(\Om)}\leq C.$$
Using $\text{Tr}(\XX_{\Om}^s(\R_+^{N+1}))\underset{\text{continuous}}{\hookrightarrow}  H_0^1(\Om)$, $u_n\rightharpoonup u$ in $H_0^1(\Om)$ and $u_n\to u$ a.e. on $\Om$. Thus, $|u_n(x)|^{p_n-2}u_n(x)-|u_n(x)|^{2^*-2}u_n(x)\to0$ a.e. on $\Om$ and up to a subsequence $|u_n(x)|^{p_n-2}u_n(x)-|u_n(x)|^{2^*-2}u_n(x)\rightharpoonup0$ in $L^{\frac{2^*}{2^*-1}}(\Om)$. Hence, $\langle \tilde{I}_\la'(U_n),V\rangle\to0$ as $n\to\infty$ and $\{U_n\}$ is a (PS) sequence of $\tilde{I}_\la$.

It is worth mentioning that every (PS) sequence of $\tilde{I}_\la$ may not converge strongly due to the noncompactness of the trace embedding $\XX_{\Om}^s(\R_+^{N+1}) \hookrightarrow L^{2^*}(\Omega)$. Moreover, the weak limit of the (PS) sequence can be zero even if $\eta>0$. This necessitates studying the profile decomposition of the (PS) sequences for the energy functional \eqref{ext_func} associated with \eqref{ext_main_pro}. We define a $C^1$ functional $\tilde{I}$ on $\DD^{1,2}(\RR)$ as $$\tilde{I}(u):=\frac12\int_{\RR}|\nabla u|^2\dx-\frac1{2^*}\int_{\RR}|u|^{2^*}\dx, \; u \in \DD^{1,2}(\RR).$$

In the following proposition, we state the profile decomposition of the (PS) sequences for \eqref{ext_func} (see also \cite{ChGuMaSr}). We need the following remark.

\begin{remark}\label{Abun-Talenti}
    Recall the following critical Sobolev inequality: 
    \begin{align}\label{critical}
        \mathcal{S}_{N}\|u\|_{L^{2^*}(\mathbb{R}^N)}^2 \leq \|\nabla u\|_{L^2(\mathbb{R}^N)}^2, \; \forall \, u \in \mathcal{D}^{1,2}(\RN), 
    \end{align}
    where $\mathcal{S}_{N}$ is the optimal constant. From \cite{Aubin1976, Talenti1976}, it is known that every minimizer of \eqref{critical} lies in the set $\mathcal{F}$, where  
    \begin{equation}\label{bubble}
\mathcal{F} := \left\{c\lambda^{-\frac{N-2}{2}} \hat{u}\left( \frac{x - x_0}{\lambda} \right) : \,c\in\R\setminus \{0\},\,\lambda > 0,\, x_0 \in \mathbb{R}^N\right\},
\end{equation}
and
\begin{equation*}
  \hat{u}(x) := (1 + |x|^2)^{-\frac{N-2}{2}}.  
\end{equation*}
Further, for $s \in (0,1)$ we observe that the seminorm $[\hat{u}]_{s}$ is finite for $N \ge 4$ and for $N = 3$ with the restriction that $s> \frac{1}{2}$. 
\end{remark}

\begin{proposition}[Palais-Smale decomposition]\label{PS}
  Let $N\geq 3$ and $s\in(0,1)$. For $N=3$ we assume that $s>\frac{1}{2}$.  Let $\la>0$ and $\{U_n\}\subset \XX_{\Om}^s(\R_+^{N+1})$ be a (PS)-sequence of $\tilde{I}_\la$, i.e. $\tilde{I}_\la(U_n) \ra \eta $ in $\R$ and $\nabla\tilde{I}_\la(U_n)\to0$ in $(\XX_{\Om}^s(\R_+^{N+1}))^*$. Then there exist $k\in\N$ and sequences $\{\si_n^i\}_n\subset\R_+$ and $\{x_n^i\}_n\subset\Om$ for $1\leq i\leq k$, a solution $U_\infty\in \XX_{\Om}^s(\R_+^{N+1})$ of \eqref{ext_main_pro}, and nontrivial solutions $\Psi^i\in \DD^{1,2}(\RR)$, $1\leq i\leq k$ to the limiting problem associated with \eqref{ext_main_pro}, namely
  $$\del \Psi^i=|\Psi^i|^{2^*-2}\Psi^i\text{ in }\RR,$$
  such that, up to a subsequence, 
   \begin{align*}
    &U_n=U_{\infty}+\sum_{i=1}^kE_s(\Psi_n^i)+ o_n(1) \, \text{ in }\,  \XX^s(\R_+^{N+1}),\\
   & \text{Tr}(U_n)=\text{Tr}(U_{\infty})+\sum_{i=1}^k\Psi_n^i+ o_n(1) \, \text{ in } \, \DD^{1,2}(\RR).
      \end{align*}
Here $\Psi_n^i$ denotes rescaled function
$$\Psi_n^i:={(\si_n^i)}^{\frac{N}{2^*}}\Psi^i(\si_n^i(x-x_n^i)), \quad 1\leq i\leq k, \, n\in\mathbb{N}.$$
Moreover, it holds
 \begin{align*}
        &\eta = \tilde{I}_\la(U_\infty)+\sum_{i=1}^k \tilde{I}(\Psi^i) + o_n(1),\\
        &\si_n^i \ra \infty, \text{ and } \left|\log\left(\frac{\si_n^i}{\si_n^j}\right)\right|+\si_n^i|x_n^i-x_n^j|\to\infty, \text{ for } 1\leq i\neq j\leq k.
    \end{align*}
    Furthermore, in the case $k=0$, the above expression holds without $\Psi_n^i$.
\end{proposition}
\begin{proof}
We divide our proof into a few steps. 
\vspace{0.2 cm} \\
\textbf{Step-1:} We first show that the sequence $\{U_n\}$ is bounded. Observe that by the trace inequality and the H\"{o}lder's inequality, when $p>2$,
\begin{align*}
    \eta+o_n(1)\tilde{\rho}(U_n)\geq  \tilde{I}_\la(U_n)-\frac1p\langle \tilde{I}_\la'(U_n),U_n\rangle
    &{\geq}\left(\frac12-\frac{1}{p}\right)\tilde{\rho}(U_n)^2.
\end{align*}
When $p=2$,
\begin{align*}
    C+o_n(1)\tilde{\rho}(U_n)\geq \left| 2\tilde{I}_\la(U_n)-\langle \tilde{I}_\la'(U_n),U_n\rangle\right|
    &=\left(1-\frac{2}{2^*}\right)\int_{\Om}|U_n(x,0)|^{2^*}\dx\\
    &\geq  C\left(\int_{\Om}|U_n(x,0)|^2\right)^{\frac{2^*}{2}}.
\end{align*}
Again using the trace inequality and the H\"{o}lder's inequality,
\begin{align*}
    \tilde{\rho}(U_n)^2&=2\tilde{I}_\la(U_n)+\la\int_{\Om}|U_n(x,0)|^2\dx+\frac2{2^*}\int_{\Om}|U_n(x,0)|^{2^*}\dx\\
    &\leq C+(C+o_n(1)\tilde{\rho}(U_n))^{\frac{2}{2^*}}+(C+o_n(1)\tilde{\rho}(U_n)).
\end{align*}
Hence $\{U_n\}$ is bounded in $\XX_{\Om}^s(\R_+^{N+1})$. By the ref{}lexivity of $\XX_{\Om}^s(\R_+^{N+1})$, $U_n\rightharpoonup U_{\infty}$ in $\XX_{\Om}^s(\R_+^{N+1})$ and $U_n\to U_{\infty}$ a.e. in $\R_+^{N+1}$. By the definition of the norm $\tilde{\rho}$, $\{\text{Tr}(U_n)\}$ is bounded in $X_0(\Om)$. By the ref{}lexivity of $X_0(\Om)$, $\text{Tr}(U_n) \rightharpoonup v$ in $X_0(\Om)$.
By the compact embedding $X_0(\Om) \hookrightarrow \dot{H}^s(\RR)$ (see \cite[Lemma 2.2]{ChGuMaSr}), up to a subsequence $\{\text{Tr}(U_n)\}$ converging to  $v$ w.r.t. the norm of $\dot{H}^s(\RR)$.
Now, using the norm equivalence \eqref{isometry}, 
$$ \int_{\R_+^{N+1}}y^{1-2s}|\nabla (U_n-E_s(v))|^2\dx\dy=[\text{Tr}(U_n)-v]_s^2\to0.$$
Thus $U_{\infty}=E_s(v)$ a.e. in $\R_+^{N+1}$ and $$ \int_{\R_+^{N+1}}y^{1-2s}|\nabla (U_n-U_\infty)|^2\dx\dy\to0\text{ as }n\to\infty.$$
Note that, though $U_n\to U_\infty$ in $\DD^{1,2}(\R^{N+1}_+, y^{1-2s})$, $\text{Tr}(U_n)$ may not converge to $v$ in $H_0^1(\Omega)$. Hence, $U_n$ need not converge to $U_\infty$ in $\XX_{\Om}^s(\R_+^{N+1})$.

\vspace{2mm}

\noi \textbf{Step-2:} Let $V_n:=U_n-U_{\infty}$. Then $V_n\rightharpoonup0$ in $\XX_{\Om}^s(\R_+^{N+1})$ and $$ \int_{\R_+^{N+1}}y^{1-2s}|\nabla V_n|^2\dx\dy=o_n(1).$$ 
For $p<2^*$, using $\text{Tr}(\XX_{\Om}^s(\R_+^{N+1}))\underset{\text{continuous}}{\hookrightarrow}  H_0^1(\Om)\underset{\text{compact}}{\hookrightarrow} L^p(\Om)$, we have $\text{Tr}(V_n)\rightharpoonup0$ in $H_0^1(\Om)$ and $\text{Tr}(V_n)\to 0$ in $L^p(\Om)$. Applying Brezis-Lieb lemma, observe that
\begin{align*}
    \tilde{I}(\text{Tr}(V_n))&=\frac12\int_{\Om}|\nabla V_n(x,0)|^2\dx-\frac1{2^*}\int_{\Om}|V_n(x,0)|^{2^*}\dx\\
    &=\frac12\tilde{\rho}(V_n)^2-\frac{\la}p\int_{\Om}|V_n(x,0)|^p\dx-\frac1{2^*}\int_{\Om}|V_n(x,0)|^{2^*}\dx+o_n(1)\\
    & =\tilde{I}_\la(V_n)+o_n(1)=\tilde{I}_\la(U_n)-\tilde{I}_\la(U_\infty)+o_n(1).
\end{align*}
Furthermore, using $\text{Tr}(V_n)\rightharpoonup0$ in $H_0^1(\Om)$, 
\begin{align*}
    \langle \tilde{I}'(\text{Tr}(V_n)),\phi\rangle&=\int_{\Om}\nabla V_n(x,0)\cdot\nabla\phi(x)\dx-\int_{\Om}|V_n(x,0)|^{2^*-2}V_n(x,0)\phi(x)\dx=o_n(1),
\end{align*}
for any $\phi\in H^1_0(\Om)$.

Thus, $\{\text{Tr}(V_n)\}$ is a (PS) sequence of $\tilde{I}$ at the level $\eta-\tilde{I}_\la(U_\infty)$. \vspace{0.2 cm} \\
\noi \textbf{Step-3:} If $\text{Tr}(V_n)\to0$ in $X_0(\Om)$, then the conclusion holds for $k=0$. Now suppose $\text{Tr}(V_n)\not\to0$ in $X_0(\Om)$. Applying \cite[Lemma III.3.3]{St}, we obtain sequences $\{x_n^1\}\subset\Om$, $\{\si_n^1\} \subset (0, \infty)$ with $\si_n^1\to\infty$, a nontrivial solution $\Psi^1$ of 
$$\del U=|U|^{2^*-2}U\text{ in }\RR,$$ and a (PS) sequence $\{w_n\}$ for $\tilde{I}$ in $H_0^1(\Om)$ 
such that
$$w_n=\text{Tr}(V_n)-\Psi_n^1+o_n(1)\text{ in }\DD^{1,2}(\RR),$$
where $\Psi_n^1=(\si_n^1)^{\frac{N-2}{2}} \Psi^1(\si_n^1(\cdot-x_n^1))$. Moreover,
\begin{align*}
    \tilde{I}(w_n)&=\tilde{I}(\text{Tr}(V_n))-\tilde{I}(\Psi^1)+o_n(1)=\tilde{I}_\la(U_n)-\tilde{I}_\la(U_\infty)-\tilde{I}(\Psi^1)+o_n(1).
\end{align*}
Following the proof of \cite[Theorem III.3.1]{St} and iterating the above process, we get the decomposition for $\text{Tr}(U_n)$:
\begin{align*}
        &\text{Tr}(U_n)=\text{Tr}(U_{\infty})+\sum_{i=1}^k\Psi_n^i  + o_n(1) \text{ in } \DD^{1,2}(\RR), \\
        & \tilde{I}_\la(U_n) = \tilde{I}_\la(U_\infty)+\sum_{i=1}^k \tilde{I}(\Psi^i) + o_n(1),\\
        &\si_n^i \ra \infty, \text{ and } \left|\log\left(\frac{\si_n^i}{\si_n^j}\right)\right|+\si_n^i|x_n^i-x_n^j|\to\infty, \text{ for } 1\leq i\neq j\leq k.
    \end{align*}

\noi \textbf{Step-4:} From Step-3 we have $$\text{Tr}(U_n)=\text{Tr}(U_{\infty})+\sum_{i=1}^k\Psi_n^i  + o_n(1) \text{ in } \DD^{1,2}(\RR),$$ and from Step-1, we have $U_n=U_\infty+o_n(1) \, \mbox{ in }\, \DD^{1,2}(\R^{N+1}_+, y^{1-2s})$. From Remark \ref{Abun-Talenti} and the hypothesis on $N,s$, we see that $[\Psi^i]_s < \infty$, and further using the isometry we get
\begin{align*}
&\sum_{i=1}^k\int_{\R_+^{N+1}}y^{1-2s}|\nabla E_s(\Psi_n^i)|^2\dx\dy
=\sum_{i=1}^k[\Psi_n^i]_s^2\\
&=\sum_{i=1}^k\iint_{\R^{2N}}\frac{|\Psi_n^i(x)-\Psi_n^i(y)|^2}{|x-y|^{N+2s}}\dx\dy\\
    &=\sum_{i=1}^k\iint_{\R^{2N}}(\si_n^i)^{\frac{2N}{2^*}+(N+2s)}\frac{|\Psi^i(\si_n^i(x-x_n^i))-\Psi^i(\si_n^i(y-x_n^i))|^2}{|\si_n^i(x-x_n^i)-\si_n^i(y-x_n^i)|^{N+2s}}\dx\dy\\
    &=\sum_{i=1}^k(\si_n^i)^{2s-2}\iint_{\R^{2N}}\frac{|\Psi^i(x)-\Psi^i(y)|^2}{|x-y|^{N+2s}}\dx\dy=O\left(\sum_{i=1}^k (\si_n^i)^{2s-2}\right)=o_n(1).
\end{align*}
Hence, $$U_n=U_{\infty}+\sum_{i=1}^kE_s(\Psi_n^i)  + o_n(1) \text{ in } \XX^s(\R_+^{N+1}).$$
This completes the proof. 
\iffalse++++++++++++++++++   
% \noi \textbf{Step-4:} Let $R_n=\text{Tr}(U_n)-\text{Tr}(U_{\infty})-\sum_{i=1}^k\Psi_n^i$. Then by the previous step, $R_n\to0$ in $\DD^{1,2}(\RR)$. Observe that 
%     \begin{align}
%     \sum_{i=1}^k[\Psi_n^i]_s^2+[R_n]_s^2&\leq C\bigg([\text{Tr}(U_n)-\text{Tr}(U_{\infty})]_s^2+\sum_{i=1}^k[\Psi_n^i]_s^2\bigg)=C\sum_{i=1}^k[\Psi_n^i]_s^2+o_n(1)\nonumber\\
%     &=C\sum_{i=1}^k\iint_{\R^{2N}}\frac{|\Psi_n^i(x)-\Psi_n^i(y)|^2}{|x-y|^{N+2s}}\dx\dy+o_n(1)\nonumber\\
%     &=C\sum_{i=1}^k\iint_{\R^{2N}}(\si_n^i)^{\frac{2N}{2^*}+(N+2s)}\frac{|\Psi^i(\si_n^i(x-x_n^i))-\Psi^i(\si_n^i(y-x_n^i))|^2}{|\si_n^i(x-x_n^i)-\si_n^i(y-x_n^i)|^{N+2s}}\dx\dy+o_n(1)\nonumber\\
%     &=C\sum_{i=1}^k(\si_n^i)^{2s-2}\iint_{\R^{2N}}\frac{|\Psi^i(x)-\Psi^i(y)|^2}{|x-y|^{N+2s}}\dx\dy+o_n(1)\nonumber\\
%     &=O\left(\sum_{i=1}^k (\si_n^i)^{2s-2}\right)+o_n(1)=o_n(1).\nonumber
% \end{align}
% By the isometry, 
% \begin{align*}
%     \sum_{i=1}^k\int_{\R_+^{N+1}}y^{1-2s}|\nabla E_s(\Psi_n^i)|^2\dx\dy&+\int_{\R_+^{N+1}}y^{1-2s}|\nabla E_s(R_n)|^2\dx\dy=\sum_{i=1}^k[\Psi_n^i]_s^2+[R_n]_s^2=o_n(1).
% \end{align*}
% Hence, using the linearity of the extension map $E_s$, we obtain the decomposition $$U_n=U_{\infty}+\sum_{i=1}^k E_s(\Psi_n^i)+ o_n(1) \text{ in } \XX^s(\R_+^{N+1}).$$
% This completes the proof. 
+++++++++++++++++++\fi
\end{proof}
\begin{remark}
    Using \eqref{Poisson_formula}, we observe that for any $1\leq i\leq k$, $n\in \NN$, $x\in\RR,y>0$, $$E_s(\Psi_n^i)(x,y)=(\si_{n}^i)^{\frac{N-2}{2}}E_s(\Psi^i)(\si_n^i((x,y)-(x_n^i,0))).$$
\end{remark}
\begin{remark}\label{error_decay} 
    From Step-4 in the proof of the Proposition \ref{PS}, we observe that 
    \begin{align*}
        \int_{\R_+^{N+1}}y^{1-2s}|\nabla (U_n-U_{\infty})|^2\dx\dy=O\left(\sum_{i=1}^k(\si_{n}^i)^{2s-2}\right)+o_n(1)=O\left(\sum_{i=1}^k(\si_{n}^i)^{2s-2}\right).
    \end{align*}
\end{remark}
\subsection{Safe regions}\label{safe_reg_subsec}
For every $n$, we denote $\si_n$ to be the minimum of $\si_n^i$ i.e. $$\si_n:=\min_{1\leq i\leq k}\si_n^i$$ and $(x_n,0)$ be the corresponding concentration point. For each $1\leq i\leq k$, choose $c_i\geq1$ with $c_{i+1} = c_{i} +6$. 
Then for each $n$, consider the half-annuli of the form $$\mathcal{A}_i:=B_{(c_i+5)\si_n^{-\frac{1}{2}}}^+(x_n,0)\setminus \ov{B^+}_{c_i\si_n^{-\frac{1}{2}}}(x_n,0).$$ Since excluding $(x_n,0)$, there are $k-1$ concentration points and $\ov{\mathcal{A}_i} \cap \ov{\mathcal{A}_j} = \{ \cdot \}$ for $1\le i \neq j \le k$, for each $n$ we can find an annulus whose f{}lat boundary $\pa\AA_i\cap\{y=0\}$ does not contain any concentration points. Since there are only $k$ many $c_i$'s, one $c_i$ will be chosen infinitely many times. Further, using Proposition \ref{PS}, $\tilde{I}(\Psi^i) \ge \frac{1}{N}S^{\frac{N}{2}}$ and $\tilde{I}_{\la}(U_n) \le C$ for some $C>0$, we get for large $n$, 
\begin{align*}
    C + |\tilde{I}_{\la}(U_{\infty})| \ge \frac{k}{N}S^{\frac{N}{2}},
\end{align*}
which infers that $k$ is bounded by a fixed constant. Hence passing through a subsequence, we find $\ov{C}\geq1$, independent of $n$, such that for any $n$, the f{}lat boundary of the annulus $$\AA_n^1:=\left(B^+_{(\ov{C}+5)\si_n^{-\frac{1}{2}}}(x_n,0)\setminus \ov{B^+}_{\ov{C}\si_n^{-\frac{1}{2}}}(x_n,0)\right)\cap(\Om\times(0,\infty)),$$ does not contain any concentration points $(x_n^i,0)$.
We further define the following thinner subsets of $\AA_n^1$:
\begin{align*}
    \AA_n^2:=\left(B^+_{(\ov{C}+4)\si_n^{-\frac{1}{2}}}(x_n,0)\setminus \ov{B^+}_{(\ov{C}+1)\si_n^{-\frac{1}{2}}}(x_n,0)\right)\cap(\Om\times(0,\infty)),\\
    \AA_n^3:=\left(B^+_{(\ov{C}+3)\si_n^{-\frac{1}{2}}}(x_n,0)\setminus \ov{B^+}_{(\ov{C}+2)\si_n^{-\frac{1}{2}}}(x_n,0)\right)\cap(\Om\times(0,\infty)).
\end{align*}

\subsection{Estimates under auxiliary norm}\label{estsunderauxnorm}
We recall an elementary inequality: Let $q \in [2,2^*]$. Then there exists $A(\lambda)>1$ independent of $q$ such that  
\begin{align}\label{ineq-1}
    \left| |t|^{q-2}t+\la |t|^{p-2}t \right| \leq A|t|^{2^*-1}+A, \text{ for all } t\in\R.
\end{align}
We consider the following superset of $\Omega$:
\begin{enumerate}[label={($\bf O_1$)}]
\item \label{O_1}  Let $\Tilde{\Omega} \supsetneq \Omega$ be such that $\text{dist}(\Om,\pa\Tilde{\Om})>0$.
\end{enumerate}
Let $u_n$ be solutions of \eqref{sub_pro} uniformly bounded in $X_0(\Om)$. For each $n \in \N$, let $v_n\in X_0(\tilde{\Om})$ weakly solve the following problem (whose existence and uniqueness are guaranteed by the Riesz representation theorem):
\begin{equation}\label{PDE1}
    \del u+\fra u=A|u_n|^{2^*-1}+A \text{ in } \Tilde{\Om}, \; u=0\text{ in } \RR \setminus \Tilde{\Om}.
\end{equation}
Taking $v_n^-\in X_0(\tilde{\Om})$ as a test function in the weak formulation, 
$$-\rho(v_n^-)^2-\iint_{\R^{2N}}[v_n^+(x)v_n^-(y)+v_n^+(y)v_n^-(x)]\d\mu=\int_{\Om}(A|u_n|^{2^*-1}+A)v_n^-\dx.$$
Hence, $v_n^-\equiv0$ and $v_n$ is a non-negative solution. Applying the strong maximum principle \cite[Theorem 3.1]{BiMuVe}, $v_n$ is a positive solution. Now, for any $0\leq \phi\in \mathcal{C}_c^{\infty}(\Om)$,
\begin{align*}
    \int_{\Om}\nabla v_n\cdot\nabla\phi\dx+\A(v_n,\phi)&=\int_{\Tilde{\Om}}\nabla v_n\cdot\nabla\phi\dx+\A(v_n,\phi)=\int_{\Om}(A|u_n|^{2^*-1}+A)\phi\dx\\& \geq \int_{\Om}(|u_n|^{p_n-2}u_n+\la |u_n|^{p-2}u_n)\phi\dx=\int_{\Om}\nabla u_n\cdot\nabla\phi\dx+\A(u_n,\phi).
\end{align*}
Similarly, we see that
$$\int_{\Om}\nabla v_n\cdot\nabla\phi\dx+\A(v_n,\phi)\geq \int_{\Om}\nabla (-u_n)\cdot\nabla\phi\dx+\A(-u_n,\phi).$$
Since $|u_n|\leq v_n$ a.e. in $\RR\setminus\Om$, using the weak comparison principle \cite[Proposition 4.1]{AnCo},
\begin{align}\label{est-1}
    |u_n|\leq v_n \; \text{ a.e. in }\Om.
\end{align}
In view of \eqref{est-1}, it is sufficient to estimate $v_n$. 

\begin{remark}\label{Ap11-1}
    Since $\{u_n\}$ is uniformly bounded in $X_0(\Om)$, it is is also uniformly bounded in $L^{2^*}(\Omega)$ and therefore
\begin{align*}
    \rho(v_n)^2=A\int_{\tilde{\Om}}|u_n|^{2^*-1}v_n+Av_n\dx\leq (A\|u_n\|_{2^*}^{2^*-1}+A|\Tilde{\Om}|^{2^*-1})\|v_n\|_{2^*, \Tilde{\Om}}\leq C\rho(v_n).
\end{align*}
Hence $\{v_n\}$ is also uniformly bounded in $X_0(\tilde{\Om})$.
\end{remark}

\begin{definition}\label{def-1}
    Let $p_1,p_2\in[1,\infty)$ be such that $p_2<2^*<p_1$, $\al>0$ and $\si>0$. We consider a system of inequalities
    \begin{equation}\label{ineq_sys}
        \begin{cases}
            \|u_1\|_{p_1}\leq\al,\\
            \|u_2\|_{p_2}\leq\al\si^{\frac{N}{2^*}-\frac{N}{p_2}}.
        \end{cases}
    \end{equation}
We define the following norm on $X_0$:
    $$\|u\|_{p_1,p_2,\si}=\inf\left\{\al>0:\exists \, u_1,u_2 \text{ satisfying }\eqref{ineq_sys} \text{ and }|u|\leq u_1+u_2 \right\}.$$
\end{definition}
Notice that 
\begin{align}\label{ineq-2}
    \|u\|_{p_1,p_2,\si}\leq \min\left\{\|u\|_{p_1},\|u\|_{p_2}\si^{\frac{N}{p_2}-\frac{N}{2^*}}\right\}.
\end{align}
\begin{lemma}\label{lemma_11}
    Let $\tilde{\Omega}$ be as given in {\rm \ref{O_1}}, $u,v\in X_0(\tilde{\Om})$,  and $a\in L^{\frac{N}{2}}(\Tilde{\Om})$ be three positive functions such that the following holds weakly: $$\del u+\fra u\leq av \text{ in } \Tilde{\Om}, \; u= 0 \text{ in } \RR \setminus \Tilde{\Om}.$$
    Then for each $p_1,p_2\in(\frac{N}{N-2},\infty)$ there exist $C(N,p_1,p_2)$ such that for any $\si>0$, $$\|u\|_{p_1,p_2,\si}\leq C(N,p_1,p_2)\|a\|_{\frac{N}{2}}\|v\|_{p_1,p_2,\si}.$$
\end{lemma}
\begin{proof}
For $\si,\eps>0$ and $\al=\|v\|_{p_1,p_2,\si}+\eps$, there exist $v_1,v_2\in X_0(\tilde{\Om})$ such that $v\leq v_1+v_2$, where $v_1,v_2$ satisfy \eqref{ineq_sys}. For $i=1,2$, let $u_i$ weakly solve the problem $$\del u_i+\fra u_i=av_i \text{ in } \Tilde{\Om}, \; u_i= 0 \text{ in } \RR \setminus \Tilde{\Om}.$$
    By Lemma \ref{Moser_lemma}-(ii),
    \begin{align*}
        \|u_1\|_{p_1}&\leq C(N,p_1)\|a\|_{\frac N2}\|v_1\|_{p_1}\leq C(N,p_1,p_2)\|a\|_{\frac N2}(\|v\|_{p_1,p_2,\si}+\eps),\\
        \|u_2\|_{p_2}&\leq C(N,p_2)\|a\|_{\frac N2}\|v_2\|_{p_2}\leq C(N,p_1,p_2)\|a\|_{\frac N2}(\|v\|_{p_1,p_2,\si}+\eps)\si^{\frac{N}{2^*}-\frac{N}{p_2}}.
    \end{align*}
    Further, since
    \begin{align*}
        \left(\del u_1 +\fra u_1\right) + \left(\del u_2 + \fra u_2\right) =av_1+av_2\geq av\geq\del u + \fra u \text{ in } \Tilde{\Om},
    \end{align*}
     using the weak comparison principle \cite[Proposition 4.1]{AnCo}, $u\leq u_1+u_2$. By the definition of $\|\cdot\|_{p_1,p_2,\si}$ with $\al=C(N,p_1,p_2)\|a\|_{\frac N2}(\|v\|_{p_1,p_2,\si}+\eps)$ and the above argument, for every $\eps>0$, $$\|u\|_{p_1,p_2,\si}\leq C(N,p_1,p_2)\|a\|_{\frac N2}(\|v\|_{p_1,p_2,\si}+\eps).$$ Taking $\eps\ra 0$, concludes the lemma. 
\end{proof}
\begin{lemma}\label{lemma_12}
    Let $p_1,p_2\in\left(\frac{N+2}{N-2},\frac{N}{2}\frac{N+2}{N-2}\right)$ with $p_2<2^*<p_1$. Let, $q_i$ be defined as $$\frac1{q_i}=\frac{N+2}{N-2}\frac{1}{p_i}-\frac2N, \quad i=1,2.$$
     Let $\tilde{\Omega}$ be as given in {\rm \ref{O_1}}. Suppose $u$ and $v$ are two positive functions whose support is contained in $\tilde{\Om}$ and the following holds weakly: 
    $$\del u+\fra u\leq Av^{2^*-1}+A  \text{ in } \Tilde{\Om}, \; u=0\text{ in } \RR \setminus \Tilde{\Om}.$$ Then there exists $C(N,p_1,p_2,\Tilde{\Om})$ such that for every $\si>0$, $$\|u\|_{q_1,q_2,\si}\leq C(N,p_1,p_2,\Tilde{\Om}) \left((\|v\|_{p_1,p_2,\si})^{\frac{N+2}{N-2}}+1\right).$$
\end{lemma}
\begin{proof}
For $\si,\eps>0$ and $\al=\|v\|_{p_1,p_2,\si}+\eps$, there exist $v_1,v_2\in X_0(\tilde{\Om})$ such that $v\leq v_1+v_2$, where $v_1,v_2$ satisfy \eqref{ineq_sys}. For $i=1,2$, let $u_1,u_2$ weakly solve
    \begin{align*}
        \del u_1+\fra u_1&=2^{\frac{4}{N-2}}Av_1^{\frac{N+2}{N-2}}+A  \text{ in } \Tilde{\Om}, \; u_1=0\text{ in } \RR \setminus \Tilde{\Om},\\
        \del u_2+\fra u_2&= 2^{\frac{4}{N-2}}Av_2^{\frac{N+2}{N-2}} \text{ in } \Tilde{\Om}, \; u_2=0\text{ in } \RR \setminus \Tilde{\Om}.
    \end{align*}
Since $$\del u+\fra u\leq Av^{\frac{N+2}{N-2}}+A\leq 2^{\frac{4}{N-2}}Av_1^{\frac{N+2}{N-2}}+A+2^{\frac{4}{N-2}}Av_2^{\frac{N+2}{N-2}}=\del u_1+\fra u_1\del u_2+\fra u_2,$$
using the weak comparison principle $u\leq u_1+u_2$. Further, notice that 
\begin{align*}
    1<p_1\frac{N-2}{N+2}<\frac N2, \text{ and } \frac{Np_1\frac{N-2}{N+2}}{N-2p_1\frac{N-2}{N+2}}=q_1.
\end{align*}
Hence using Lemma \ref{Moser_lemma}-(i),
\begin{align*}
    \|u_1\|_{q_1}&\leq C(N,p_1)\|2^{\frac{4}{N-2}}Av_1^{\frac{N+2}{N-2}}+A\|_{p_1\frac{N-2}{N+2}}\leq C(N,p_1)\left(\|v_1\|_{p_1}^{\frac{N+2}{N-2}}+|\Om|^{\frac1{p_1}\frac{N+2}{N-2}}\right)\\
    &\leq C(N,p_1,\tilde{\Om})\left((\|v\|_{p_1,p_2,\si}+\eps)^{\frac{N+2}{N-2}}+1\right).
\end{align*}
Similarly, we have
\begin{align*}
    \|u_2\|_{q_2}\leq C(N,p_2)\|v_2\|_{p_2}^{\frac{N+2}{N-2}}&\leq C(N,p_2)(\|v\|_{p_1,p_2,\si}+\eps)^{\frac{N+2}{N-2}}\si^{\left(\frac N{2^*}-\frac N{p_2}\right)\frac{N+2}{N-2}}\\
    &=C(N,p_2)(\|v\|_{p_1,p_2,\si}+\eps)^{\frac{N+2}{N-2}}\si^{\left(\frac N{2^*}-\frac N{q_2}\right)},
\end{align*}
where in the last equality, we use the fact that $$\left(\frac N{2^*}-\frac N{p_2}\right)\frac{N+2}{N-2}=\frac N{2^*}-\frac N{q_2}.$$
Hence, in view of the  definition of $\|\cdot\|_{p_1,p_2,\si}$, with $\al= C(N,p_1,p_2,\Tilde{\Om})((\|v\|_{p_1,p_2,\si}+\eps)^{\frac{N+2}{N-2}}+1)$, we get $$\|u\|_{q_1,q_2,\si}\leq C(N,p_1,p_2,\Tilde{\Om})\left((\|v\|_{p_1,p_2,\si}+\eps)^{\frac{N+2}{N-2}}+1\right).$$
Taking $\eps\to0$, we conclude the proof.
\end{proof}
\begin{lemma}\label{lemma_13}
Let $\{u_n\}$ be a sequence of weak solutions to \eqref{sub_pro} and uniformly bounded in $X_0(\Om)$. Assume, $\{v_n\}$ is a sequence of weak solutions to \eqref{PDE1}. Then there exist $C>0$, and $p_1,\, p_2\in (\frac{N+2}{N-2}, \frac{N}{2}\frac{N+2}{N-2})$ with $p_2<2^*<p_1$ such that $\|v_n\|_{p_1,p_2,\si_n}\leq C.$
\end{lemma}
\begin{proof}
    From the Palais-Smale decomposition (Proposition \ref{PS}), we
    write $u_n=u_\infty+u_n^1+u_n^2$, where $u_\infty$ weakly solves \eqref{main_PDE}, $$u_{n}^1=\sum_{i=1}^k \Psi_n^i,$$ and $u_{n}^2=u_n-u_\infty-u_n^1$ and $u_n^2\to 0$ in  $L^{2^*}(\tilde{\Om})$. Let $$a_0=A\max(1,3^{\frac{6-N}{N-2}})|u_{\infty}|^{2^*-2},\,a_1=A\max(1,3^{\frac{6-N}{N-2}})|u_n^1|^{2^*-2},\, a_2=A\max(1,3^{\frac{6-N}{N-2}})|u_n^2|^{2^*-2}.$$
    With the above notations, we have $$\del v_n+\fra v_n =A|u_n|^{2^*-1}+A\leq (a_0+a_1+a_2)|u_n|+A.$$
    Let $\GG:X_0(\tilde{\Om})^*\to X_0(\tilde{\Om})$ be the inverse of $\del+\fra$ i.e. $v=\GG(u)$ implies $$\del v+\fra v=u\text{ in }\tilde{\Om},\, v=0\text{ in }\RR\setminus\tilde{\Om}.$$ Observe that $\GG$ is well defined by the Riesz representation theorem. Further, using the weak comparison principle, $u_1\leq u_2$ implies $\GG(u_1)\leq \GG(u_2)$, i.e., $\GG$ is monotone. Hence 
    $$v_n\leq \GG(a_0|u_n|+A)+\GG(a_1|u_n|)+\GG(a_2|u_n|).$$
    Since $u_\infty\in L^{\infty}(\Om)$, $a_0\in L^{\infty}(\Om)$, choosing $p\in (\frac{2N}{N+2},\min\{\frac{2N}{N-2},\frac{N+2}{4}\})\subset (1,\frac{N}{2})$, %which holds because $$\frac{2N}{N-2}<\frac{N}{2}\iff N>6,$$ 
    we get by Lemma \ref{Moser_lemma}-(i), $$\|\GG(a_0|u_n|+A)\|_{p_1}\leq \|a_0|u_n|+A\|_p\leq C,$$ where $p_1=\frac{Np}{N-2p}\in(2^*,\frac{N}{2}\frac{N+2}{N-2})$ and $C>0$ is independent of $n$. Thus using \eqref{ineq-2}, $$\|\GG(a_0|u_n|+A)\|_{p_1,p_2,\si_n}\leq \|\GG(a_0|u_n|+A)\|_{p_1}\leq C.$$
    Choose $r\in(\max\{\frac N4,\frac{2N(N+2)}{N^2+12}\},\frac N2)$ and $p_2$ such that $\frac1{p_2}=\frac1r+\frac1{2^*}-\frac2N$. Then we have $p_2\in(\frac{N+2}{N-2},2^*)$. By Lemma \ref{Moser_lemma}-(iii),
    \begin{align*}
        \|\GG(a_1|u_n|)\|_{p_2}&\leq  C\|a_1\|_r\|u_n\|_{2^*}\leq C\|a_1\|_r\\
        &\leq C\sum_{i=1}^k(\si_n^i)^{\frac{2r-N}{r}}\left(\int_{\RR}\Psi^i(x)^{(2^*-2)r}\dx\right)^{\frac1r}\\
        &\leq Ck\si_n^{\frac{2r-N}{r}}\left(\int_{\RR}\frac{1}{(1+|x|^2)^{2r}}\dx\right)^{\frac1r}\leq C\si_n^{\frac N{2^*}-\frac N{p_2}},
    \end{align*}
    where we use $r>\frac N4$ in the last integral and $\frac{2r-N}{r}=\frac N{2^*}-\frac N{p_2}$ in the exponent of $\si_n$. Hence again using \eqref{ineq-2},
    $$\|\GG(a_1|u_n|)\|_{p_1,p_2,\si_n}\leq \si_n^{\frac{N}{p_2}-\frac{N}{2^*}}\|\GG(a_1|u_n|)\|_{p_2}\leq C.$$
    Next, for $a_2\in L^{\frac{N}{2}}(\tilde{\Om})$, applying Lemma \ref{lemma_11}, $$\|\GG(a_2|u_n|)\|_{p_1,p_2,\si_n}\leq C\|a_2\|_{\frac{N}{2}}\|u_n\|_{p_1,p_2,\si_n}\leq C\|u_n^2\|_{2^*}\|v_n\|_{p_1,p_2,\si_n}\leq \frac12\|v_n\|_{p_1,p_2,\si_n},$$
    where in the last inequality we used $u_n^2\to0$ in $L^{2^*}(\tilde{\Om})$. Hence, the triangle inequality gives 
$$\|v_n\|_{p_1,p_2,\si_n}\leq 2\|\GG(a_0|u_n|+A)\|_{p_1,p_2,\si_n}+2\|\GG(a_1|u_n|)\|_{p_1,p_2,\si_n} \leq C,$$
which is required.
\end{proof}
\begin{remark}\label{interpol_aux_norm}
  Let $\{u_n\}$ and $\{v_n\}$ be as in Lemma~\ref{lemma_13}. We claim that if $\|v_n\|_{p_1,p_2,\si_n}\leq C_1$ for some $p_2<2^*<p_1$ and $C_1>0$ independent of $n$, then for any $p_2<r_2<2^*<r_1<p_1$, $\|v_n\|_{r_1,r_2,\si_n}\leq C_2$ for some $C_2>0$ independent of $n$. Indeed, by the definition of the norm $\|\cdot\|_{p_1,p_2,\si_n}$, there exists $v_{n,i}$ with $|v_n|\leq v_{n,1}+v_{n,2}$ such that $$\|v_{n,1}\|_{p_1}\leq C_1,\quad\|v_{n,2}\|_{p_2}\leq C_1\si_n^{\frac N{2^*}-\frac N{p_2}}.$$
    Set $w_{n,1}=v_{n,1}$ and $w_{n,2}=\min\{v_{n,2},|v_n|\}$. Then $|v_n|\leq w_{n,1}+w_{n,2}$ and $$\|w_{n,1}\|_{r_1}=\|v_{n,1}\|_{r_1}\leq C\|v_{n,1}\|_{p_1}\leq CC_1.$$ Let $\frac1{r_2}=\frac{\theta}{p_2}+\frac{1-\theta}{2^*}$. Observe that
    \begin{align*}
        \|w_{n,2}\|_{r_2}\leq \|w_{n,2}\|_{p_2}^\theta\|w_{n,2}\|_{2^*}^{1-\theta}\leq \|v_{n,2}\|_{p_2}^\theta\|v_n\|_{2^*}^{1-\theta}\leq CC_1\si_n^{(\frac N{2^*}-\frac N{p_2})\theta}=CC_1\si_n^{\frac{N}{2^*}-\frac{N}{r_2}},
    \end{align*}
    where in the last inequality, we use the fact that $\{v_n\}$ is uniformly bounded in $L^{2^*}(\tilde{\Om})$ (see Remark~\ref{Ap11-1}).
\end{remark}
\begin{proposition}\label{new_norm}
    Let $\{u_n\}$ be a sequence of weak solutions to \eqref{sub_pro} and uniformly bounded in $X_0(\Om)$. Then for every $p_1,p_2$ satisfying $\frac{N}{N-2}<p_2<2^*<p_1$, there exists $C>0$ such that $\|u_n\|_{p_1,p_2,\si_n}\leq C.$
\end{proposition}
\begin{proof}
    Recall that $|u_n|\leq v_n$. Thus, $v_n$ satisfies $$\del u+\fra u\leq Av_n^{2^*-1}+A  \text{ in } \Tilde{\Om}, \; u=0\text{ in } \RR \setminus \Tilde{\Om}.$$ Applying Lemma \ref{lemma_13}, we obtain $p_1,p_2\in(\frac{N+2}{N-2},\frac N2\frac{N+2}{N-2})$ with $p_2<2^*<p_1$ such that $\|v_n\|_{p_1,p_2,\si_n}\leq C$. Now applying Lemma \ref{lemma_12} and keeping Remark \ref{interpol_aux_norm} in mind, we enlarge the interval $(p_2,p_1)$ to $(q_2,q_1)$, where $q_1,q_2$ are defined in Lemma \ref{lemma_12}. We apply Lemma \ref{lemma_12} repeatedly until we move to an interval larger than $(\frac{N+2}{N-2},\frac N2\frac{N+2}{N-2})$. Finally we notice that as $p_1\to\frac N2\frac{N+2}{N-2}$, $q_1\to\infty$ and as $p_2\to\frac{N+2}{N-2}$, $q_2\to\frac{2^*}{2}$. This concludes the proof.
\end{proof}
\subsection{Integral estimates over the safe regions} 
In this subsection, we aim to derive several integral estimates of $u_n$ and $U_n$ over the safe regions $\AA_n^i$. We recall a few facts from \cite{BySo}. Suppose $u$ is a weak solution to $$\del u+\fra u=f \text{ in }\Om,\, u=0\text{ in }\RR \setminus \Om,$$ where $f\in L^1(\Om)$. Let $B_{\tilde{r}}^N(x_0)\subset\Om$. Then for any $0<\rho<\tilde{r}$, using \cite[Eq. (6.2)]{BySo}, one has 
\begin{equation}\label{byso1}
    \int_\rho^{\tilde{r}} E(u;x_0,t)\frac{\dt}{t}\leq cE(u;x_0,\tilde{r})+c\int_{\rho}^{\tilde{r}}\left(\frac{1}{t^{N-1}}\int_{B_t^N}|f|\dx\right)\dt,
\end{equation}
where $c=c(N,s,\text{diam}(\Om))$. Now for any $0<\Tilde{\rho}\leq \frac{\rho}{2}<\frac{\tilde{r}}{8}$ and $\theta\in(\frac{1}{4},\frac12]$, by \cite[Eq. (6.3)]{BySo} and \eqref{byso1}, we get 
\begin{align}
    |(u)_{B_\rho^N(x_0)}-(u)_{B_{\Tilde{\rho}}^N(x_0)}|&\leq \int_{\Tilde{\rho}}^{\frac{\rho}{\theta}}E(u;x_0,t)\frac{\dt}{t}\leq \int_{\Tilde{\rho}}^{\tilde{r}}E(u;x_0,t)\frac{\dt}{t}\nonumber\\
    &\leq cE(u;x_0,\tilde{r})+c\int_{\tilde{\rho}}^{\tilde{r}}\left(\frac{1}{t^{N-1}}\int_{B_t^N}|f|\dx\right)\dt.\nonumber
\end{align}
Since, the above inequality is true for any $\rho<\frac{\tilde r}{4}$, taking $\rho\uparrow\frac{\tilde r}{4}$ we have
\begin{equation}\label{byso2}
   |(u)_{B_\frac{\tilde r}{4}^N(x_0)}-(u)_{B_{\Tilde{\rho}}^N(x_0)}|\leq cE(u;x_0,\tilde{r})+c\int_{\tilde{\rho}}^{\tilde{r}}\left(\frac{1}{t^{N-1}}\int_{B_t^N}|f|\dx\right)\dt
\end{equation}

We also observe that for any $\hat{\rho}\in[\frac {\tilde{r}}4,\tilde{r}]$, 
\begin{align}\label{byso3}
    |(u)_{B_{\hat{\rho}}^N(x_0)}-(u)_{B_{\tilde{r}}^N(x_0)}|&\leq\fint_{B_{\hat{\rho}}^N(x_0)}|u-(u)_{B_{\tilde{r}}^N(x_0)}|\dx \no \\
    &\leq 4^N\fint_{B_{\tilde{r}}^N(x_0)}|u-(u)_{B_{\tilde{r}}^N(x_0)}|\dx\leq 4^N E(u;x_0,\tilde{r}).
\end{align}
Recall that $v_n$ weakly solves $$\del u+\fra u=A|u_n|^{2^*-1}+A\text{ in }\Tilde{\Om}, \; u=0\text{ in } \RR \setminus \Tilde{\Omega},$$
where $\tilde{\Om}$ is given in \ref{O_1} and $\{v_n\}$ is also uniformly bounded in $X_0(\tilde{\Om})$.
\begin{proposition}\label{estimate-1}
For $\gamma\in[1,\frac{N}{N-2})$, $r\in [\ov{C}\si_n^{-\frac12},\frac{\text{dist}(\Om,\pa\Tilde{\Om})}{8}]$; where $\Tilde{\Om}$ is given in {\rm\ref{O_1}}, and $x_0\in\Om$, the following holds
$$\left(\frac1{r^N}\int_{B_r^N(x_0)\cap\Om}|u_n|^\gamma\right)^{\frac1{\gamma}}\leq C(N,s,\Om,\Tilde{\Om}).$$
\end{proposition}
\begin{proof}
    Let $x_0\in\Om$, $R:=\text{dist}(x_0,\pa\Tilde{\Om})$ and $0<r\leq \frac{R}{8}$. As $v_n$ is a nonnegative supersolution to the homogeneous problem $$\del u+\fra u=0\text{ in }\Tilde{\Om}, \; u=0\text{ in } \RR \setminus \Tilde{\Omega},$$ using the weak Harnack inequality (see \cite[Theorem 8.4]{GaKi}) for any $\gamma\in[1,\frac{N}{N-2})$,  
    \begin{equation}\label{ineq_1.1}
        \left(\fint_{B_r^N(x_0)}v_n^\gamma\right)^{\frac1{\gamma}}\leq c(N,s)\essinf_{B_{2r}(x_0)} v_n.
    \end{equation}
By the triangle inequality and using \eqref{byso3} (with $\hat{\rho}$ and $\tilde r$ replaced by $\frac{R}{4}$ and $R$) and \eqref{byso2} (with 
$\tilde\rho,\, \tilde r$ replaced by $r$ and $R$ respectively) we see that
\begin{align}
    &\left|(v_n)_{B_R^N(x_0)}-(v_n)_{B_{r}^N(x_0)}\right|\no\\
    &\leq \left|(v_n)_{B_R^N(x_0)}-(v_n)_{B_{\frac{R}{4}}^N(x_0)}\right|+\left|(v_n)_{B_{\frac{R}{4}}^N(x_0)}-(v_n)_{B_{r}^N(x_0)}\right|\no\\
    &\leq c E(v_n;x_0,R)+ c\int_r^ R\frac1{t^{N-1}}\left(\int_{B_t^N(x_0)}(A|u_n|^{2^*-1}+A)\dx\right)\dt,\label{ineq_1.2}
\end{align}
where $c=c(N,s,\text{diam}(\Tilde{\Om}))$. From \eqref{ineq_1.2}, we get 
\begin{align}
    \essinf_{B_{2r}^N(x_0)}v_n&\leq \essinf_{B_{r}^N(x_0)}v_n \leq(v_n)_{B_{r}^N(x_0)}\leq |(v_n)_{B_R^N(x_0)}-(v_n)_{B_{r}^N(x_0)}|+(v_n)_{B_R^N(x_0)}\nonumber\\
    &\leq (v_n)_{B_R^N(x_0)}+c E(v_n;x_0,R)+ c\int_r^R\frac1{\rho^{N-1}}\left(\int_{B_\rho^N(x_0)}(A|u_n|^{2^*-1}+A)\dx\right)\d\rho.\label{ineq_1.3}
\end{align}
Further, there exists $c>0$ such that 
\begin{align}
(v_n)_{B_R^N(x_0)}+cE(v_n;x_0,R)&\leq c (v_n)_{B_R^N(x_0)}+c\text{Tail}(v_n-(v_n)_{B_R^N(x_0)};x_0,R)\nonumber
\\
&\leq c((v_n)_{B_R^N(x_0)}+\text{Tail}(v_n;x_0,R)),
\end{align}
where the last inequality holds using $R<\text{diam}(\Tilde{\Om})$.
From the uniform boundedness of $\{v_n\}$ in $L^{2^*}$ norm, we see that the average 
\begin{align}
    (v_n)_{B_R^N(x_0)}=\fint_{B_R^N(x_0)}|v_n|\dx\leq C(N) \|v_n\|_{2^*} R^{\frac{2-N}{2}}\leq C,
\end{align}
where $C$ is independent of $n$ and $R$. For the Tail term we see
\begin{align}
    \text{Tail}(v_n;x_0,R)&=R^2\int_{\RR\setminus B_R^N(x_0)}\frac{|v_n(x)|}{|x-x_0|^{N+2s}}\dx\nonumber\\
    &\leq R^2\left(1+\frac{1+|x_0|}{R}\right)^{N+2s}\int_{\RR}\frac{|v_n(x)|}{(1+|x|)^{N+2s}}\dx\nonumber\\
    &\leq C\int_{\Tilde{\Om}}\frac{|v_n(x)|}{(1+|x|)^{N+2s}}\dx\leq C\|v_n\|_{2^*}|\Tilde{\Om}|^{\frac{1}{(2^*)'}}\leq C,\label{ineq_1.4}
\end{align}
where in the third inequality we use the fact that $\text{dist}(\Om,\pa\Tilde{\Om})<R<\text{diam}(\tilde{\Om})$. Hence combining \eqref{ineq_1.1} and \eqref{ineq_1.3}-\eqref{ineq_1.4} for $0<r\leq\frac{\text{dist}(\Om,\pa\Tilde{\Om})}{8}$, we have 
\begin{align}\left(\frac1{r^N}\int_{B_{r}^N(x_0)\cap\Om}|u_n|^\gamma\right)^{\frac1{\gamma}}&\leq C(N)\left(\fint_{B_r^N(x_0)}v_n^\gamma\right)^{\frac1{\gamma}}\nonumber\\&\leq C+C\int_r^R\frac1{\rho^{N-1}}\left(\int_{B_\rho^N(x_0)}|u_n|^{2^*-1}\dx\right)\d\rho\label{ineq_1.5}.
\end{align}
Next, we estimate the last integration for $r\in [\ov{C}\si_n^{-\frac12},\frac{\text{dist}(\Om,\pa\Tilde{\Om})}{8}]$. By Proposition \ref{new_norm}, there exists $C>0$ independent of $n$ such that $\|u_n\|_{q_1,q_2,\si_n}\leq C$ for any $q_1,q_2$ satisfying $\frac{2^*}{2}<q_2<2^*<q_1.$ Let $q_2=\frac{N+2}{N-2}$ and $q_1=N\frac{N+2}{N-2}$. Then we choose $u_{1,n}$ and $u_{2,n}$ such that $$|u_n|\leq u_{1,n}+u_{2,n},\quad \|u_{1,n}\|_{q_1}\leq C,\quad \|u_{2,n}\|_{q_2}\leq C\si_n^{\frac{N}{2^*}-\frac N{q_2}}=C\si_n^{-\frac{N}{2^*(2^*-1)}}.$$
We estimate
\begin{align*}
    \int_r^R\frac1{\rho^{N-1}}\left(\int_{B_\rho^N(x_0)}|u_{1,n}|^{2^*-1}\dx\right)\d\rho&\leq C\|u_{1,n}\|_{q_1}^{2^*-1}\int_0^R\frac1{\rho^{N-1}}(\rho^N)^{\frac{N-1}{N}}\d\rho\\
    &\leq CR\|u_{1,n}\|_{q_1}^{2^*-1}\leq C,
\end{align*}
and using $r\geq \ov{C}\si_n^{-\frac12}$,
\begin{align*}
    \int_r^R\frac1{\rho^{N-1}}\left(\int_{B_\rho^N(x_0)}|u_{2,n}|^{2^*-1}\dx\right)\d\rho&\leq C\int_r^R\frac1{\rho^{N-1}}\|u_{2,n}\|_{q_2}^{2^*-1}\d\rho\\
    &\leq C\int_r^R\frac1{\rho^{N-1}}\si_n^{-\frac{N}{2^*}}\d\rho\leq C\si_n^{-\frac{N}{2^*}}r^{2-N}\leq C.
\end{align*}
Combining the above estimates with \eqref{ineq_1.5}, we obtain the required result.
\end{proof}
\begin{remark}
   Using similar arguments we can give more information about a general weak solution to $$\del u+\fra u=f \text{ in }\Om, u=0\text{ in }\RR\setminus\Om.$$
    Assume that $f\in L^{(2^*)'}(\RR)$ where $(2^*)'$ is the H\"{o}lder conjugate of $2^*$ (or assume $f\in L^{(2^*)'}(\Om)$ and $f$ is zero outside $\Om$). As in Section \ref{estsunderauxnorm}, let $v$ be the unique weak solution to $$\del u+\fra u=|f| \text{ in }\tilde{\Om}, \quad u=0\text{ in }\RR\setminus\tilde{\Om}.$$
    Then $v$ is non-negative and $$|u|\leq v \,\text{ a.e. in }\Om.$$
    Furthermore, $$\rho(v)^2=\int_{\tilde{\Om}}|f|v\dx\leq \|f\|_{(2^*)'}\|v\|_{2^*}\leq c\|f\|_{(2^*)'}\rho(v).$$
    Thus, $$\|v\|_{2^*}\leq c\rho(v)\leq c\|f\|_{(2^*)'}.$$
    Following the arguments of Proposition \ref{estimate-1},  for any $x_0\in\Om,$ $0<r\leq \frac{R_0}{8}:=\frac{\text{dist}(\Om,\pa\tilde{\Om})}{8}$,
    we get
    \begin{align*}
        \left(\fint_{B_r^N(x_0)}|u|^\gamma\right)^{\frac1{\gamma}}&\leq\left(\fint_{B_r^N(x_0)}v^\gamma\right)^{\frac1{\gamma}}\\&\leq c((v)_{B_{R_0}^N(x_0)}+\text{Tail}(v;x_0,{R_0}))+c\int_r^{R_0}\frac1{\rho^{N-1}}\left(\int_{B_\rho^N(x_0)}|f|\dx\right)\d\rho\\
        &\leq c\|f\|_{(2^*)'}+c\int_r^{R_0}\frac1{\rho^{N-1}}\left(\int_{B_\rho^N(x_0)}|f|\dx\right)\d\rho,
    \end{align*}
    where $c=c(N,s,\text{diam}(\tilde{\Om}))$. This estabishes \eqref{form_2}.
\end{remark}
Next, we prove a similar estimate for the extension $U_n$. For that, we impose the second condition on $\tilde{\Om}$:

\begin{enumerate}[label={($\bf O_2$)}]
\item \label{O_2} Let $\tilde{\Om}$ be such that $\text{dist}(\Om,\pa\tilde{\Om})\geq16\text{ diam}(\Om)$.
\end{enumerate}

\begin{proposition}\label{extension_estimate-1}
Let $r\in[\ov{C}\si_n^{-\frac12},\text{diam}(\Om)]$, $x_0\in\Om$, and $\tilde{\Om}$ be given as in {\rm \ref{O_1} and \ref{O_2}}. Then
    \begin{align}\label{est-10}
        \frac1{r^{N+2-2s}}\int_{B_r^+(x_0,0)}y^{1-2s}\left( \int_{\Om}P(x-\xi,y)|u_n(\xi)|\d\xi\right)\dx\dy\leq C(N,s,\Om,\tilde{\Om}).
    \end{align}
    As a consequence,
    \begin{equation}
        \frac1{r^{N+2-2s}}\int_{B_r^+(x_0,0)}y^{1-2s}|U_n(x,y)|\dx\dy\leq  C(N,s,\Om,\tilde{\Om}).
    \end{equation}
\end{proposition}
\begin{proof}
    Using \eqref{Poisson_formula} and Tonelli's theorem, 
    \begin{align*}
        \int_{B_r^+(x_0,0)}&y^{1-2s}|U_n(x,y)|\dx\dy\leq\int_{B_r^+(x_0,0)}y^{1-2s}\left( \int_{\Om}P(x-\xi,y)|u_n(\xi)|\d\xi\right)\dx\dy\\
        &\leq\int_{\Om}|u_n(\xi)|\left(\int_0^r\int_{B_r^N(x_0)}y^{1-2s}P(x-\xi,y)\dx\dy\right)\d\xi\\
        &=\left( \int_{B_{2r}^N(x_0)} + \int_{\Om\setminus B_{2r}^N(x_0)} \right) |u_n(\xi)|\left(\int_0^r\int_{B_r^N(x_0)}y^{1-2s}P(x-\xi,y)\dx\dy\right)\d\xi.
    \end{align*}
    Using the fact $$\int_{\RR}P(x,y)\dx=1,$$ we observe that
    \begin{align*}
        &\int_{B_{2r}^N(x_0)}|u_n(\xi)|\left(\int_0^r\int_{B_r(x_0)}y^{1-2s}P(x-\xi,y)\dx\dy\right)\d\xi\\
        &\leq \int_{B_{2r}^N(x_0)}|u_n(\xi)|\left(\int_0^r\int_{\RR}y^{1-2s}P(x-\xi,y)\dx\dy\right)\d\xi\\
        &=\frac{r^{2-2s}}{2-2s}\int_{B_{2r}^N(x_0)}|u_n(\xi)|\d\xi\leq  C(N,s,\Om,\tilde{\Om}) r^{N+2-2s},
    \end{align*}
    where the last inequality follows from  Proposition \ref{estimate-1}. When, $x\in B_{r}^N(x_0)$ and $\xi\in \Om\setminus B_{2r}^N(x_0)$, we have $|x-\xi|>\frac12|\xi-x_0|$. Thus, 
    $$y^{1-2s}P(x-\xi,y)= \frac{y}{(|x-\xi|^2+y^2)^{\frac{N+2s}{2}}}\leq  \frac{Cy}{|\xi-x_0|^{N+2s}}.$$
    Thus, 
    $$\int_0^r\int_{B_r^N(x_0)}y^{1-2s}P(x-\xi,y)\dx\dy\leq  \frac{Cr^{N+2}}{|\xi-x_0|^{N+2s}}.$$
    Now, 
    \begin{align*}
        &\int_{\Om\setminus B_{2r}^N(x_0)}|u_n(\xi)|\left(\int_0^r\int_{B_r^N(x_0)}y^{1-2s}P(x-\xi,y)\dx\dy\right)\d\xi\\
        &\leq  Cr^{N+2}\sum_{j=1}^{\infty}\int_{B_{2^{j+1}r}^N(x_0)\setminus B_{2^{j}r}^N(x_0)}\frac{|u_n(\xi)|}{|\xi-x_0|^{N+2s}}\d\xi\\
        &\leq   Cr^{2-2s}\sum_{j=1}^{\infty}\frac1{2^{j(N+2s)}}\int_{B_{2^{j+1}r}^N(x_0)\setminus B_{2^{j}r}^N(x_0)}|u_n(\xi)|\d\xi.
    \end{align*}
    Observe that if $2^{j+1}r\leq \frac{\text{dist}(\Om,\pa\tilde{\Om})}{8}$, then
    $$\int_{B_{2^{j+1}r}^N(x_0)\setminus B_{2^{j}r}^N(x_0)}|u_n(\xi)|\d\xi\leq \int_{B_{2^{j+1}r}^N(x_0)}|u_n(\xi)|\d\xi\leq C(N,s,\Om,\tilde{\Om})2^{(j+1)N}r^N.$$
    If $2^jr\geq \text{diam}(\Om)$, then $B_{2^{j+1}r}^N(x_0)\setminus B_{2^{j}r}^N(x_0)$ lies outside $\Om$ and thus
    $$\int_{B_{2^{j+1}r}^N(x_0)\setminus B_{2^{j}r}^N(x_0)}|u_n(\xi)|\d\xi=0.$$
    Only remaining case is when $2^{j+1}r> \frac{\text{dist}(\Om,\pa\tilde{\Om})}{8}$ and $2^jr< \text{diam}(\Om)$ but this can not happen, as \ref{O_2} infers that 
    $$2^{j+1}r> \frac{\text{dist}(\Om,\pa\tilde{\Om})}{8}\geq 2\text{diam}(\Om)\implies 2^j r>\text{diam}(\Om).$$
    Hence,
    \begin{align*}
        &\int_{\Om\setminus B_{2r}^N(x_0)}|u_n(\xi)|\left(\int_0^r\int_{B_r(x_0)}y^{1-2s}P(x-\xi,y)\dx\dy\right)\d\xi\\
        &\leq C(N,s,\Om,\tilde{\Om})r^{N+2-2s}\sum_{j=1}^{\infty}\frac{2^{(j+1)N}}{2^{j(N+2s)}}=C(N,s,\Om,\tilde{\Om})r^{N+2-2s}2^N\sum_{j=1}^{\infty}\frac{1}{2^{2sj}}\\
        &\leq C(N,s,\Om,\tilde{\Om})r^{N+2-2s}.
    \end{align*}
Accumulating all the estimates, we get \eqref{est-10}.  
\end{proof}
In the next proposition, we get the estimate of $u_n$ and $U_n$ over the safe region $\AA_n^2$ for any $q \ge 1$. 
\begin{proposition}\label{integ_prop}
For every $q\geq1$, it holds 
    \begin{align}
        &\int_{\pa\AA_n^2\cap\{y=0\}}|u_n|^q\dx\leq C\si_n^{-\frac N2},\label{q-lemma_1}\\
        &\int_{\AA_n^2}y^{1-2s}|U_n|^q\dx\dy\leq C\si_n^{-\left(\frac{N+2-2s}{2}+\frac{(s-1)q}{2^*}\right)},\label{q-lemma_2}
    \end{align}
    where $C$ is independent of $n$.
\end{proposition}
\begin{proof}
By Proposition \ref{PS}, $\text{Tr}(U_n)=\text{Tr}(U_\infty)+\sum_{i=1}^k \Psi_n^i+o_n(1)$ in $L^{2^*}(\RR)$. Since $\text{Tr}(U_\infty)$ weakly solves \eqref{main_PDE}, by Moser iteration \cite[Theorem 1.1]{SuVaWeZh1}, $\text{Tr}(U_\infty) \in L^{\infty}(\Om)$. Note that $B_{\si_n^{-\frac12}}^N(x_0)$ (where $(x_0,0)\in\pa\AA_n^2\cap\{y=0\}$) does not contain any concentration points, i.e. 
\begin{align*}
    |x-x_n^i|\geq C\si_n^{-\frac12}, \; i=1,\ldots,k, \text{ for every } x\in B_{\si_n^{-\frac12}}^N(x_0),
\end{align*}
where $C>0$ is independent of $n$. Now using the definition of $\si_n$ and $N\geq 3$, for every $x\in B_{\si_n^{-\frac12}}^N(x_0)$, 
$$\Psi_n^i(x)=(\si_n^i)^{\frac N{2^*}}\frac{1}{(1+|\si_n^i(x-x_n^i)|^2)^{\frac{N-2}{2}}}\leq \frac{(\si_n^i)^{\frac{2-N}{2}}}{|x-x_n^i|^{N-2}}\leq C\left(\frac{\si_n}{\si_n^i}\right)^{\frac
{N-2}2}\leq C,$$
where $C$ is independent of $n$. Thus,
\begin{align}\label{N/2_norm _0}
    \displaystyle \int_{B_{\si_n^{-\frac12}}(x_0)}|\text{Tr}(U_n)|^{2^*}\dx
        &\leq C\int_{B_{\si_n^{-\frac12}}(x_0)} \left(|\text{Tr}(U_\infty)|^{2^*}+\sum_{i=1}^k|\Psi_n^i|^{2^*}\dx+(o_n(1))^{2^*} \right) \dx \no\\
        &\leq C\si_n^{-\frac N2}+C\|o_n(1)\|^{2^*}_{L^{2^*}(\RR)}\to0.
\end{align}
Define $w_n(x):=|u_n|(\si_n^{-\frac12}x)$ supported on $\Om_n=\si_n^{\frac12}\Om$. Applying Kato-type inequalities, we see that $w_n$ weakly satisfies  
\begin{align*}
    \del w_n+\si_n^{s-1}\fra w_n\leq \si_n^{-1}(w_n^{p_n-2}+\la w_n^{p-2})w_n\text{ a.e. in } \Omega_n, \; w_n = 0 \text{ in } \RR \setminus \Omega_n.
\end{align*}
The Caf{}farelli-Silvestre's extension $W_n$ of $w_n$ weakly satisfies 
\begin{equation*}
    \begin{cases}
        \text{div}(y^{1-2s}\nabla W_n)=0\text{ in }\R_+^{N+1},\\
        W_n(x,0)=w_n(x)\text{ in }\RR,\\
        -\Delta w_n-\si_n^{s-1}\lim_{y\to0^+}y^{1-2s}\frac{\pa W_n}{\pa y}\leq (\si_n^{-1}(w_n^{p_n-2}+\la w_n^{p-2}))w_n\text{ in }\Om_n.
    \end{cases}
\end{equation*}
Using the Poisson kernel, we get the following relation between $U_n$ and $W_n$:
\begin{align*}
    |U_n(\si_n^{-\frac12}x,\si_n^{-\frac12}y)|\leq\int_{\Om}P(\si_n^{-\frac12}x-\xi,\si_n^{-\frac12}y)|u_n(\xi)|\d\xi
    =\int_{\Om_n}P(x-\zeta,y)w_n(\zeta)\d\zeta=W_n(x,y).
\end{align*}
Let $y_0=\si_n^{\frac12}x_0$. Then by \eqref{N/2_norm _0},
\begin{align*}
    \int_{B_1(y_0)}|\si_n^{-1}(w_n^{p_n-2}+\la w_n^{p-2})|^{\frac{N}{2}}\dz&\leq C\int_{B_1(y_0)}|\si_n^{-1}(w_n^{2^*-2}+1)|^{\frac{N}{2}}\dz\leq C\int_{B_{\si_n^{-\frac12}}(x_0)}(|u_n|^{2^*}+1)\dx\\&\leq C\int_{B_{\si_n^{-\frac12}}(x_0)} \left( |\text{Tr}(U_n)|^{2^*}+1 \right)\dx\to0.
\end{align*}
Now, applying Lemma \ref{Moser_iteration}, Proposition~\ref{estimate-1} and Proposition~\ref{extension_estimate-1} for large $n$, we get the following estimate for every $(x_0,0)\in\pa\AA_n^2\cap\{y=0\}$:
\begin{align}
    &\left(\si_n^{\frac N2}\int_{B_{\frac12\si_n^{-\frac12}}(x_0)}|u_n|^q\dx\right)^{\frac1q}+\si_n^{\frac{s-1}{2^*}}\left(\si_n^{\frac{N+2-2s}{2}}\int_{B_{\frac12\si_n^{-\frac12}}^+(x_0,0)}y^{1-2s}|U_n(x,y)|^q\dx\dy\right)^{\frac1q}\nonumber\\
    &\leq\|w_n\|_{L^{q}(B_{\frac12}(y_0))}+\si_n^{\frac{s-1}{2^*}}\|W_n\|_{L^q(B_{\frac12}^+(y_0,0)),y^{1-2s})}\nonumber\\
    &\leq C\left(\|w_n\|_{L^{1}(B_{1}(y_0))}+\si_n^{\frac{s-1}{2^*}}\|W_n\|_{L^1(B_{1}^+(y_0,0)),y^{1-2s})}\right)\nonumber\\
    &= C\si_n^{\frac N2}\int_{B_{\si_n^{-\frac12}}(x_0)}|u_n|\dx+\si_n^{\frac{s-1}{2^*}}\si_n^{\frac{N+2-2s}{2}}\int_{B_{\si_n^{-\frac12}}^+(x_0,0)}y^{1-2s}\left( \int_{\Om}P(x-\xi,y)|u_n(\xi)|\d\xi\right)\dx\dy\nonumber\\
    &\leq C\label{est_q0},
\end{align}
where $C$ is independent of $n$. The last equality holds using \eqref{Poisson_formula} for $W_n$ and the change of variables. Since $\pa\AA_n^2\cap\{y=0\}$ can be covered by finitely many balls $B_{\frac12\si_n^{-\frac12}}(x_0)$, number of balls being independent of $n$, we get
$$\int_{\pa\AA_n^2\cap\{y=0\}}|u_n|^q\dx\leq C\si_n^{-\frac N2}, \text{ for any } q \ge 1.$$ Observe that $\AA_n^2$ can not be fully covered by half balls of the form $B_{\frac12\si_n^{-\frac12}}^+(x_0,0)$. We decompose $\AA_n^2$ as
\begin{align*}
    \AA_n^2 = \AA_n^2\cap \left( \left\{0\leq y\leq\frac12\si_n^{-\frac12} \right\} \cup \left\{\frac12\si_n^{-\frac12}\leq y\leq(\ov{C}+4)\si_n^{-\frac12}\right\} \right).
\end{align*}
\noi \textbf{I:} Using \eqref{est_q0} and a similar covering argument, we conclude that
\begin{align}
\int_{\AA_n^2\cap\left\{0\leq y\leq\frac12\si_n^{-\frac12}\right\}}y^{1-2s}|U_n|^q\dx\dy\leq C\si_n^{-\left(\frac{N+2-2s}{2}+\frac{(s-1)q}{2^*}\right)},\label{est_q1} 
\end{align}
for every $q\geq1$. 

\noi \textbf{II:} To get an estimate on the remaining region, we apply the Harnack inequality.
Consider two domains $\mathcal{A}_2 \Subset \mathcal{A}_1 \subset\R_+^{N+1}$ such that
\begin{align*}
    \mathcal{A}_2&=\left(B_{(\ov{C}+4)}^+(y_n,0)\setminus\ov{B^+}_{(\ov{C}+1)}(y_n,0)\right)\cap\left\{\frac14\leq y\leq \ov{C}+4\right\},\\
    \mathcal{A}_1&=\left(B_{(\ov{C}+5)}^+(y_n,0)\setminus\ov{B^+}_{\ov{C}}(y_n,0)\right)\cap\left\{\frac15< y< \ov{C}+5\right\}, \text{ where } y_n=\si_n^{\frac12}x_n.
\end{align*} 
Recall that $W_n$ weakly satisfies  
$$\text{div}(y^{1-2s}\nabla W_n)=0\text{ in }\R_+^{N+1}.$$
Observe that the above operator is uniformly elliptic on $\mathcal{A}_1$. Applying Harnack inequality \cite[Corollary 8.21]{GiTr},
$$\sup_{\mathcal{A}_2}W_n\leq C\inf_{\mathcal{A}_2}W_n,$$ where $C$ is independent of $n$.
Using $|U_n(\si_n^{-\frac12}(x,y))|\leq W_n(x,y)$, we get
$$\sup_{\AA_n^2\cap\left\{\frac14\si_n^{-\frac12}\leq y\leq (\ov{C}+4)\si_n^{-\frac12}\right\}}|U_n|\leq C\inf_{\mathcal{A}_2}W_n,$$
where $C$ is independent of $n$.
Finally, 
\begin{align}
    &\int_{\AA_n^2\cap\left\{\frac14\si_n^{-\frac12}\leq y\leq (\ov{C}+4)\si_n^{-\frac12}\right\}}y^{1-2s}|U_n|^q\dx\dy\leq C\si_n^{-\frac{N+2-2s}{2}} \left(\sup_{\AA_n^2\cap\left\{\frac14\si_n^{-\frac12}\leq y\leq (\ov{C}+4)\si_n^{-\frac12}\right\}}|U_n|^q \right) \nonumber\\
    &\leq C\inf_{\mathcal{A}_2}W_n^q\leq C\int_{\mathcal{A}_2 \cap\left\{\frac14\leq y\leq\frac12\right\}}y^{1-2s}W_n^q\dx\dy\leq C\si_n^{\frac{(1-s)q}{2^*}}.\label{est_q2}
\end{align}
In the last inequality, we have used \eqref{est_q0}. 

Therefore, combining \eqref{est_q1} and \eqref{est_q2}, we obtain
$$\int_{\AA_n^2}y^{1-2s}|U_n|^q\dx\dy\leq C\si_n^{-\left(\frac{N+2-2s}{2}+\frac{(s-1)q}{2^*}\right)},$$
for every $q\geq1$. This completes the proof. 
\end{proof}

Finally, we obtain the following gradient estimate of $u_n$ and $U_n$.

\begin{proposition}\label{gradient_prop}
It holds that
$$\int_{\pa\AA_n^3\cap\{y=0\}}|\nabla u_n|^2\dx+\int_{\AA_n^3}y^{1-2s}|\nabla U_n|^2\dx\dy\leq C\si_n^{1-\frac N2},$$
    where $C$ is independent of $n$.
\end{proposition}
\begin{proof}
    Consider a cut-of{}f function $\eta\in \mathcal{C}_c^{\infty}(\AA_n^2)$ such that $\eta=1$ in $\AA_n^3$ and $|\nabla \eta|\leq C\si_n^{\frac12}$. Using the test function $\eta^2U_n$, we see that
\begin{align*}
    &\int_{\Om}\nabla u_n\cdot\nabla_x(\eta^2(x,0)u_n)\dx+\int_{\R_+^{N+1}}y^{1-2s}\nabla U_n\cdot\nabla(\eta^2U_n)\dx\dy\\
    &\qquad=\int_{\Om}(|u_n|^{p_n-2}+\la |u_n|^{p-2})\eta^2(x,0)u_n^2\dx\\
    &\qquad\leq C\int_{\Om}\left(|u_n|^{2^*-2}+1\right)\eta^2(x,0)u_n^2\dx\\
    &\qquad\leq C\int_{\pa\AA_n^2\cap\{y=0\}}|u_n|^{2^*}\dx+C\int_{\pa\AA_n^2\cap\{y=0\}}u_n^2\dx.
\end{align*}
For the gradient terms, using Young's inequality with $\eps=\frac14$, we see that 
\begin{align*}
    \int_{\Om}\nabla u_n\cdot\nabla_x(\eta^2(x,0)u_n)\dx&=\int_{\Om}\eta^2(x,0)|\nabla u_n|^2\dx+2\int_{\Om}u_n\eta(x,0)\nabla u_n\cdot\nabla \eta\dx\\
    &\geq \frac12\int_{\Om}\eta^2(x,0)|\nabla u_n|^2\dx-C\int_{\Om}u_n^2|\nabla \eta(x,0)|^2\dx,
\end{align*}
and similarly,
$$\int_{\R_+^{N+1}}y^{1-2s}\nabla U_n\cdot\nabla(\eta^2U_n)\dx\dy\geq\frac12\int_{\R_+^{N+1}}y^{1-2s}\eta^2|\nabla U_n|^2\dx\dy-C\int_{\R_+^{N+1}}y^{1-2s}U_n^2|\nabla \eta|^2\dx\dy.$$
Combining the above inequalities, we obtain
\begin{align*}
    \int_{\pa\AA_n^3\cap\{y=0\}}|\nabla u_n|^2\dx+\int_{\AA_n^3}y^{1-2s}|\nabla U_n|^2\dx\dy&\leq C\si_n\left(\int_{\pa\AA_n^2\cap\{y=0\}}u_n^2\dx+\int_{\AA_n^2}y^{1-2s}|U_n|^2\dx\dy\right)\\
    &+ C\int_{\pa\AA_n^2\cap\{y=0\}}|u_n|^{2^*}\dx+C\int_{\pa\AA_n^2\cap\{y=0\}}u_n^2\dx.
\end{align*}
Estimating the RHS of the above inequality using Proposition~\ref{integ_prop} yields us
\begin{align*}
    \int_{\pa\AA_n^3\cap\{y=0\}}|\nabla u_n|^2\dx+\int_{\AA_n^3}y^{1-2s}|\nabla U_n|^2\dx\dy&\leq C\si_n^{1-\frac{N}{2}}+C\si_n^{1-\left(\frac{N+2-2s}{2}+\frac{2(s-1)}{2^*}\right)}+C\si_n^{-\frac{N}{2}}\\
    &\leq C\si_n^{1-\frac{N}{2}}+C\si_n^{-\frac{N^2-2N+4-4s}{2N}}\leq C\si_n^{1-\frac{N}{2}},
\end{align*}
as required.
\end{proof}
\subsection{Proof of compactness results}\label{subsec_proof_uniform}
This subsection aims to prove Theorem \ref{Uniform_Theorem_CK}.

\noi \textbf{Proof of Theorem \ref{Uniform_Theorem_CK}:}
We start with the local Pohozaev's identity. Consider a cut-of{}f function $\phi_n$ with 
\begin{align*}
   0\leq \phi_n\leq 1, \quad  \phi_n=1 \text{ in } B_{(\ov{C}+2)\si_n^{-\frac12}}^{N+1}(x_n,0), \quad\mbox{and}\quad \text{supp}(\phi_n)\subseteq B_{(\ov{C}+3)\si_n^{-\frac12}}^{N+1}(x_n,0).
\end{align*}
Set $B_n^+:=B_{(\ov{C}+3)\si_n^{-\frac12}}^+(x_n,0)$ and $B_n^\eps:=B_n^+\cap\{y>\eps\}$, $\eps>0$. Let $X_\eps=(x_0,\eps)$ where $x_0$ will be chosen later.

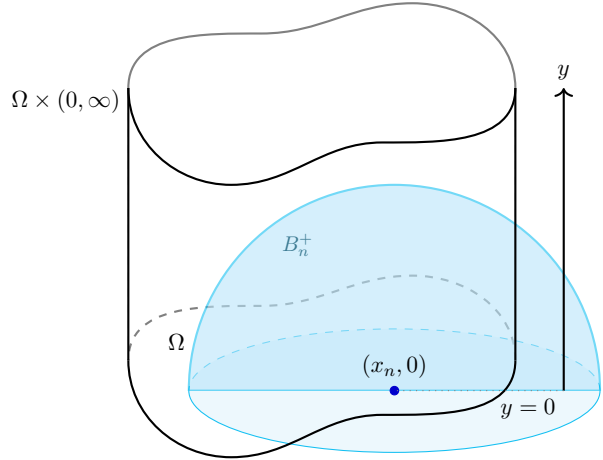
\begin{figure}[ht]
\begin{tikzpicture}[scale=0.8, transform shape]
    % Define Styles
    \tikzset{
        ball/surface/.style={cyan!25, opacity=0.6, draw=cyan!70, thick},
        ball/slice/.style={cyan!15, opacity=0.4}
    }
    
    % Coordinates for the cylinder geometry
    \coordinate (Center) at (0,0);
    \def\rx{3.2} % Max x-width of the arbitrary blob
    \def\h{4.5}  % Height of the cylinder
    % Sequence point x_n near the boundary
    \coordinate (Xn) at (1.2, -0.5); 
    \def\Rball{3.4} % Large radius
    
    % --- 1. Draw Background of the Cylinder (Arbitrary Blob) ---
    % We use out=90 and in=90 at the edges so the vertical walls attach smoothly
    \draw[thick, gray, dashed] (\rx, 0) to[out=90, in=0] (1.5, 1.4) to[out=180, in=0] (-1, 0.9) to[out=180, in=90] (-\rx, 0);
    \draw[dashed, gray!50] (-\rx, 0) -- (-\rx, \h);
    \draw[dashed, gray!50] (\rx, 0) -- (\rx, \h);
    
    % --- 2. Draw the Full Disk Base (UNCLIPPED) ---
    % Shaded base filling
    \fill[ball/slice] (Xn) ellipse ({\Rball} and {\Rball*0.3});
    % Farther (back) edge of the disk - dashed
    \draw[cyan!70, thin, dashed] (Xn) ++(\Rball,0) arc (0:180:{\Rball} and {\Rball*0.3});
    % Near (front) edge of the disk - solid
    \draw[cyan!70, thin] (Xn) ++(-\Rball,0) arc (180:360:{\Rball} and {\Rball*0.3});
    
    % --- 3. Draw the Upper Half Ball Dome B^+ (CLIPPED) ---
    \begin{scope}
    % Clip to the upper half space relative to the base y-level of Xn
    \clip (-4, -0.5) rectangle (5, 5);
    % Shaded hemispherical dome - centered precisely at (Xn)
    \fill[ball/surface] (Xn) ++(\Rball,0) arc (0:180:\Rball) -- cycle;
    \end{scope}
    
    % --- 4. Draw Foreground of the Cylinder (Arbitrary Blob) ---
    \draw[thick] (-\rx, 0) to[out=-90, in=180] (-1.5, -1.6) to[out=0, in=180] (1, -0.9) to[out=0, in=-90] (\rx, 0);
    \draw[thick] (-\rx, 0) -- (-\rx, \h);
    \draw[thick] (\rx, 0) -- (\rx, \h);
    
    % Draw the top boundary of the cylinder at height \h
    \begin{scope}[yshift=\h cm]
        \draw[thick, gray] (\rx, 0) to[out=90, in=0] (1.5, 1.4) to[out=180, in=0] (-1, 0.9) to[out=180, in=90] (-\rx, 0);
        \draw[thick] (-\rx, 0) to[out=-90, in=180] (-1.5, -1.6) to[out=0, in=180] (1, -0.9) to[out=0, in=-90] (\rx, 0);
    \end{scope}
    
    % --- 5. Labels and Points ---
    % Point x_n with label moved to sit just above the point
    \filldraw[blue!80!black] (Xn) circle (2pt) node[above, yshift=2pt, black] {$(x_n, 0)$};
    
    % Label for Omega
    \node at (-2.4, 0.3) {$\Omega$};
    
    % Label for the Cylinder (Shifted to the left side)
    \node[left] at (-\rx, \h*0.95) {$\Omega \times (0, \infty)$};
    
    % Label for the half ball
    \node[cyan!50!black, font=\small] at ($(Xn) + (-1.6, 2.4)$) {$B_n^+$};
        
    % Vertical Extension Axis y (Height adjusted to \h)
    \draw[->, thick] (4, -0.5) -- (4, \h) node[above] {$y$};
    \draw[dotted, gray] (4, -0.5) -- (Xn);
    
    % Label y=0 shifted slightly left so it doesn't cross the lines
    \node[below left] at (4, -0.5) {$y=0$};
\end{tikzpicture}
\caption{A schematic illustration of the cylinder $\Omega\times(0,\infty)$ and the upper half-ball $B_n^+$, where $\Omega$ represents an arbitrary bounded domain.}
\label{fig:halfball}
\end{figure}
Multiplying $((X-X_{\eps})\cdot\nabla U_n)\phi_n$ with $\text{div}(y^{1-2s}\nabla U_n)$ and integrating over $B_n^{\eps}$, we have
\begin{align}\label{30-26-1}
    0&=\int_{B_n^{\eps}}\text{div}(y^{1-2s}\nabla U_n)((X-X_{\eps})\cdot\nabla U_n)\phi_n\dx\dy\nonumber\\
    &=-\int_{B_n^{\eps}}y^{1-2s}\nabla U_n\cdot\nabla(((X-X_{\eps})\cdot\nabla U_n)\phi_n)\dx\dy\nonumber\\&\quad+\int_{\pa B_n^{\eps}}((X-X_{\eps})\cdot\nabla U_n)\phi_n y^{1-2s}(\nabla U_n\cdot\nu_N)\d\HH^N.
\end{align}
Expanding the first term of \eqref{30-26-1}, we obtain
\begin{align*}
    &-\int_{B_n^{\eps}}y^{1-2s}\nabla U_n\cdot\nabla(((X-X_{\eps})\cdot\nabla U_n)\phi_n)\dx\dy\\
    &=-\int_{B_n^{\eps}}y^{1-2s}\nabla U_n\cdot\left[\nabla((X-X_{\eps})\cdot\nabla U_n)\phi_n+\nabla\phi_n((X-X_{\eps})\cdot\nabla U_n)\right]\dx\dy\\
    &=-\int_{B_n^{\eps}}y^{1-2s}|\nabla U_n|^2\phi_n\dx\dy-\int_{B_n^{\eps}}y^{1-2s}\nabla\left(\frac12|\nabla U_n|^2\right)\cdot(X-X_{\eps})\phi_n\dx\dy\\
    &\quad-\int_{B_n^{\eps}}y^{1-2s}((X-X_{\eps})\cdot\nabla U_n)(\nabla U_n\cdot\nabla\phi_n)\dx\dy\\
    &=-\int_{B_n^{\eps}}y^{1-2s}|\nabla U_n|^2\phi_n\dx\dy+\int_{B_n^{\eps}}\frac12|\nabla U_n|^2\nabla\cdot((X-X_{\eps})y^{1-2s}\phi_n)\dx\dy\\
    &\quad-\int_{\pa B_n^{\eps}}y^{1-2s}\frac12|\nabla U_n|^2\phi_n((X-X_{\eps})\cdot\nu_N)\d\HH^N-\int_{B_n^{\eps}}y^{1-2s}((X-X_{\eps})\cdot\nabla U_n)(\nabla U_n\cdot\nabla\phi_n)\dx\dy\\
    &=\frac{N-2s}{2}\int_{B_n^{\eps}}y^{1-2s}|\nabla U_n|^2\phi_n\dx\dy+\frac12\int_{B_n^{\eps}}y^{1-2s}|\nabla U_n|^2(X-X_{\eps})\cdot\nabla\phi_n\dx\dy\\
    &\quad-\int_{B_n^{\eps}}y^{1-2s}((X-X_{\eps})\cdot\nabla U_n)(\nabla U_n\cdot\nabla\phi_n)\dx\dy.
\end{align*}
For the first term in the sixth line of the above identities, we use $X-X_{\eps}\perp\nu_N$ on $\pa B_n^{\eps}\cap\{y={\eps}\}$ and $\phi_n=0$ on $\pa B_n^{\eps}\cap\{y>{\eps}\}$. By dominated convergence theorem and the fact that $\{U_n\}$ is bounded in $\XX_{\Om}^s(\R_+^{N+1})$, for any fixed $n \in \N$, we get
\begin{align}\label{30-26-2}
    &\lim_{\eps\to0^+}-\int_{B_n^{\eps}}y^{1-2s}\nabla U_n\cdot\nabla(((X-X_{\eps})\cdot\nabla U_n)\phi_n)\dx\dy\nonumber\\
    &=\frac{N-2s}{2}\int_{B_n^+}y^{1-2s}|\nabla U_n|^2\phi_n\dx\dy+\frac12\int_{B_n^+}y^{1-2s}|\nabla U_n|^2(X-X_0)\cdot\nabla\phi_n\dx\dy\nonumber\\
    &\quad-\int_{B_n^+}y^{1-2s}((X-X_0)\cdot\nabla U_n)(\nabla U_n\cdot\nabla\phi_n)\dx\dy,
\end{align}
where $X_0=(x_0,0)$.
For brevity, for any fixed $n \in \N$, we set
\begin{align*}
     &B_n^N:=\pa B_n^+\cap\{y=0\}=B_{(\ov{C}+3)\si_n^{-\frac12}}^N(x_n), \text{ and } \\
     &B_{n,\eps}^N:=\left\{x\in \RR:|x-x_n| \le  \left((\ov{C}+3)^2\si_n^{-1}-\eps^2 \right)^{\frac{1}{2}}\right\}, \text{ where } \eps<(\ov{C}+3)\si_n^{-\frac12}.
\end{align*}
Note that $B_{n,\eps}^N \to  B_n^N$, as $\eps\to0$.
Next, for the boundary integral in \eqref{30-26-1}, we split
\begin{align}\label{2-26-1}
    &\int_{\pa B_n^{\eps}}((X-X_{\eps})\cdot\nabla U_n)\phi_n y^{1-2s}(\nabla U_n\cdot\nu_N)\d\HH^N\nonumber\\
    &=\left(\int_{\pa B_n^{+}\cap\{y>\eps\}}+\int_{ B_n^{+}\cap\{y=\eps\}}\right)((X-X_{\eps})\cdot\nabla U_n)\phi_n y^{1-2s}(\nabla U_n\cdot\nu_N)\d\HH^N\nonumber\\
    &=\int_{B_n^{+}\cap\{y=\eps\}}((X-X_{\eps})\cdot\nabla U_n)\phi_n y^{1-2s}(\nabla U_n\cdot\nu_N)\d\HH^N\nonumber\\
    %&=\int_{B_n^{+}\cap\{y=\eps\}}((x-x_0)\cdot\nabla_xU_n)\phi_n\left(-y^{1-2s}\frac{\pa U_n}{\pa y}\right)\d\HH^N\nonumber\\
    &=\int_{B_{n,\eps}^N}((x-x_0)\cdot\nabla_xU_n(x,\eps))\phi_n(x,\eps)\left(-\eps^{1-2s}\frac{\pa U_n}{\pa y}(x,\eps)\right)\dx.
\end{align}
From \cite{anup-mitesh2023, SuVaWeZh1}, $u_n\in \mathcal{C}^2(\Om)\cap\mathcal{C}^{1,\al}(\ov{\Om})$ for every $\al\in(0,\min\{1,2-2s\})$. Using $u_n=0$ in $\RR \setminus \Omega$, we get $u_n\in \mathcal{C}^{0,1}(\RR)$. Thus for any $y>0$, there exists $C>0$ independent of $y$ such that $$\|\nabla_x U(\cdot,y)\|_{L^{\infty}(\RR)}=\|P\ast\nabla u_n\|_{L^{\infty}(\RR)}\leq \|\nabla u_n\|_{L^{\infty}(\RR)}\leq C.$$Moreover, we have the following pointwise convergence$$\nabla_x U(x,\eps)\to\nabla u(x)\text{ for every }x\in B_n^N\setminus\pa\Omega\text{ as }\eps\to0.$$Further, using \eqref{DtoN}, for every $x\in\RR$, $$-\eps^{1-2s}\pa_y U_n(x,\eps)\to\fra u_n(x)\text{ as }\eps\to0.$$
By Lemma \ref{convergence}, for any $\eps>0$ and $x\in B_n^N$, we obtain 
$$\eps^{1-2s}|\pa_y U_n(x,\eps)|\leq C(1+\chi_{\{s<\frac12\}}+\chi_{\{s=\frac12\}}(|\log(\text{dist}(x,\pa\Om))|+\text{dist}(x,\pa\Om)^{-\frac12})+\chi_{\{s>\frac12\}}\text{dist}(x,\pa\Om)^{1-2s}).$$
Finally observe that $$|\log(\text{dist}(x,\pa\Om))|,\text{dist}(x,\pa\Om)^{-\frac12}, \text{dist}(x,\pa\Om)^{1-2s}\in L^1(B_n^N).$$
Hence applying dominated convergence theorem to the last line of \eqref{2-26-1} yields
\begin{align}\label{30-26-3}
&\lim_{\eps\to0^+}\int_{\pa B_n^{\eps}}((X-X_{\eps})\cdot\nabla U_n)\phi_n y^{1-2s}(\nabla U_n\cdot\nu_N)\d\HH^N\nonumber\\
&=\int_{B_n^N}((x-x_0)\cdot\nabla u_n(x))\phi_n(x,0)\fra u_n(x)\dx.
\end{align}

Putting \eqref{30-26-2} and \eqref{30-26-3} back into  \eqref{30-26-1} we obtain
\begin{align}\label{30-26-4}
0=&\frac{N-2s}{2}\int_{B_n^+}y^{1-2s}|\nabla U_n|^2\phi_n\dx\dy+\frac12\int_{B_n^+}y^{1-2s}|\nabla U_n|^2(X-X_0)\cdot\nabla\phi_n\dx\dy\nonumber\\
    &\qquad\qquad-\int_{B_n^+}y^{1-2s}((X-X_0)\cdot\nabla U_n)(\nabla U_n\cdot\nabla\phi_n)\dx\dy\nonumber\\
&\qquad\qquad+\int_{B_n^N}((x-x_0)\cdot\nabla u_n(x))\phi_n(x,0)\fra u_n(x)\dx.
\end{align}

Set $f_n(t):=\la |t|^{p-2}t+|t|^{p_n-2}t$ and $F_n(t):=\int_0^tf_n(s)\ds$. Then, using $u_n=0$ outside $\Om$, we get
\begin{align}\label{30-26-5}
    &\int_{B_n^N}((x-x_0)\cdot\nabla u_n(x))\phi_n(x,0)\fra u_n(x)\dx\nonumber\\
    % &=\lim_{\eps \ra 0} \int_{B_{n,\eps}^N}((x-x_0)\cdot\nabla_xU_n(x,\eps))\phi_n(x,\eps)\left(-\eps^{1-2s}\frac{\pa U_n}{\pa y}(x,\eps)\right)\dx \\
    %&=\left(\int_{\pa B_n^+\cap\{y>0\}}+\int_{\pa B_n^+\cap\{y=0\}}\right)((X-X_0)\cdot\nabla U_n)\phi_n y^{1-2s}(\nabla U_n\cdot\nu_N)\d\HH^N\\
    %&=\int_{\pa B_n^+\cap\{y=0\}}-y^{1-2s}\frac{\pa U_n}{\pa y}((X-X_0)\cdot\nabla U_n)\phi_n \d\HH^N\\
    &=\int_{B_n^N\cap\Om}(f_n(u_n)+\Delta u_n)((x-x_0)\cdot\nabla u_n)\phi_n(x,0) \dx\nonumber\\
    &=-\int_{B_n^N\cap\Om}F_n(u_n)\nabla\cdot((x-x_0)\phi_n)\dx+\int_{\pa (B_n^N\cap\Om)}F_n(u_n)\phi_n(x-x_0)\cdot\nu_{N-1}\d\HH^{N-1}\nonumber\\
    &\quad-\int_{B_n^N\cap\Om}\nabla u_n\cdot\nabla(((x-x_0)\cdot\nabla u_n)\phi_n)\dx+\int_{\pa (B_n^N\cap\Om)}(\nabla u_n\cdot\nu_{N-1})((x-x_0)\cdot\nabla u_n)\phi_n\d\HH^{N-1}\nonumber\\
    &=-N\int_{B_n^N\cap\Om}F_n(u_n)\phi_n(x,0)\dx-\int_{B_n^N\cap\Om}F_n(u_n)((x-x_0)\cdot\nabla_x\phi_n(x,0))\dx\nonumber\\
    &\quad+\frac{N-2}{2}\int_{B_n^N\cap\Om}|\nabla u_n|^2\phi_n(x,0)\dx+\frac12\int_{B_n^N\cap\Om}|\nabla u_n|^2((x-x_0)\cdot\nabla_x\phi_n(x,0))\dx\nonumber\\
    &-\int_{B_n^N\cap\Om}(\nabla u_n\cdot\nabla_x\phi_n(x,0))((x-x_0)\cdot\nabla u_n))\dx-\int_{ B_n^N\cap\pa\Om}\frac12|\nabla u_n|^2\phi_n((x-x_0)\cdot\nu_{N-1})\d\HH^{N-1}\nonumber\\
    &\quad +\int_{ B_n^N\cap\pa\Om}(\nabla u_n\cdot\nu_{N-1})((x-x_0)\cdot\nabla u_n)\phi_n\d\HH^{N-1}.
\end{align}
Inserting \eqref{30-26-5} into \eqref{30-26-4}, we get
\begin{align}
    0&=\frac{N-2s}{2}\int_{B_n^+}y^{1-2s}|\nabla U_n|^2\phi_n\dx\dy+\frac12\int_{B_n^+}y^{1-2s}|\nabla U_n|^2(X-X_0)\cdot\nabla\phi_n\dx\dy\nonumber\\
    &\quad-\int_{B_n^+}y^{1-2s}((X-X_0)\cdot\nabla U_n)(\nabla U_n\cdot\nabla\phi_n)\dx\dy\nonumber\\
    &\quad-N\int_{B_n^N\cap\Om}F_n(u_n)\phi_n(x,0)\dx-\int_{B_n^N\cap\Om}F_n(u_n)((x-x_0)\cdot\nabla_x\phi_n(x,0))\dx\nonumber\\
    &\quad+\frac{N-2}{2}\int_{B_n^N\cap\Om}|\nabla u_n|^2\phi_n(x,0)\dx+\frac12\int_{B_n^N\cap\Om}|\nabla u_n|^2((x-x_0)\cdot\nabla_x\phi_n(x,0))\dx\nonumber\\
    &-\int_{B_n^N\cap\Om}(\nabla u_n\cdot\nabla_x\phi_n(x,0))((x-x_0)\cdot\nabla u_n))\dx -\int_{ B_n^N\cap\pa\Om}\frac12|\nabla u_n|^2\phi_n((x-x_0)\cdot\nu_{N-1})\d\HH^{N-1}\nonumber\\
    &\quad  +\int_{ B_n^N\cap\pa\Om}(\nabla u_n\cdot\nu_{N-1})((x-x_0)\cdot\nabla u_n)\phi_n\d\HH^{N-1}.\label{poh_1.1}
\end{align}
On the other hand multiplying $U_n\phi_n$ with $\text{div}(y^{1-2s}\nabla U_n)$ and integrating by parts over $B_n^+$ yields
\begin{align}
    0&=\int_{B_n^+}\text{div}(y^{1-2s}\nabla U_n)U_n\phi_n\dx\dy\nonumber\\
    &=-\int_{B_n^+}y^{1-2s}\nabla U_n\cdot\nabla(U_n\phi_n)\dx\dy+\int_{\pa B_n^+}y^{1-2s}(\nabla U_n\cdot\nu_N)U_n\phi_n\d\HH^{N}\nonumber\\
    &=-\int_{B_n^+}y^{1-2s}|\nabla U_n|^2\phi_n\dx\dy-\int_{B_n^+}y^{1-2s}U_n(\nabla U_n\cdot\nabla\phi_n)\dx\dy\\&\quad+\int_{\pa B_n^{+}\cap\{y=0\}}-y^{1-2s}\frac{\pa U_n}{\pa y}U_n\phi_n\d\HH^{N}\nonumber\\
    &=-\int_{B_n^+}y^{1-2s}|\nabla U_n|^2\phi_n\dx\dy-\int_{B_n^+}y^{1-2s}U_n(\nabla U_n\cdot\nabla\phi_n)\dx\dy\\&\quad+\int_{B_n^{N}\cap\Om}(f_n(u_n)+\Delta u_n)u_n\phi_n(x,0)\dx\nonumber\\
    &=-\int_{B_n^+}y^{1-2s}|\nabla U_n|^2\phi_n\dx\dy-\int_{B_n^+}y^{1-2s}U_n(\nabla U_n\cdot\nabla\phi_n)\dx\dy+\int_{ B_n^{N}\cap\Om}f_n(u_n)u_n\phi_n(x,0)\dx\nonumber\\
    &\quad-\int_{B_n^N\cap\Om}\nabla u_n\cdot\nabla(u_n\phi_n(x,0))\dx+\int_{\pa (B_n^N\cap\Om)}(\nabla u_n\cdot\nu_{N-1})u_n\phi_n(x,0)\d\HH^{N-1}\nonumber\\
    &=-\int_{B_n^+}y^{1-2s}|\nabla U_n|^2\phi_n\dx\dy-\int_{B_n^+}y^{1-2s}U_n(\nabla U_n\cdot\nabla\phi_n)\dx\dy+\int_{ B_n^{N}}f_n(u_n)u_n\phi_n(x,0)\dx\nonumber\\
    &\quad-\int_{B_n^N\cap\Om}|\nabla u_n|^2\phi_n(x,0)\dx-\int_{B_n^N\cap\Om}u_n\nabla u_n\cdot\nabla_x\phi_n(x,0)\dx.\label{poh_1.2}
\end{align}
Adding \eqref{poh_1.1} and $\frac{N-2}{2}$\eqref{poh_1.2}, we obtain
\begin{align*}
    &\int_{B_n^N\cap\Om}\left(NF_n(u_n)-\frac{N-2}2u_nf_n(u_n)\right)\phi_n(x,0)\dx-(1-s)\int_{B_n^+}y^{1-2s}|\nabla U_n|^2\phi_n\dx\dy\\
    &=\frac12\int_{B_n^+}y^{1-2s}|\nabla U_n|^2(X-X_0)\cdot\nabla\phi_n\dx\dy+\frac12\int_{B_n^N\cap\Om}|\nabla u_n|^2(x-x_0)\cdot\nabla_x\phi_n(x,0)\dx\\
    &\quad-\int_{B_n^+}y^{1-2s}((X-X_0)\cdot\nabla U_n)(\nabla U_n\cdot\nabla\phi_n)\dx\dy-\int_{B_n^N\cap\Om}(\nabla u_n\cdot\nabla_x\phi_n(x,0))((x-x_0)\cdot\nabla u_n)\dx\\
    &\quad-\int_{B_n^N\cap\Om}F_n(u_n)((x-x_0)\cdot\nabla_x\phi_n(x,0))\dx\\
    &\quad-\int_{ B_n^N\cap\pa\Om}\frac12|\nabla u_n|^2\phi_n((x-x_0)\cdot\nu_{N-1})\d\HH^{N-1}+\int_{ B_n^N\cap\pa\Om}(\nabla u_n\cdot\nu_{N-1})((x-x_0)\cdot\nabla u_n)\phi_n\d\HH^{N-1}.\\
    &\quad-\frac{N}{2^*}\left(\int_{B_n^+}y^{1-2s}U_n(\nabla U_n\cdot\nabla\phi_n)\dx\dy+\int_{B_n^N\cap\Om}u_n(\nabla u_n\cdot\nabla_x\phi_n(x,0))\dx\right).
\end{align*}
For large $n$, we choose $X_0=(x_0,0)\in\Om^c$ such that $|x-x_0|\leq C\si_n^{-\frac12}$ in $B_n^N$ and $(x-x_0)\cdot\nu_{N-1}\leq0$ on $B_n^N\cap\pa\Om$. Combining this with the fact that $u_n=0$ on $\pa\Om$ yields
\begin{align*}
    &-\int_{ B_n^N\cap\pa\Om}\frac12|\nabla u_n|^2\phi_n((x-x_0)\cdot\nu_{N-1})\d\HH^{N-1}+\int_{ B_n^N\cap\pa\Om}(\nabla u_n\cdot\nu_{N-1})((x-x_0)\cdot\nabla u_n)\phi_n\d\HH^{N-1}\\
    &=\frac12\int_{ B_n^N\cap\pa\Om}|\nabla u_n|^2\phi_n((x-x_0)\cdot\nu_{N-1})\d\HH^{N-1}\leq0.
\end{align*}
Furthermore, using $(\frac N{p_n}-\frac N{2^*})\int_{B_n^N}|u_n|^{p_n}\dx\geq0$,
\begin{align*}
    &\left(\frac Np-\frac N{2^*}\right)\la\int_{B_n^N \cap \Omega}|u_n|^p\phi_n(x,0)\dx\\
    &\leq\frac12\int_{B_n^+}y^{1-2s}|\nabla U_n|^2(X-X_0)\cdot\nabla\phi_n\dx\dy+\frac12\int_{B_n^N\cap\Om}|\nabla u_n|^2(x-x_0)\cdot\nabla_x\phi_n(x,0)\dx\\
    &\quad-\int_{B_n^+}y^{1-2s}((X-X_0)\cdot\nabla U_n)(\nabla U_n\cdot\nabla\phi_n)\dx\dy-\int_{B_n^N\cap\Om}(\nabla u_n\cdot\nabla_x\phi_n(x,0))((x-x_0)\cdot\nabla u_n)\dx\\
    &\quad-\int_{B_n^N\cap\Om}F_n(u_n)((x-x_0)\cdot\nabla_x\phi_n(x,0))\dx+(1-s)\int_{B_n^+}y^{1-2s}|\nabla U_n|^2\phi_n\dx\dy\\
    &\quad-\frac{N}{2^*}\left(\int_{B_n^+}y^{1-2s}U_n(\nabla U_n\cdot\nabla\phi_n)\dx\dy+\int_{B_n^N\cap\Om}u_n(\nabla u_n\cdot\nabla_x\phi_n(x,0))\dx\right).
\end{align*}
We set $$\BB_n:=\left(B^+_{(\ov{C}+3)\si_n^{-\frac{1}{2}}}(x_n,0)\setminus \ov{B^+}_{(\ov{C}+2)\si_n^{-\frac{1}{2}}}(x_n,0)\right)\cap\{y>0\}.$$
Now we use that $\text{supp}(\nabla\phi_n)\cap\R_+^{N+1}\subset\BB_n$,  $|\nabla\phi_n|\leq C\si_n^{\frac12}$. Since, $|x-x_0|\leq C\si_n^{-\frac12}$ in $B_n^N$, it holds $|X-X_0|\leq C\si_n^{-\frac12}$ in $B_n^+$. Hence, using Propositions \ref{integ_prop} and \ref{gradient_prop}, we obtain
\begin{align*}
    &\left(\frac Np-\frac N{2^*}\right)\la\int_{B_n^N \cap \Omega}|u_n|^p\phi_n(x,0)\dx\\
    &\leq C\int_{B_n^+}y^{1-2s}|\nabla U_n|^2\dx\dy+\int_{\pa\AA_n^3\cap\{y=0\}}(|u_n|^p+|u_n|^{p_n}+|\nabla u_n|^2)\dx\\
    &\quad+C\int_{\BB_n}y^{1-2s}(\si_n|U_n|^2+|\nabla U_n|^2)\dx\dy+\int_{\pa\AA_n^3\cap\{y=0\}}(\si_n|u_n|^2+|\nabla u_n|^2)\dx\\
    &\leq C\int_{B_n^+}y^{1-2s}|\nabla U_n|^2\dx\dy+C\si_n^{1-\frac N2}+C\si_n^{-\frac{N}{2}}\\
    &\leq C\int_{B_n^+}y^{1-2s}|\nabla U_n|^2\dx\dy+C\si_n^{1-\frac N2}.
\end{align*}
For the first term in the third line of the above inequalities, we have used  H\"{o}lder's inequality and the inequality (see \cite[Lemma 2.1]{TaXi})
$$\int_{\BB_n}y^{1-2s}|U|^{2}\dx\dy\leq C\si_n^{-1}\int_{\BB_n}y^{1-2s}|\nabla U|^{2}\dx\dy.$$
Using the (PS) decomposition, $p>\frac{N}{N-2}$ and arguing as in \cite{CaPeYa}, 
\begin{align*}
    \la\int_{B_n^N}u_n^p\phi_n(x,0)\dx\ge\la \int_{B_{\si_n^{-1}}(x_n)\cap\Om}u_n^p\dx\geq C\si_n^{\frac {N-2}2p-N}.
\end{align*}
Using Remark \ref{error_decay}, we see that 
\begin{align*}
    \int_{B_n^+}y^{1-2s}|\nabla U_n|^2\phi_n\dx\dy&\leq \int_{B_n^+}y^{1-2s}|\nabla U_n|^2\dx\dy\\
    &=\int_{B_n^+}y^{1-2s}|\nabla U_\infty|^2\dx\dy+O\left(\sum_{i=1}^k(\si_{n}^i)^{2s-2}\right)\\
&\leq C\si_n^{-\frac{N+2-2s}{2}}+C\si_n^{2s-2}\leq C\si_n^{2s-2}.
\end{align*}
\noi\textbf{Case-1:} Suppose $N>6-4s$, then combining all the estimates we get
$$\si_n^{\frac{N-2}{2}p-N}\leq C\si_n^{2s-2}+C\si_n^{1-\frac{N}{2}}\leq C\si_n^{2s-2}.$$
This implies that $$p\leq\frac{2(N-2+2s)}{N-2},$$
which is a contradiction to the given hypothesis on $p$.\\
\noi\textbf{Case-2:} Now suppose $N\leq 6-4s$, then we get
$$\si_n^{\frac{N-2}{2}p-N}\leq C\si_n^{2s-2}+C\si_n^{1-\frac{N}{2}}\leq C\si_n^{1-\frac{N}{2}}.$$
This implies that 
$$p\leq \frac{N+2}{N-2},$$
which contradicts the given hypothesis on $p$.

Therefore, the (PS) decomposition (Proposition \ref{PS}) holds without any $\Psi_n^i$ and  as a consequence, $U_n$ strongly converges in $\XX_{\Om}^s(\R_+^{N+1})$. This completes the proof.  \qed
\section{Infinitely many sign-changing solutions}\label{infinite}
In this section, as an application of Theorem \ref{Uniform_Theorem} and following the strategy developed in \cite{ScZo}, we prove the existence of infinitely many sign-changing solutions to
\begin{equation*}
\begin{cases}
            \del u+\fra u=\lambda |u|^{p-2}u+|u|^{2^\ast-2}u&\text{in }\Om,\\
            u=0&\text{in }\RR \setminus \Om,
        \end{cases}
\end{equation*}
where $N,s,p$ be as in Theorem~\ref{Existence_Theorem}. 

Let $0<\la_1<\la_2\leq\la_3\leq\cdots$ be the eigenvalues of $(-\Delta+\fra,\Om)$ and $\phi_k$ be the eigenfunction corresponding to $\la_k$. The set $\{\phi_k\}$ forms an orthonormal basis of $L^2(\Om)$ and (after a rescaling) also an orthonormal basis of $X_0(\Om)$. Define $E_k=\text{span}\{\phi_1,\ldots,\phi_k\}$. These properties follow from standard variational methods and the theory of compact self-adjoint operators. Thus $$X_0(\Om)=\ov{\cup_{k=1}^\infty E_k}^{\rho(\cdot)},\quad \text{dim}(E_k)=k,\quad E_k\subset E_{k+1}.$$
We choose a sequence $p_n\in(p,2^*)$ such that $p_n\to 2^*$. As $p_n\in(p,2^*)$, there exists $C_0>0$ independent of $n$ such that
\begin{equation}\label{A_0}
    \|\cdot\|_{p_n}\leq C_0\rho(\cdot) \text{ in }X_0(\Om),\quad X_0(\Om)\Subset L^{p_n}(\Om).\tag{$A_0$}
\end{equation}
We consider the functional $$I_{n,\la}(u):=\frac12\rho(u)^2-\frac{\la}{2}\|u\|_p^p-\frac1{p_n}\|u\|_{p_n}^{p_n}, \; \forall \, u \in X_0.$$
Notice that for any $u,v \in X_0(\Om)$,
\begin{align*}
    I_{n,\la}'(u)(v)=\int_{\Om}\nabla u\cdot\nabla v\dx+\A(u,v)-\la\int_{\Om}|u|^{p-2}uv\dx-\int_{\Om}|u|^{p_n-2}uv\dx.
\end{align*}
By Riesz representation theorem, there exists $L(u), G(u)\in X_0(\Om)$ such that $$\langle L(u),v\rangle=\la\int_{\Om}|u|^{p-2}uv\dx,\quad \langle G(u),v\rangle=\int_{\Om}|u|^{p_n-2}uv\dx.$$
Hence $\nabla I_{n,\la}=u- K_{n,\la}(u)$ where $K_{n,\la}(u)=L(u)+G(u)$. 
Observe that both $L$, $G$ are continuous and as a result $K_{n,\la}:X_0(\Om)\to X_0(\Om)$ is continuous. Let $u_n\in X_0(\Om)$ be a critical point of $I_{n,\la}$ i.e. $u_n$ weakly solves $$\del u+\fra u=\la |u|^{p-2}u+|u|^{p_n-2}u \text{ in }\Om,\quad u=0\text{ in }\RR \setminus \Omega.$$ 
By \cite{SuVaWeZh1,anup-mitesh2023}, $u_n\in \mathcal{C}^{1,\al}(\ov{\Om})$ for some $\al \in (0,1)$.
%By \cite{SuVaWeZh1}, $u_n\in \mathcal{C}^{1,\al}(\ov{\Om})$ for any $\al\in(0,\min\{1,2-2s\})$. 
Now consider the map $I_{n,\la}''(u_n):X_0(\Om)\to X_0(\Om)^*$, defined for any $v,w\in X_0(\Om)$ by
$$I_{n,\la}''(u_n)[v,w]=\int_{\Om}\nabla v\cdot\nabla w\dx+\A(v,w)-\la(p-1)\int_{\Om}|u_n|^{p-2}vw\dx-(p_n-1)\int_{\Om}|u_n|^{p_n-2}vw\dx.$$
Observe that if $v\in \text{ker}(I_{n,\la}''(u_n))$ then $v$ weakly solves $$\del v+\fra v=\la(p-1)|u_n|^{p-2}v+(p_n-1)|u_n|^{p_n-2}v \text{ in }\Om,\quad v=0\text{ in }\RR \setminus \Omega.$$
For $f\in L^2(\Om)$, we define $K_nf\in X_0(\Om)$ to be the unique solution of the problem $$\del u+\fra u=\la(p-1)|u_n|^{p-2}f+(p_n-1)|u_n|^{p_n-2}f \text{ in }\Om,\quad u=0\text{ in }\RR \setminus \Omega.$$ Now, $K_n$ is well defined by the Riesz Representation theorem. Since $u_n\in L^\infty(\Om)$,
\begin{align*}
    \rho(K_nf)^2&\leq \left( \la(p-1)\|u_n\|_{\infty}^{p-2}+(p_n-1)\|u_n\|_{\infty}^{p_n-2}\right)\|f\|_2\|K_nf\|_2\\
    &\leq\left(\la(p-1)\|u_n\|_{\infty}^{p-2}+(p_n-1)\|u_n\|_{\infty}^{p_n-2}\right)\|f\|_{2}\rho(K_nf).
\end{align*}
Using the fact that $X_0(\Om)\Subset L^2(\Om)$, each  $K_n$ is a bounded, linear, and compact operator. Thus, by Fredholm alternative, $I_{n,\la}''(u_n)$ is a Fredholm operator.

\begin{remark}
   Let $\KK_n=\{u\in X_0:I_{n,\la}'(u)=0\}$, $\PP=\{u\in X_0: u\geq0\}$. Then $\PP$ is a closed convex (thus weakly closed) positive cone of $X_0(\Om)$. Since eigenfunctions $\phi_k$ are sign-changing for $k\geq2$, $\pm\PP\cap(E_k^\perp\setminus\{0\})=\emptyset$ for $k\geq2$. For $\mu>0$, define $\DD(\mu)=\{u\in X_0:\text{dist}(u,\PP)<\mu\}$. Then $\DD(\mu)$ is an open convex set containing $\PP$ in its interior. Further, $\DD^*=\DD^*(\mu):=\DD(\mu)\cup(-\DD(\mu))$ and $\SS^*=X_0(\Om)\setminus\DD^*$. 
\end{remark}

\begin{lemma}\label{lemma_A_1}
    Let $\la>0$. For any  $\mu_0>0$ small enough,   $K_{n,\la}\left(\pm\DD(\mu_0)\right) \subset \pm\DD(\mu)\subset \pm\DD(\mu_0)$ for some $\mu\in(0,\mu_0)$. Moreover, $\pm\DD(\mu_0)\cap\KK_n \subset \pm\PP$.    
\end{lemma}
%{\cb Where is $K_{n,\la}$ defined?}
\begin{proof}
    Clearly, $\pm\DD(\mu)\subset \pm\DD(\mu_0)$ for every $\mu\in(0,\mu_0)$. Further, notice that if $u\in\PP$, $L(u),G(u)\in\PP$. To see this, observe by the definition of $L(u)$ and the inner product,
    \begin{align*}
        \langle L(u), L(u)^-\rangle&=\int_{\Om}\nabla L(u)\cdot\nabla L(u)^-\dx+\A(L(u),L(u)^-)\\&=-\rho(L(u)^-)^2-\iint_{\R^{2N}}[L(u)^+(x)L(u)^-(y)+L(u)^+(y)L(u)^-(z)]\d\mu\leq0.
    \end{align*}
    On the other hand, since $u\in\PP$, $$\langle L(u), L(u)^-\rangle=\la\int_{\Om} |u|^{p-2}u L(u)^-\dx\geq0.$$
    Hence, $L(u)^-=0$ and $L(u)\in\PP$. Similarly, we can check that $G(u)\in\PP$. Further for every $u\in X_0(\Om)$, \begin{align*}
        &\text{dist}(L(u),\PP)\rho(L(u)^-)\leq \rho(L(u)-L(u)^+)\rho(L(u)^-)=\rho(L(u)^-)^2\leq-\langle L(u),L(u)^-\rangle\\
        &=-\la\int_{\Om}|u|^{p-2}u\,L(u)^-\dx\leq \la\int_{\Om}|u|^{p-2}u^-\,L(u)^-\dx\\&=\la\int_{\Om}(u^-)^{p-1}L(u)^-\dx\leq \la\|u^-\|_{p}^{p-1}\|L(u)^-\|_{p}.
    \end{align*} 
    Observe that $$\|u^-\|_{p}=\min_{v\in\PP}\|u-v\|_{p}.$$ To see this, notice that since $u^+\in\PP$, $\min_{v\in\PP}\|u-v\|_{p}\leq \|u^-\|_{p}$. The reverse inequality follows from \cite[Proposition 1.9]{DoMa}. Thus,
    \begin{align*}
        \text{dist}(L(u),\PP)\rho(L(u)^-)&\leq\la\|u^-\|_{p}^{p-1}\|L(u)^-\|_{p}\leq C\la\left(\min_{v\in\PP}\|u-v\|_{p}\right)^{p-1}\rho(L(u)^-)\\
        &\leq C\la\left(\min_{v\in\PP}\rho(u-v)\right)^{p-1}\rho(L(u)^-)=C\la\text{dist}(u,\PP)^{p-1}\rho(L(u)^-).
    \end{align*}
    Similarly, for any $u\in X_0(\Om)$,
    \begin{align*}
        \text{dist}(G(u),\PP)\rho(G(u)^-)&\leq C\,\text{dist}(u,\PP)^{p_n-1}\rho(G(u)^-).
    \end{align*}
    Since $p,p_n>2$, we choose $\mu_0$ suf{}ficiently small so that for any $\al<\nu<1$ and for all $u\in \DD(\mu_0)$,
    \begin{align*}
        &\text{dist}(L(u),\PP)\leq C\la\text{dist}(u,\PP)^{p-1}\leq \al\,\text{dist}(u,\PP), \text{ and }\\
        &\text{dist}(G(u),\PP)\leq C\,\text{dist}(u,\PP)^{p_n-1}\leq \left(\nu-\al\right)\text{dist}(u,\PP).
    \end{align*}
    Hence, for all $u\in\DD(\mu_0)$,
    $$\text{dist}(K_{n,\la}u,\PP)\leq \text{dist}(L(u),\PP)+\text{dist}(G(u),\PP)\leq\nu\text{dist}(u,\PP).$$
    Thus, $K_{n,\la}(\DD(\mu_0))\subset \DD(\mu)$ for some $\mu<\mu_0$. Let $u\in \DD(\mu_0)\cap\KK_n$. Then,  $u=K_{n,\la}u$ and thus, $$\text{dist}(u,\PP)=\text{dist}(K_{n,\la}u,\PP)\leq\nu\,\text{dist}(u,\PP).$$ Since $\nu<1$, this implies that $u\in\PP$. The rest of the proof is similar.
\end{proof}
\begin{lemma}\label{lemma_A_2}
    Let $\la>0$. For each $k\geq1$, $$\lim_{\rho(u)\to\infty,\, u\in E_k}I_{n,\la}(u)=-\infty.$$
\end{lemma}
\begin{proof}
Recall $$I_{n,\la}(u)=\frac12\rho(u)^2-\frac1{p_n}\|u\|_{p_n}^{p_n}-\frac\la{p}\|u\|_p^p \quad \forall \, u \in X_0.$$
Since $E_k$ is finite dimensional, there exists $C_{k,n}>0$ such that $\rho(u)\leq C_{k,n}\|u\|_{p_n}$ for all $ u \in E_k.$ Ignoring the $L^p$ term, $$I_{n,\la}(u)\leq \frac12\rho(u)^2-\frac{C_{k,n}^{p_n}}{p_n}\rho(u)^{p_n}\quad \forall \, u \in X_0.$$
Since $p_n>2$, the claim follows.
\end{proof}
\begin{lemma}\label{lemma_A_3}
    Let $\la>0$. For any $\al_1,\al_2>0$ there exists $\al_3(\al_1,\al_2)>0$ such that $\rho(u)\leq\al_3$ for all $u\in I_{n,\la}^{\al_1}\cap\{u\in X_0:\|u\|_{p_n}\leq\al_2\}$ where $I_{n,\la}^{\al_1}=\{u\in X_0:I_{n,\la}(u)\leq\al_1\}$.
\end{lemma}
\begin{proof}
    Using the definition of $I_{n,\la}$ and H\"{o}lder's inequality,
    \begin{align*}
        \frac12\rho(u)^2\leq\al_1+\frac 1{p_n}\al_2^{p_n}+\frac1p\al_2^p|\Om|^{\frac{p_n-p}{p_n}}.
    \end{align*}
    Since $p<p_n$, the claim follows.
\end{proof}
We define 
\begin{align*}
    C_{k}^{**}(n,\la):=\sup_{E_{k}}I_{n,\la}.
\end{align*}
By Lemma \ref{lemma_A_2}, $C_{k}^{**}(n,\la)$ is well defined and $C_{k}^{**}(n,\la)<\infty$. 
\begin{lemma}\label{Upper_bound}
    There exists $T_1>0$ independent of $k$ and $n$ such that $$C_{k+1}^{**}(n,\la)\leq T_1\la_{k+1}^{\frac{p}{p-2}}.$$
\end{lemma}
\begin{proof}
    As $\{\phi_k\}$ is an orthonormal  basis of $X_0(\Om)$ (up to scaling) and $L^2(\Om)$, using the definition of  $E_{k+1}$, $$\rho(u)^2\leq \la_{k+1}\|u\|_2^2.$$

    Since $\Om$ is bounded and $2^*>p_n>p>2$, we have for all $u\in E_{k+1}$,
    \begin{align*}
        I_{n,\la}(u)\leq\frac12\rho(u)^2-\frac1{p_n}\|u\|_{p_n}^{p_n}&\leq\frac12\rho(u)^2-C_1\|u\|_{2}^{p_n} \;(\text{using } 2<p_n<2^*)\\
        &\leq \frac12\rho(u)^2-C_2\|u\|_{2}^{p}+C_3 \;(\text{using } C_1x^{p_n}\geq C_2x^{p}-C_3, x\geq0)\\
        &\leq \frac12\rho(u)^2-C_2\la_{k+1}^{-\frac{p}{2}}\rho(u)^{p}+C_3\leq C_4\la_{k+1}^{\frac{p}{p-2}}+C_3
        \leq T_1\la_{k+1}^{\frac{p}{p-2}}.
    \end{align*}
    The constants $C_1,\ldots,C_4,T_1$ are independent of $n$ and $k$.
\end{proof}
\noi \textbf{Proof of Theorem \ref{Existence_Theorem}:}
Observe that $I_{n,\la}$ satisfies the (PS)$_c$ condition for every $c\in\R$. Applying \cite[Theorem 2]{ScZo}, $I_{n,\la}$ has a nontrivial sign-changing critical point $u^*_{n,k}\in \SS^*$ at the level $C^*(n,\la,k)$, where %$C^*(n,\la,k)=\inf_{A\in\LL}\sup_{A\cap\SS^*}I_{n,\la}$ and
$$C^*(n,\la,k)\in[-\La_0,C_{k+1}^{**}(n,\la)]\subset[-\La_0,T_1\la_k^{\frac{p}{p-2}}].$$ Furthermore the augmented Morse index (the number of non-positive eigenvalues of the linearized operator) $m^*(u^*_{n,k})\geq k$ and for fixed $\la,k$, $\{u^*_{n,k}\}$ weakly solves \eqref{sub_pro} for $p_n \in (p,2^*)$.
For fixed $\la,k$, using $C^*(n,\la,k)\leq T_1\la_{k+1}^{\frac{p}{p-2}}$, we see that
\begin{align*}
\left(\frac12-\frac1{p}\right)\rho(u^*_{n,k})^2&\leq \left(\frac12-\frac1{p}\right)\rho(u^*_{n,k})^2+\left(\frac1p-\frac1{p_n}\right)\|u^*_{n,k}\|_{p_n}^{p_n}\\&=
    \frac12\rho(u^*_{n,k})^2-\frac{\la}{p}\|u^*_{n,k}\|_p^p-\frac1{p_n}\|u^*_{n,k}\|_{p_n}^{p_n}\leq T_1\la_{k+1}^{\frac{p}{p-2}}.
\end{align*}
Hence the sequence $\{u^*_{n,k}\}_n$ is uniformly bounded in $X_0(\Om)$. By Theorem \ref{Uniform_Theorem}, up to a subsequence, we get $u^*_{n,k}\to u^*_k$ in $X_0(\Om)$ and $u^*_k$ weakly solves \eqref{main_PDE}. Since no solution of \eqref{main_PDE} has negative energy, $$0\leq C^*(\la,k)=\lim_{n\to\infty} C^*(n,\la,k)\leq T_1\la_{k+1}^{\frac{p}{p-2}}.$$  
Now we show that $u^*_k$ is also sign-changing. Since $u^*_{n,k}$ weakly solves \eqref{sub_pro}, we have
\begin{equation*}
    \begin{aligned}
        &\rho((u^*_{n,k})^{\pm})^2+2\iint_{\R^{2N}}(u^*_{n,k})^+(x)(u^*_{n,k})^-(y)\d\mu\\
&=\la\|(u^*_{n,k})^{\pm}\|_p^p+\|(u^*_{n,k})^{\pm}\|_{p_n}^{p_n}\leq \la C\rho((u^*_{n,k})^\pm)^p+C\rho((u^*_{n,k})^\pm)^{p_n}.
    \end{aligned}
\end{equation*}
This implies 
$$1\leq\la C\rho((u^*_{n,k})^\pm)^{p-2}+C\rho((u^*_{n,k})^\pm)^{p_n-2}.$$
Since $2<p<p_n$, $\rho((u^*_{n,k})^\pm)\geq C>0$, where $C$ is independent of $n$, we take the limit as $n\to\infty$, to get $\rho((u^*_k)^\pm)\geq C>0$. Hence $u^*_k$ is sign-changing. To prove the infinitude of sign-changing solutions, we prove that $\lim_{k\to\infty}C^*(\la,k)=\infty$. Suppose $\lim_{k\to\infty}C^*(\la,k)=c'<\infty$. Then there exists  $n_k\in\N$ such that $\lim_{k\to\infty}C^*(n_k,\la,k)=\lim_{k\to\infty}C^*(\la,k)=c'$. Since $I_{n_k,\la}(u^*_{n_k,k})=C^*(n_k,\la,k)$ and $I_{n_k,\la}'(u^*_{n_k,k})=0$, $\{u^*_{n_k,k}\}$ is uniformly bounded in $X_0(\Om)$ and using Theorem \ref{Uniform_Theorem}, $\{u^*_{n_k,k}\}$ converges strongly to $u_0$ in $X_0(\Om)$ as $k\to\infty$ for some $u_0\in X_0(\Om)$. Set 
\begin{align*}
    V_k&:=\la (p-1)|u^*_{n_k,k}|^{p-2}-(p_{n_k}-1)|u^*_{n_k,k}|^{p_{n_k}-2}, \text{ and } \\
    V_0&:=\la (p-1)|u_0|^{p-2}-(2^*-1)|u_0|^{2^*-2}.
\end{align*}
Consider the linearized operators
\begin{align*}
    S_k\phi :=\left(\del+\fra-V_k\right)\phi=0, \, \quad k\geq0.
\end{align*}
By the Courant-Fischer min-max theorem, we write the eigenvalues as
$$\la_m(S_k):=\min_{\substack{W_m\subset X_0(\Om)\\\text{dim}(W_m)=m}}\,\max_{\substack{\phi\in W_m\\\|\phi\|_2=1}}\left(\rho(\phi)^2-\int_{\Om}V_k\phi^2\dx\right),\,\quad m\geq1, k\geq0.$$
Since $u^*_{n_k,k}\to u_0$ in $X_0(\Om)$,  $\la_m(S_k)\to \la_m(S_0)$ as $k\to\infty$ for every $m$. On the other hand, we have for any $k\geq0$, $\la_1(S_k)\leq\la_2(S_k)\leq\cdots\to\infty$.
This implies that$$m^*(u^*_{n_k,k})\leq m^*(u_0),\;\text{for large }k.$$This contradicts the fact that $m^*(u^*_{n_k,k})\geq k$ for every $k$. Hence, $C^*(\la,k)\to\infty$ as $k\to\infty$.
\qed

\appendix
\section{Norm estimates via Moser iteration}\label{A}
This section accumulates several key estimates needed for the proof of the compactness theorem. 
In the following lemma, we prove some integral estimates for the weak solution to a certain linear mixed local-nonlocal problem with a potential. 
\begin{lemma}\label{Moser_lemma}
    Let $u\in X_0(\Om)$ weakly solve the following problem 
    $$\del u+\fra u=f \text{ in } \Omega, \; u=0\text{ in } \RR \setminus \Omega.$$ Then the following hold:
    \begin{enumerate}
        \item[{\rm(i)}] Let $f \in L^p(\Om)$ for $p\in(1,\frac{N}{2})$. Then there exist $C(N,p)$ such that $$\|u\|_{\frac{Np}{N-2p}}\leq C\|f\|_p.$$ 
        \item[{\rm(ii)}] Let $f=av$, where $a\in L^{\frac N2}(\Om)$ and $v\in X_0(\Om)$. Then for $p>\frac{N}{N-2}$, there exists $C=C(N,p)$ such that $$\|u\|_p\leq C\|a\|_{\frac N2}\|v\|_p.$$
        \item[{\rm(iii)}] Let $f=av$, where $a\in L^{\frac N2}(\Om)$ and $v\in X_0(\Om)$. Then for $\frac{N}{N-2}<p_2<\frac{2N}{N-2}$, there exists $C=C(N,p_2)$ such that $$\|u\|_{p_2}\leq C\|a\|_{r}\|v\|_{2^*},$$ where $\frac1{p_2}=\frac1r+\frac{1}{2^*}-\frac{2}{N}.$
    \end{enumerate}
\end{lemma}
\begin{proof}
(i)  Following the proof of \cite[Proposition 4.1]{BhBiDa}, we obtain for $\be>1$, 
    \begin{align*}
        \|u\|_{2^*\be}^{2\be}\leq \frac{\be}{S_0}\int_{\Om}|u|^{2\be-1}\abs{f} \dx\leq \frac{\be}{S_0}\|u\|_{2^*\be}^{2\be-1}\|f\|_{\frac{2^*\be}{2^*\be-2\be+1}}.
    \end{align*}
    The last inequality follows using the H\"{o}lder's inequality with coef{}ficients $\left(\frac{2^*\be}{2\be-1},\frac{2^*\be}{2^*\be-2\be+1}\right)$. Hence,
    \begin{align}\label{i-1}
        \|u\|_{2^*\be}\leq \frac{\be}{S_0}\|f\|_{\frac{2^*\be}{2^*\be-2\be+1}}
    \end{align}
Now using the fact that 
\begin{align*}
    q=\frac{2^*\be}{2^*\be-2\be+1} \Longleftrightarrow 2^*\be=\frac{Nq}{N-2q},
\end{align*}
\eqref{i-1} concludes (i). 

\noi (ii) For $f=av$ and $q>1$, applying H\"{o}lder's inequality with exponents $(\frac{N}{2q},\frac{N}{N-2q})$, we get
    $$\|u\|_{\frac{Nq}{N-2q}}\leq \frac{\be}{S_0}\|av\|_{q}\leq \frac{\be}{S_0}\|a\|_{\frac{N}{2}}\|v\|_{\frac{Nq}{N-2q}},$$
    which concludes (ii). 
    
\noi (iii) For $p_2\in(\frac{2^*}{2},2^*)$, using (ii) we get $$\|u\|_{p_2}\leq \frac{\be}{S_0}\|av\|_{\frac{Np_2}{N+2p_2}}\leq \frac{\be}{S_0}\|a\|_{r}\|v\|_{2^*},$$ where $\frac1r=\frac{N+2p_2}{Np_2}-\frac1{2^*}=\frac1{p_2}-\frac{1}{2^*}+\frac{2}{N}.$
\end{proof}
In the following lemma, we prove a Moser iteration for the weak solution to the extension problem. 
\begin{lemma}\label{Moser_iteration}
    Let $\theta\in(0,1)$, and let $0\leq W\in \XX_{\Om}^s(\R_+^{N+1})$ weakly solve
    \begin{equation}
    \begin{cases}
        \text{div}(y^{1-2s}\nabla W)=0\text{ in }\R_+^{N+1},\\
        w(x,0)=w(x)\text{ in }\RR,\\
        -\Delta w-\theta\lim_{y\to0^+}y^{1-2s}\frac{\pa W}{\pa y}\leq a(x)w\text{ in }\Om.
    \end{cases}
\end{equation}
Then there exists $\de>0$ such that if $\displaystyle\int_{B_1^N(x_0)}|a|^{\frac{N}{2}}\dx\leq\de,$ the following holds
$$\|w\|_{L^p(B_{1/2}^N(x_0))}+\theta^{\frac1{2^*}}\|W\|_{L^p(B_{1/2}^+(x_0,0)),y^{1-2s})}\leq C\left(\|w\|_{L^1(B_{1}^N(x_0))}+\theta^{\frac1{2^*}}\|W\|_{L^1(B_{1}^+(x_0,0)),y^{1-2s})}\right)$$
for any $p\geq1$ and in any $B_1^N(x_0)\subset\Om$.
\end{lemma}
\begin{proof}
   For $1\geq R>r\geq\frac{1}{2}$ consider a  cut-of{}f function 
   $\xi$ in $\R^{N+1}$ such that $\xi=1$ in $B_r^{N+1}(x_0, 0)$,  $\text{supp}(\xi)\subset B_R^{N+1}(x_0, 0)$ and $|\nabla\xi|\leq \frac{2}{R-r}$.

   Using the test function $\phi=\xi^2W^{2q-1}, q>1$, we get
    \begin{equation*}
        \int_{\Om}\nabla w\cdot\nabla_x\phi(x,0)\dx+\theta\int_{\R_+^{N+1}}y^{1-2s}\nabla W\cdot\nabla\phi\dx\dy\leq\int_{\Om}a(x)w(x)\phi(x,0)\dx.
    \end{equation*}
Observe that
\begin{align*}
    &\int_{\Om}\nabla w\cdot\nabla_x\phi(x,0)\dx\geq \frac{1}{q}\int_{\Om}|\nabla_x(w^q\xi(x,0))|^2\dx-\frac{C}{q(R-r)^2}\int_{B_R^N(x_0)}w^{2q}\dx,\\
    &\theta\int_{\R_+^{N+1}}y^{1-2s}\nabla W\cdot\nabla\phi\dx\dy\geq \frac{\theta}{q}\int_{\R_+^{N+1}}y^{1-2s}|\nabla(W^q\xi)|^2\dx\dy-\frac{C\theta}{q(R-r)^2}\int_{B_R^+(x_0,0)}y^{1-2s}W^{2q}\dx\dy.
\end{align*}
Thus,
\begin{align*}
    &\int_{\Om}|\nabla_x(w^q\xi(x,0))|^2\dx+\theta\int_{\R_+^{N+1}}y^{1-2s}|\nabla(W^q\xi)|^2\dx\dy\\
    &\leq \frac C{(R-r)^2}\left[\int_{B_R^N(x_0)}w^{2q}\dx+\theta\int_{B_R^+(x_0,0)}y^{1-2s}W^{2q}\dx\dy\right]+q\int_{\Om}a(x)w^{2q}(x)\xi^2(x,0)\dx\\
    &\leq \frac C{(R-r)^2}\left[\int_{B_R^N(x_0)}w^{2q}\dx+\theta\int_{B_R^+(x_0,0)}y^{1-2s}W^{2q}\dx\dy\right]+q\|a\|_{\frac{N}{2}}\left(\int_{\Om}(w^{q}(x)\xi(x,0))^{2^*}\dx\right)^{\frac2{2^*}}.
\end{align*}
By the Sobolev inequality and the trace inequality, we get
\begin{align*}
    &\int_{\Om}|\nabla_x(w^q\xi(x,0))|^2\dx+\theta\int_{\R_+^{N+1}}y^{1-2s}|\nabla(W^q\xi)|^2\dx\dy\\
    &\leq \frac C{(R-r)^2}\left[\int_{B_R^N(x_0)}w^{2q}\dx+\theta\int_{B_R^+(x_0,0)}y^{1-2s}W^{2q}\dx\dy\right].
\end{align*}
where $C>0$ is independent of $N$. By \eqref{Tan-Xiong}, for any $1\leq k\leq\frac{N+1}{N}+\de$, using $\frac12\leq r<R\leq1$, we have 
\begin{align*}
    &\left(\int_{B_r^N(x_0)}w^{2kq}\dx+\theta^k\int_{B_r^+(x_0,0)}y^{1-2s}W^{2kq}\dx\dy\right)^{\frac1k}\\
    &\leq \left(\int_{B_R^N(x_0)}(w^q\xi(x,0))^{2k}\dx\right)^{\frac1k}+\theta\left(\int_{B_R^+(x_0,0)}y^{1-2s}(W^q\xi)^{2k}\dx\dy\right)^{\frac1k}\\
    &\leq\frac C{(R-r)^2}\left[\int_{B_R^N(x_0)}w^{2q}\dx+\theta\int_{B_R^+(x_0,0)}y^{1-2s}W^{2q}\dx\dy\right].
\end{align*}
Hence for any $1\leq k\leq\frac{N+1}{N}+\de$ and for any $q>1$, we have
\begin{align}
    &\left(\int_{B_r^N(x_0)}w^{2kq}\dx+\theta^k\int_{B_r^+(x_0,0)}y^{1-2s}W^{2kq}\dx\dy\right)^{\frac1{2kq}}\nonumber\\
    &\leq\left(\frac C{R-r}\right)^{\frac1q}\left(\int_{B_R^N(x_0)}w^{2q}\dx+\theta\int_{B_R^+(x_0,0)}y^{1-2s}W^{2q}\dx\dy\right)^{\frac1{2q}}.\label{iter_1}
\end{align}
Let $k>1$, $\frac12\leq r^*<R^*\leq1$ and define $r_i=r^*+\frac{1}{2^i}(R^*-r^*)$, $i\geq0$. Then $r_{i+1}-r_i=\frac1{2^{i+1}}(R^*-r^*)$. Since \eqref{iter_1} is true for any $\theta\in(0,1)$, iterating \eqref{iter_1} with $R=r_i$, $r=r_{i+1}$ and $q=\frac{k^{i-1}2^*}{2}$, we get 
\begin{align*}
    &\left(\int_{B_{r_{i+1}}^N(x_0)}w^{k^{i}2^*}\dx\right)^{\frac1{k^{i}2^*}}+\left(\theta^{k^i}\int_{B_{r_{i+1}}^+(x_0,0)}y^{1-2s}W^{k^{i}2^*}\dx\dy\right)^{\frac1{k^{i}2^*}}\\
    &\leq C\left(\int_{B_{r_{i+1}}^N(x_0)}w^{k^{i}2^*}\dx+\theta^{k^i}\int_{B_{r_{i+1}}^+(x_0,0)}y^{1-2s}W^{k^{i}2^*}\dx\dy\right)^{\frac1{k^{i}2^*}}\\
    &\leq\left(\frac {C2^{i+1}}{R^*-r^*}\right)^{\frac2{k^{i-1}2^*}}\left(\int_{B_{r_i}^N(x_0)}w^{k^{i-1}2^*}\dx+\theta^{k^{i-1}}\int_{B_{r_i}^+(x_0,0)}y^{1-2s}W^{k^{i-1}2^*}\dx\dy\right)^{\frac1{k^{i-1}2^*}}\\
    &\leq \frac {C}{(R^*-r^*)^{\sum_{j=1}^i\frac2{k^{j-1}2^*}}}\left(\int_{B_{R^*}^N(x_0)}w^{2^*}\dx+\theta\int_{B_{R^*}^+\theta(x_0,0)}y^{1-2s}W^{2^*}\dx\dy\right)^{\frac1{2^*}}.
\end{align*}
Hence for any $p>2^*$, $\frac12\leq r<R\leq 1$ and for some $\si'>0$,
\begin{align*}
    &\left(\int_{B_{r}^N(x_0)}w^{p}\dx\right)^{\frac1{p}}+\theta^{\frac1{2^*}}\left(\int_{B_{r}^+(x_0,0)}y^{1-2s}W^{p}\dx\dy\right)^{\frac1{p}}\\
    &\leq \frac{C}{(R-r)^{\sigma'}}\left(\int_{B_{R}^N(x_0)}w^{2^*}\dx\right)^{\frac1{2^*}}+\theta^{\frac1{2^*}}\frac{C}{(R-r)^{\sigma'}}\left(\int_{B_{R}^+(x_0,0)}y^{1-2s}W^{2^*}\dx\dy\right)^{\frac1{2^*}}.
\end{align*}
Let $1<2^*<p$, then by H\"{o}lder's inequality and Young's inequality,
\begin{align*}
    \frac{C}{(R-r)^{\si'}}\|w\|_{L^{2^*}(B_R^N(x_0))}&\leq \frac{C}{(R-r)^{\si'}}\|w\|_{L^1(B_R^N(x_0))}^{\kappa}\|w\|_{L^{p}(B_R^N(x_0))}^{1-\kappa}\\
    &\leq \frac12\|w\|_{L^{p}(B_R^N(x_0))}+\frac{C}{(R-r)^{\si''}}\|w\|_{L^1(B_R^N(x_0))},
\end{align*}
similarly,
$$\frac{C}{(R-r)^{\si'}}\|W\|_{L^{2^*}(B_R^{N+1}(x_0,0), y^{1-2s})}\leq \frac12\|W\|_{L^{p}(B_R^{N+1}(x_0,0), y^{1-2s})}+\frac{C}{(R-r)^{\si''}}\|W\|_{L^{1}(B_R^{N+1}(x_0,0), y^{1-2s})},$$
where $\kappa\in(0,1)$ and $\si''>0$ are constants. Thus we obtain for any $p>2^*$ and $\frac12\leq r<R\leq1$,
\begin{align*}
    &\|w\|_{L^{p}(B_r^N(x_0))}+\theta^{\frac1{2^*}}\|W\|_{L^{p}(B_r^{N+1}(x_0,0), y^{1-2s})}\\
    &\leq \frac12\left(\|w\|_{L^{p}(B_R^N(x_0))}+\theta^{\frac1{2^*}}\|W\|_{L^{p}(B_R^{N+1}(x_0,0), y^{1-2s})}\right)\\&\quad+\frac{C}{(R-r)^{\si''}}\left(\|w\|_{L^{1}(B_R^N(x_0))}+\theta^{\frac1{2^*}}\|W\|_{L^{1}(B_R^{N+1}(x_0,0), y^{1-2s})}\right)
\end{align*}
Using an iteration argument (see \cite[Lemma 1.1]{GiGi}), we obtain for any $p>2^*$ and for some $\si>0$,
\begin{align*}
    &\left(\int_{B_{r}^N(x_0)}w^{p}\dx\right)^{\frac1{p}}+\theta^{\frac1{2^*}}\left(\int_{B_{r}^+(x_0,0)}y^{1-2s}W^{p}\dx\dy\right)^{\frac1{p}}\\
    &\leq \frac{C}{(R-r)^{\sigma}}\left(\int_{B_{R}^N(x_0)}w\dx+\theta^{\frac1{2^*}}\int_{B_{R}^+(x_0,0)}y^{1-2s}W\dx\dy\right),
\end{align*}
Finally, using interpolation, the required estimate holds for every $p\geq1$.
\end{proof}

\vspace{5mm}

\noi \textbf{Acknowledgement:}
We sincerely thank Prof. Debabrata Karmakar for many insightful discussions and for proposing the use of the Caf{}farelli-Silvestre extension in estimating certain integrals. The research of Mousomi Bhakta is supported by the Swarnajaynti Fellowship (SB/SJF/2021-22/09) and the ANRF-MATRICS grant (ANRF/ARGM/2025/000126/MTRDST). Nirjan Biswas acknowledges the support of the National Board for Higher Mathematics Postdoctoral Fellowship (0204/16(9)/2024/RD-II/6761). The research of Paramananda Das is supported by the National Board for Higher Mathematics PhD Fellowship (0203/5(38)/ 2024-R\&D-II/11224). Paramananda Das also expresses his gratitude to TIFR-CAM for its hospitality during his visit in February 2026. 

% \vspace{0.1cm}

% \noindent \textbf{Competing interests:} The authors have no competing interests to declare that are relevant to the content of this article.

% \vspace{0.1cm} 

% \noindent\textbf{Data availability statement:} Data sharing is not applicable to this article as no data sets were generated or analysed during the current study.

% \vspace{5mm}

\bibliographystyle{abbrvnat}

\end{document}